\documentclass[11pt,reqno,a4paper]{amsart}
\usepackage{color,graphicx,enumerate,amssymb}
\usepackage{hyperref}
\usepackage{graphicx}
\usepackage[export]{adjustbox}
\usepackage{eucal}
\usepackage[utf8]{inputenc}

\usepackage{mathtools}
\usepackage{comment}

\newcommand{\bR}{{\mathbb R}}

\newcommand{\cS}{{\mathcal S}}
\newcommand{\cK}{{\mathcal K}}

\newcommand{\cA}{{\mathcal A}}

\newcommand{\cG}{{\mathcal G}}
\newcommand{\cP}{{\mathcal P}}

\newcommand{\cL}{{\mathcal L}}

\newcommand{\cX}{{\mathcal X}}

\newcommand{\uu}{\mathbf{ u}}
\newcommand{\vv}{\mathbf{ v}}
\newcommand{\ww}{\mathbf{ w}}

\newcommand{\veps}{\varepsilon}

\newcommand{\rdiv}{{\rm div\,}}

\def\XXint#1#2#3{{\setbox0=\hbox{$#1{#2#3}{\int}$ }
\vcenter{\hbox{$#2#3$ }}\kern-.6\wd0}}

\newcommand{\ra}{\rightarrow}

\newcommand\norm[1]{\Arrowvert {#1}\Arrowvert}

\newtheorem{theorem}{Theorem}[section]

\newtheorem{lemma}[theorem]{Lemma}

\newtheorem{proposition}[theorem]{Proposition}
\newtheorem{remark}[theorem]{Remark}

\title {Graph based semi-supervised learning  using spatial Segregation theory}
\author{Farid Bozorgnia, Morteza Fotouhi,   Avetik Arakelyan and Abderrahim Elmoataz}

\address{CAMGSD, Department of Mathematics, Instituto Superior T\'{e}cnico, Lisbon, Portugal}
\email{faridb.bozorgnia@tecnico.ulisboa.pt}

 \address{Department of Mathematical Sciences, Sharif University of Technology, Tehran, Iran.}
\email{fotouhi@sharif.edu}

  \address{Institute of Mathematics, NAS of Armenia, 0019 Yerevan, Armenia }
\email{arakelyanavetik@gmail.com}

\address{University of Caen Normandy,   GREYC UMR CNRS 6072, France}
\email{abderrahim.elmoataz@unicaen.fr}

\thanks{F. Bozorgnia was supported by the FCT  fellowship SFRH/BPD/33962/2009 and by Marie Skłodowska-Curie grant agreement No. 777826 (NoMADS)}

\subjclass[2000]{}


\begin{document}
\begin{abstract}
In this work we address graph based semi-supervised learning using the theory of the spatial segregation of competitive systems. First, we define a discrete counterpart over connected graphs by using direct analogue of the corresponding competitive system. This model turns out doesn't have a unique solution as we expected. Nevertheless, we suggest gradient projected and regularization methods to reach some of the solutions. Then we focus on a slightly different model motivated from the recent numerical results on the spatial segregation of reaction-diffusion systems. In this case we show that the model has a unique solution and propose a novel classification algorithm based on it. Finally,  we present numerical experiments showing the method is efficient  and comparable to other semi-supervised learning algorithms at high and low label rates.

\end{abstract}
\bigskip

\maketitle

\section{Introduction}

In this paper we consider a semi-supervised learning approach  which deals with the classification of a large number of unlabeled data when very few labels just are  available. In some  applications  such as  medical images, we have few
training examples which are   labeled.  The aim is  to find  efficient  algorithms with  good performance
  with these few  labeled examples. In this situation, geometric or topological properties of the unlabeled data has been used to propose and to   improve several algorithms.

A common way to use the unlabeled data in semi-supervised learning is to build a graph over the data e.g., in image classification.
To start, we  requires to construct  an adjacency matrix, or weight matrix  $W$, for the data set, which encodes the similarities between pairs of date nodes.
If our data set consists of $n$ points  $\cX=\{x_1,  x_2, \cdots ,x_n \}  \subset \mathbb{R}^d$, then the  weight matrix $W$   is an  $n \times n $  symmetric matrix, where the element  $w_{ij}$  represents the similarity between two data points  $x_i$  and $x_j$.
The similarity is always nonnegative and should be large when  $x_i$  and $x_j$ are close together spatially,
and small (or zero), when $x_i$  and $x_j$  are far apart.
As a  choice,  the weight matrix can be the  Gaussian weights
\[
 w_{ij}  =\textrm{exp}(-\frac{| x_{i}-x_{j}|}{2 \sigma^2}),
   \]
where $|\cdot|$ is the Euclidean norm and  $\sigma $ is a free parameter that controls the scale at which points are
connected by strong similarities  $w_{ij}$  in the weight matrix. In fact, the
weight matrix $W$  provides  the data set with a graph structure, where each pair of
points $(x_i,  x_j)$ is connected by an edge with edge weight $w_{ij}.$ Other choices
of weight matrix are possible,   such as the $k$-nearest neighbor graph, see \cite{Lux}.

Recently, many  works   aim to transpose and adapt
Partial Differential Equations (PDEs) on graphs. This  reformulation of continuous problems onto a graph
is such  that the solution behaves analogously to the continuous formulation see \cite{A6, DEL}.

In graph-based semi-supervised learning, we are given a few labeled data on the graph and  we aim  to extend these  labels from a given  set  to the rest of the nodes in graph in  a decisive manner.
To model propagating labels in  semi-supervised  learning,  it is   assumed   that  the learned labels   vary   smoothly  and not change fast  within high density regions of the graph (smoothness  assumptions).
Based on this assumption different approaches have  been proposed.
One of the pioneer methods  is {\it Laplace learning}, \cite{A19}.
Later it has been observed that the Laplace learning can give poor results in classification \cite{A18}.
The results are often poor because the solutions have localized spikes near the labeled points, while being almost constant far from them.
To overcome this problem several versions of Laplace learning algorithm  have been  proposed, for instance Laplacian regularization,  \cite{A1}, weighted Laplacian, \cite{Shi-Weighted, A14} and $p$-Laplace learning, \cite {A10, slepcev2019analysis}.
Also, the limiting  case in $p$-Laplacian when $p$ tends to infinity  is so  called Lipschitz learning is studied in \cite{K17} and similar to continuum PDEs  is related to finding  the absolutely minimal lipschitz extension of the training data.
Recently,   in \cite{A8}  another approach to increase   accuracy of Laplace learning  is   given  and called {\it Poisson learning}.

  To explain these methods,  let as before  $\cX={\{ x_1, \cdots, x_n}\}$ denotes the data points or equivalently vertices in a graph.
  We assume there is a subset of the nodes $ \Gamma={\{ x_1, \cdots, x_m}\}  \subset  \cX $  that their labels are given with  a label function $g:\Gamma\rightarrow\bR^k$.
It is further assumed that $y_i=g(x_i)\in \{e_1, \cdots , e_k\}$ where $e_i$ is the standard basis in $\mathbb{R}^k$ and represents the $i\textsuperscript{th}$ class.
In graph-based semi-supervised learning, we aim to  extend labels  to the rest of the vertices $\{x_{m+1}, \cdots ,x_n\}$.

 In Laplace learning algorithm the labels are extended by finding the minimizer   $ \uu: \cX\rightarrow \mathbb{R}^k$ for the following problem
 \begin{equation}\label{objective}
 \begin{cases}
&\min J_n(\uu):=\frac{1}{2}\sum\limits_{i=1}^{n}\sum\limits_{j=1}^{n} w_{ij} |\uu(x_{i}) - \uu(x_{j})|^{2}\\
&\text{subject to } \uu(x_i) =y_i,\quad  \text{ for } i=1, 2, \cdots ,m.
\end{cases}
\end{equation}
The minimizer  will be a harmonic function satisfying
\begin{equation*}
\left \{
\begin{array}{ll}
\cL \uu(x)=0,  &  x\in \cX\setminus\Gamma ,\\
  \uu=g,   &   \text{ on }\Gamma,\\
  \end{array}
\right.
\end{equation*}
 where $\cL$ is the unnormalized graph Laplacian given by
  \begin{equation*}
 \mathcal{L}\uu(x_i)=\sum\limits_{j=1}^{n} w_{ij}\, (\uu(x_{i}) - \uu(x_j)).
 \end{equation*}
Let $\uu=(u_1,\cdots,u_k)$ be a solution of \eqref{objective}, the label of node $x_i\in \cX\setminus\Gamma$ is dictated by
 \begin{equation*}
\underset{j\in {\{1,\cdots, k}\}}{ \textrm{arg max}} {u_{j}(x_i)}.
  \end{equation*}

In $p$-Laplacian algorithm, the object function is replaced with
\begin{align*}
\min_{\uu\in\cK}J_p(\uu):=\sum_{i=1}^{n}\sum\limits_{j=1}^{n} w_{ij} |\uu(x_{i}) - \uu(x_{j})|^{p},
\end{align*}
or for weighted Laplacian the following object is considered
\begin{align*}
\min_{\uu\in\cK}J_\gamma(\uu):=\sum_{i=1}^{n}\sum\limits_{j=1}^{n} \gamma(x_i)w_{ij} |\uu(x_{i}) - \uu(x_{j})|^{2},
\end{align*}
where $\gamma(x)\approx \text{dist}(x,\Gamma)^{-\alpha}$ increases the weights of edges adjacent to labels much more than other edges.
Using this method encourages the label functions to be flat (more regular) near labels, thus preventing the appearance of spikes (discontinuous solutions).  (see \cite{Shi-Weighted, A14} for more details).

The authors in \cite{A8}  have proposed a  scheme, called Poisson learning that  replaces the  label values at training points  as sources and sinks, and solves the Poisson equation on the graph as follows:
\[
 \mathcal{L}\uu(x_i) =\sum\limits_{j=1}^{m}  (y_{j}- \overline{y})\delta_{ij},\qquad   i=1,\cdots, n,
\]
with further condition    $\sum_{i=1}^{n} d(x_i) \uu(x_i)  =0$, where
$
\overline{y}=\frac{1}{m}\sum_{i=1}^{m} y_{i}
$
is the average label vector,    $ \delta_{ij}  $    is  Kronecker delta and $d(x_i)=\sum_{j=1}^nw_{ij}$ is the degree of vertex $x_i$.



 A  major topic in this strand  concerns the continuum limits of these PDEs or functional on  graphs, linking between  the discrete and continuum perspectives and  the study of the consistency of the above methods in the large data limit, we  refer the reader to e.g. \cite{A14,Abde1,GGHS}.


Let  $x_1, x_2,\cdots , x_n $  be a sequence of independent and identically distributed random variables on  $\Omega$ with  smooth  distribution   $ \rho(x)$.
 Define the weight matrix
\[
w_{ij}  =\eta_{\varepsilon} (| x_{i}-x_{j}|)  = \frac{1}{ \varepsilon ^d}  \eta (\frac{ | x_{i}-x_{j}|}\varepsilon)
\]
where $\eta$ is  a radial kernel $\eta:  [ 0, \infty )  \rightarrow [ 0, \infty )$   which is nonincreasing, continuous at $0$ and given by
\begin{equation*}
	\eta(t)= \left \{
	\begin{array}{ll}
	1  &  0 \le t\le 1 \\
	0     & t>2.
	\end{array}
	\right.
	\end{equation*}
In \cite{ACRWJ}  it has been shown that for $u$ sufficiently smooth, with probability one
\[
\frac{1}{\varepsilon^{2} n^2} J_{n}(\uu)\rightarrow   \sum_{i=1}^k\int_\Omega|\nabla u_i(x)|^{2} \rho(x) \, dx:=  J_{\infty}(\uu),
\]
where  $\rho$   is density function of a probability measure  that data points are generated.

In this paper, we propose a novel classification scheme based on the segregation model. Our motivation for the current work is based on
properties of a class of reaction diffusion system with highly comparative rate which yields segregation of different components  which means at each point in the domain different components can not coexist. In this model, we solve problem \eqref{objective} with additional constraint
\begin{equation*}
u_i(x)\cdot u_j(x)=0, \qquad\text{ for all }x\in \cX, \;1\leq i\neq j\leq k.
\end{equation*}

The continuous form corresponding to the segregation model has been studied   extensively, for instance  \cite{CL, F1, Av3}.
We state the results related to limiting configuration of  the following coupled system as parameter $ \varepsilon$ tends to zero.

 \begin{align}\label{f20}
\begin{cases}
\Delta  u_{i}^{\varepsilon}=   \frac{ u_{i}^{\varepsilon}}{\varepsilon}  \sum\limits_{j \neq i}   u_{j}^{\varepsilon} (x)\qquad\qquad & \text{ in  } \Omega,\\
u_{i}^{\varepsilon} \ge 0\;  & \text{ in  } \Omega,\\
u^{\varepsilon}_{i}(x) =\phi_{i}(x)    &   \text{ on} \,  \partial \Omega,\\
  \end{cases}
\end{align}
for  $i=1,\cdots, m.$   The boundary values  satisfy
\[
\phi_{i}(x) \cdot \phi_{j}(x)=0,  \quad i \neq j \textrm{ on the boundary}.
\]
First, for each fixed  $\varepsilon $ the exist   unique positive solution  $ (u_{1}^{\varepsilon},\cdots   ,u_{m}^{\varepsilon})$. Next, to explain the asymptotic behaviour of \ref{f20} by construction barrier functions, one can show that the normal derivative of $ u_{i}^{\varepsilon}$  is bounded independent of  $ \varepsilon$, this consequently proves that the  $H^1$-norm of  $ u_{i}^{\varepsilon}$
is bounded.    Next integrating  the equation  in \ref{f20}  over $\Omega$ indicates
\[
 (u_{i}^{\varepsilon}  \sum\limits_{j \neq i}   u_{j}^{\varepsilon} (x))   \rightarrow 0 \,\,  \text{as $\varepsilon$ tends to zero}.
\]
Let $ (u_1, \cdots  ,u_m)$ be the limiting configuration, then the  results in \cite{CL} shows that $u_i$ are pairwise segregated, i.e., $u_{i}(x)\cdot u_{j}(x)=0,$   harmonic in their supports and satisfy the following differential inequalities

 \begin{itemize}
\item  $ -\Delta u_{i} \ge 0$,\\
\item  $ -\Delta (u_{i}- \sum\limits_{j\neq i}u_j )\le 0$,
 \item  Let $x$  belongs to interface    then
        \[
   \underset{ y\rightarrow x} {\text{ lim}} \ \nabla u_{i}(y)=-
\underset{ y\rightarrow x}{ \text{ lim}} \  \nabla u_{j}(y)\quad \text{Free } \, \text{boundary}\, \text{condition}.
  \]
 \end{itemize}

 In \cite{Conti-Terracini}, it has been shown that the limiting solution of (\ref{f20})   minimizes the following functional

\begin{equation}\label{continuous-model}
\left\{
\begin{split}
&\min J(\uu):=\frac{1}{2}\int_\Omega\sum\limits_{i=1}^k|\nabla u_i|^2\, dx\\
&\text{subject to } u_i=\phi_i,\quad  \text{ on } \partial\Omega, \\[10pt]
&\qquad u_i\geq0, \quad\text{ and }\quad u_i\cdot u_j=0\quad\text{ in }\Omega,
\end{split}\right.
\end{equation}

 In  \cite{F1, BA, Av1, Av2, Av3}, the authors have proposed and analyzed  the following numerical scheme to solving limiting configuration of system  \eqref{f20} and  \eqref{continuous-model}
 \begin{equation*}
u_{i}^{t+1} (x) =\max\Big(\overline{u}^{t}_{i}(x)-\sum\limits_{j\neq i}\overline{u}^{t}_{j}(x),\,  0\Big)
\end{equation*}
where $\bar v(x)$   denotes the average of values of  $v$  of neighbors  of    mesh point $x,$ and $t$ refers to iterations.

\section{Calculus on graphs and setting the  problem }
This section is devoted to review some facts about the calculus on graphs and setting our problem.
Let $\cX=\{x_1,\cdots,x_n\}$ denote the vertices of a  graph with the symmetric edge weight $w_{xy}$ between $x,y\in \cX$.
The degree of a vertex $x$ is given by $d(x)=\sum_{y\in \cX} w_{xy}$.  Let $\ell^{2}(\cX)$ denote the set of functions $u: \cX \rightarrow  \mathbb{R}$ equipped with the inner product
\[
(u,v)=\sum\limits_{x\in \cX}  u(x)v(x),
\]
for functions $ u,v : \cX\rightarrow \mathbb{R}$.
We also define a vector field on the graph to be an antisymmetric function $V:\cX\times \cX\rightarrow \mathbb{R}^2$, i.e. $V(x,y)=-V(y,x)$ and denote the space of all vector fields by $\ell^2(\cX^2)$.
The gradient of a function $ u\in \ell^{2}(\cX) $ is the vector field
\[
\nabla u(x,y)=u(y)-u(x).
\]
For two vector fields $V_1, V_2$ the inner product is
\[
(V_1, V_2)_{\ell^{2}(\cX^2)}= \frac{1}{2} \sum\limits_{x,y\in \cX}w_{xy} V_{1} (x, y) V_{2} (x, y),
\]
so the norm of vector field $V$ is $\|V\|_{\ell^{2}(\cX^2)}= \sqrt{(V, V)_{\ell^{2}(\cX^2)}}$.
The graph divergence of a vector field $V$ is defined by
\[
\rdiv V(x)=\sum_{y\in X}w_{xy}V(x,y),
\]
which satisfies the divergence formula
\[
(\nabla u, V)_{\ell^{2}(\cX^2)}=-(u,\rdiv V).
\]
The unnormalized graph laplacian $\cL $ of a function $u\in \ell^2(\cX)$ is defined as
\[
\cL u(x):=-\rdiv(\nabla u)(x)=\sum_{y\in \cX}w_{xy}(u(x)-u(y)).
\]
The operator $\cL$ satisfies
\begin{equation}\label{green-formula}
(\cL u,v)=(\nabla u,\nabla v)_{\ell^2(\cX^2)}.
\end{equation}
In Appendix, we revisit some important tools for PDE on graphs, such as Poincar\'e inequality and maximum principle.

We consider a subset of the nodes $ \Gamma  \subset  \cX $  as the boundary of the graph and define the admissible set
\[
\cK:=\left\{\uu=(u_1,\cdots, u_k)\in \left(\ell^2(\cX)\right)^k: u_i=\phi_i\text{ on } \Gamma\text{ for }i=1,\dots,k\right\},
\]
where the boundary data $\phi_i$ are known and satisfy the following assumption
\begin{equation}\label{phi-assumption}
 \phi_i\in \{ 0,1\}.
\end{equation}
We are going  to solve the optimization problem
\begin{align}\label{L2-Problem}
&\min_{\uu\in\cK}J(\uu):=\|\nabla\uu\|^2=\sum_{i=1}^k\|\nabla u_i\|_{\ell^2(\cX^2)}^2\notag\\
 & \textrm{subject to: }\\
 & u_i(x)\geq0, \ \forall x\in \cX,  \\
&u_i(x)\cdot u_j(x)=0\ \forall x\in \cX\text{ and }i\neq j.\notag
\end{align}

The following theorem states the existence of solution to problem \eqref{L2-Problem} and describes some properties of the solution.

\begin{theorem} \label{property-minimizers}
Problem \eqref{L2-Problem} has a solution. Moreover, the solution satisfies
\begin{enumerate}[(i)]
\item
$\cL u_i(x)\leq 0, \text{ if } u_i(x)=0,\text { and }x\in \cX.$

\item
$\cL u_i(x)=0, \text{ if } u_i(x)>0,\text { and }x\in \cX\setminus \Gamma.$

\item
For every $x\in \cX\setminus \Gamma$, there is one component $u_i$ such that $u_i(x)>0$.
\end{enumerate}
\end{theorem}
\begin{proof}
Consider a minimizing sequence $\uu^n\in \cK$ for problem \eqref{L2-Problem}. By Poincar\'e inequality, Proposition \ref{prop:Poincare inequaity}, we obtain that
\[
\|u_i^n\|\leq \frac1{\lambda_1}\|\nabla(u_i^n-\phi_i)\|+\|\phi_i\|\leq \frac1{\lambda_1}(\|\nabla u_i^n\|+\|\nabla \phi_i\|)+\|\phi_i\|.
\]
Thus for every $i=1,\cdots, k$ the sequence $\{u^n_i\}$ is bounded. Hence, there exists a subsequence such that for every components $i$
\[
u^{n_j}_i\rightarrow u_i.
\]
It is obvious that $\uu=(u_1,\cdots, u_k)$ satisfies the constraints in \eqref{L2-Problem} and is a minimizer.

$(i)$ To prove this part of theorem, notice that if $u_i(x)=0$ for some $x\in \cX$, then
$$\cL u_i(x)=\sum_{y\in \cX}w_{xy}(u_i(x)-u_i(y))=-\sum_{y\in \cX}w_{xy}u_i(y)\leq 0.$$

$(ii)$ Now consider the case $u_i(x)>0$ for some fixed node $x\in \cX\setminus\Gamma$. Let us define
$$ v_i=u_i+t\delta_x, \quad v_j=u_j \text{ when }j\neq i,$$
where $\delta_x$ is delta function which is $\delta_x(y)=0$ for every $y\neq x$ and $\delta_x(x)=1$. We also consider some values of $t$ such that $v_i(x)\geq0$, ($t$ can be negative).
Obviously, $\vv\in\cK$ and satisfies the constraints in \eqref{L2-Problem}.
Therefore,
\begin{align*}
0\leq &\|\nabla\vv\|^2-\|\nabla\uu\|^2=\|\nabla v_i\|^2-\|\nabla u_i\|^2\\
=&\frac12\sum_{y,z\in \cX}w_{yz}\left((v_i(z)-v_i(y))^2-(u_i(z)-u_i(y))^2\right)\\
=&\sum_{y\in\cX}w_{xy}\left(t^2+2t(u_i(x)-u_i(y))\right)\\
=&t^2d(x)+2t\cL u_i(x).
\end{align*}
Since $u_i(x)>0$, so parameter $t$ can be take some negative values and when $t\rightarrow 0^\pm$ we conclude that $\cL u_i(x)=0$.

$(iii)$ Let $A:=\{x\in  \cX\setminus \Gamma: u_1(x)=\cdots=u_k(x)=0\}$. We claim that $A=\emptyset$.
Otherwise, there is some $x\in A$ such that $w_{xy}\neq0$ and $u_i(y)>0$ for some $i\in\{1,\cdots,k\}$.
Thus
\[
\cL u_i(x)= - \sum_{z\in\cX}w_{xz} u_i(z) \leq -w_{xy}u_i(y)<0.
\]
Now choose the competitor $\vv$ with
\[
 v_i=u_i+t\delta_x, \quad v_j=u_j \text{ when }j\neq i,
\]
for some $t\geq0$ and repeat the calculation in the previous part to get
\[
0\leq \|\nabla\vv\|^2-\|\nabla\uu\|^2 = t^2d(x)+2t\cL u_i(x).
\]
Since $\cL u_i(x)<0$, we can choose small value for $t$ to arrive at a contradiction.
\end{proof}

\begin{remark}
Problem \eqref{L2-Problem} has not necessary a unique solution.
For example, in a symmetry model, there are different choices for classification.
In a toy example, consider a graph with four vertices $A$, $B$, $C$ and $D$.
Let $w_{AB}=w_{BC}=w_{CD}=w_{AD}=1$ and $w_{AC}=w_{BD}=0$.
Also, $A$ and $C$ are boundary points with boundary data $u_1(A)=u_2(C)=1$ and $u_1(C)=u_2(A)=0$.
This problem has four solutions

\begin{enumerate}[(i)]
\item
$u_1(B)=u_1(D)=\frac12,  \quad u_2(B)=u_2(D)=0,$
\item
$u_2(B)=u_2(D)=\frac12,  \quad u_1(B)=u_1(D)=0,$
\item
$u_1(B)=u_2(D)=\frac12,  \quad u_2(B)=u_1(D)=0,$
\item
$u_1(D)=u_2(B)=\frac12,  \quad u_1(B)=u_2(D)=0.$
\end{enumerate}
\end{remark}


\section{Gradient projection method}\label{section:Gradient projection method}
Gradient projection is one method that we use to solve the problem \eqref{L2-Problem}.
In the sequel, we use the following notation for averaging of a function
\[
\cA u(x):= \frac 1{d(x)}\sum_{y\in \cX} w_{xy} u(y),
\]
where 
\[
d(x) = \sum_{y\in \cX} w_{xy},
\]
and the admissible set
\[
\cS:=\left\{\uu=(u_1,\cdots, u_k)\in \left(\ell^2(\cX)\right)^k: u_i=\phi_i\text{ on } \Gamma, u_i\geq0, u_i\cdot u_j=0 \text{ for }i\neq j\right\}.
\]
We also use the projection $\cP:\left(\ell^2(\cX)\right)^k\rightarrow \cS$ which is defined as follows.
For every $\vv\in \left(\ell^2(\cX)\right)^k$ and every $x\in \cX\setminus \Gamma$, first we find
\[
i_x:=\arg\max_{1\leq j\leq k} (v_j(x))^+
\]
and if it has more than one solution we choose the smallest index. ($v^+:=\max(v,0)$.)
Then we define $\uu:=\cP\vv$ with $u_{i_x}(x)=(v_{i_x}(x))^+$ and $u_j(x)=0$ for $j\neq i_x$.
For any $x\in \Gamma$, we obviously define $u_i(x)=\phi_i(x)$.

The following lemma shows why $\cP$ is a projection.

\begin{lemma}\label{lem:projection}
Consider $\vv\in \left(\ell^2(\cX)\right)^k$, then
\[
\|\vv-\cP\vv\| \leq \|\vv-\ww\|,\quad \text{ for all }\ww\in\cS.
\]
\end{lemma}
\begin{proof}
Consider $\ww\in\cS$ and define the index function $\sigma:\cX\ra \{1,\cdots, k\}$ such that  $w_j(x)=0$ for $j\ne \sigma(x)$.
So,
\[\begin{split}
\|\vv-\ww\|^2=&\sum_{x\in\cX}\left((v_{\sigma(x)}(x)-w_{\sigma(x)}(x))^2+\sum_{i\ne \sigma(x)}(v_i(x))^2\right)\\
=&\sum_{x\in\cX}\sum_{i=1}^k (v_i(x))^2 + \sum_{x\in\cX}\left((v_{\sigma(x)}(x)-w_{\sigma(x)}(x))^2 - (v_{\sigma(x)}(x))^2\right)
\end{split}\]
Similarly we have,
\[
\|\vv-\cP\vv\|^2=\sum_{x\in\cX}\sum_{i=1}^k (v_i(x))^2  + \sum_{x\in\cX}\left((v_{i_x}(x)-(v_{i_x}(x))^+)^2 - (v_{i_x}(x))^2\right).
\]
It is enough to show
\begin{equation}\label{grad-proj:eq1}
(v_{i_x}(x)-(v_{i_x}(x))^+)^2 - (v_{i_x}(x))^2 \le (v_{\sigma(x)}(x)-w_{\sigma(x)}(x))^2 - (v_{\sigma(x)}(x))^2,
\end{equation}
for every $x\in \cX$.
If $v_{i_x}(x)\le 0$, then $v_{\sigma(x)}(x)\le 0$ by the definition of $i_x$.
Thus the left hand side of \eqref{grad-proj:eq1} is zero as well as the right hand side is positive (recall that $w_{\sigma(x)}(x)\ge0$).
If $v_{i_x}(x)\ge 0\ge v_{\sigma(x)}(x)$, the left hand side of \eqref{grad-proj:eq1} is negative and the right hand side will be positive.
If $v_{i_x}(x)\ge v_{\sigma(x)}(x) \ge 0$, then \eqref{grad-proj:eq1} will hold trivially.
\end{proof}

Our algorithm according to the gradient projection method is as follows:
\begin{enumerate}[(1)]
\item
Choose an initial guess in $\cS$. It might be an extension of boundary data as $u_{i,0}=\phi_i$ on $\Gamma$ and $u_{i,0}=0$ in $\cX\setminus\Gamma$.

\item
For $t\geq 0$, calculate the gradient of the cost function $J$ at $\uu^t=(u_{1,t},\cdots,u_{k,t})$. It is equal to
\[
\delta J(\uu^t):=(\cL u_{1,t},\cdots,\cL u_{k,t}).
\]

\item
Update the value of each components by
\[
\uu^{t+1}:=\cP(\uu^{t}-\frac 1d\cL\uu^t)=\cP(\cA\uu^t).
\]

\item
If $\| \uu^{t+1} - \uu^t\|$ is small then stop the algorithm, otherwise set $t=t+1$ and iterate the previous steps.
\end{enumerate}

The following proposition proves why the algorithm works.

\begin{proposition}
Assume $\uu$ is a solution of problem \eqref{L2-Problem}.
Consider an arbitrary point $x\in \cX\setminus\Gamma$ such that $u_i(x)>0$, then
\[u_i(x)=\cA u_i(x) \geq \cA u_j(x),\qquad\text{ for all }j\neq i.\]
\end{proposition}
\begin{proof}
For a fixed index $j\neq i$,  define a competitor $\vv\in\cS$
\[
v_i:=u_i- u_i(x)\delta_x,\quad v_j:=u_j+ u_i(x)\delta_x, \quad v_\ell:=u_\ell,\text{ for }\ell\neq i,j.
\]
Therefore,
\[\begin{split}
0\leq & \|\nabla\vv\|^2-\|\nabla\uu\|^2 = \|\nabla v_i\|^2+\|\nabla v_j\|^2-\|\nabla u_i\|^2-\|\nabla u_j\|^2 \\
= &\sum_{y}w_{xy}\left((v_i(x)-v_i(y))^2+(v_j(x)-v_j(y))^2-(u_i(x)-u_i(y))^2-(u_j(x)-u_j(y))^2\right)\\
=&\sum_{y\neq x}w_{xy}\left(u_i(y)^2+(u_i(x)-u_j(y))^2-(u_i(x)-u_i(y))^2-u_j(y)^2\right)\\
=&\sum_{y\neq x}2w_{xy}u_i(x)\left(u_i(y)-u_j(y)\right).
\end{split}\]
Since $u_i(x)>0$, we get
\[
\frac 1{d(x)}\sum_{y}w_{xy}u_i(y)\geq  \frac 1{d(x)}\sum_{y}w_{xy}u_j(y)=\cA u_j(x).
\]
Now apply result $(ii)$ of Theorem \ref{property-minimizers}, we obtain that $u_i(x)=\frac 1{d(x)}\sum\limits_{y}w_{xy}u_i(y)$.
\end{proof}

Define the map $\cG:(\ell_+^2(\cX))^k\longrightarrow \cS$ with rule $\cG\uu=\vv$, where
\[
v_i =\max \left( u_i-\sum_{j\ne i} u_j , 0\right),
\]
and $\ell_+^2(\cX)$ is the set of nonnegative functions.
If we replace this map instead of projection $\cP$ in the gradient projection algorithm, we will obtain  the segregation method.
We will study this method in Section \ref{cluste_seg}.

\begin{proposition}
Suppose $\cP$ is the projection on $\cS$ defined in Section \ref{section:Gradient projection method}.
Then there is a positive constant  $C_0$ such that
\[
\norm{\cP \uu -\uu} \le \norm{\cG \uu-\uu}\le C_0\norm{\cP \uu -\uu},
\]
for every $\uu\in (\ell_+^2(\cX))^k$.
\end{proposition}
\begin{proof}
The left inequality will be hold according to Lemma \ref{lem:projection}.
For a fixed node $x\in\cX$, we need to show
\begin{equation}\label{prop:segr-proj-eq1}
\sum_{j=1}^k|v_j(x)-u_j(x)|^2 \le C_0 \sum_{j=1}^k |w_j(x)-u_j(x)|^2,
\end{equation}
where $\cG\uu=\vv$ and $\cP\uu=\ww$.
Let $i:=\arg\max_{1\leq j\leq k} u_j(x)$.
If there is more than one index for $i$, then $\vv(x)=\cG(\uu)(x)=0$.
Thus \eqref{prop:segr-proj-eq1} holds for $C_0\ge 2$.

Now assume that $i$ is the unique solution of $i:=\arg\max_{1\leq j\leq k} u_j(x)$. Hence,
\[
w_i(x)=u_i(x),\text{ and } w_j(x)=v_j(x)=0\text{ for }j\ne i.
\]
Therefore, using the definition of $v_i$ and applying Cauchy-Schwartz inequality we obtain
\[\begin{split}
\sum_{j=1}^k|v_j(x)-u_j(x)|^2 &= \left(\sum_{j\ne i} u_j(x)\right)^2 + \sum_{j\ne i}|u_j(x)|^2 \\
&\le ((k-1)+1)\sum_{j\ne i}|u_j(x)|^2 =k\sum_{j=1}^k|w_j(x)-u_j(x)|^2,
\end{split}\]
which implies that \eqref{prop:segr-proj-eq1} holds for $C_0\ge k$.
\end{proof}

\section{Penalization method}

In this section, we apply the penalization method to solve problem \eqref{L2-Problem}.
Since finding the solution directly is not efficient (the optimization problem \eqref{L2-Problem} is a problem with $(n-m)^k$ parameters), we would prefer to solve a PDE instead.
In this case, we can just find a PDE for points that  $u_i>0$ and this subdomain is not known.
In fact, we have a free boundary problem and if we know the domain $\{u_i>0\}$, we are able to find the solution.
In order to overcome this difficulty, we relax the constraint with a penalty term and try to estimate the solution for the original problem \eqref{L2-Problem}.

Indeed, we consider the following problem
\begin{equation}\label{L2-Problem*}
\min_{\uu\in\cK}J_\veps(\uu):=\sum_{i=1}^k\|\nabla u_i\|_{\ell^2(X^2)}^2+\frac1\veps\sum_{i\neq j}(u_i^2,u_j^2).
\end{equation}
Since the energy function is convex, it is straightforward that the problem has a unique solution which satisfies
\begin{equation}\label{PDE-penalized}
\cL u_i+\frac{u_i}\veps\sum_{j\neq i}u_j^2=0,\quad\text{ in }X\setminus\Gamma.
\end{equation}
Furthermore, we know that the solution is nonnegative due to the maximum principle, Proposition \ref{Maximum principle}.

\begin{theorem}
Let $\uu^\veps$ be the solution of \eqref{L2-Problem*} for every $\veps>0$.
For any sequence $\veps_n\rightarrow0$, there is a  subsequence of $\uu^{\veps_n}$ which converges to  a minimizer of \eqref{L2-Problem}.
\end{theorem}
\begin{proof}
Let $\vv$ be an arbitrary minimizer of \eqref{L2-Problem}, then we have $J_\veps(\vv)=J(\vv)$ thanks to the constraint in \eqref{L2-Problem}. So,
$$J_\veps(\uu^\veps)\leq J_\veps(\vv)=J(\vv)=:\Lambda.$$
Therefore $\|\nabla \uu^\veps\|\leq \sqrt\Lambda$ is uniformly bounded and then by Poincar\'e inequality we get $\|\uu^\veps\|$ is bounded, since
$$\lambda_1\|u_i^\veps-\phi_i\|\leq \|\nabla(u_i^\veps-\phi_i)\|\leq\sqrt\Lambda+\|\nabla\phi_i\|.$$
Hence, toward a subsequence we can assume that $\uu^{\veps_n}\rightarrow\uu\in\cK$.
We need to show that $\uu$ is a minimizer of  \eqref{L2-Problem} and  satisfies its constraints.
First, we have
$$
\frac1{\veps_n}\sum_{i\neq j}\left((u_i^{\veps_n})^2,(u_j^{\veps_n})^2\right)\leq J_{\veps_n}(\uu^{\veps_n})\leq \Lambda,
$$
so,
$$
\sum_{i\neq j}\left((u_i^{\veps_n})^2,(u_j^{\veps_n})^2\right)\longrightarrow0.
$$
Thus, $(u_i^2,u_j^2)=0$ and taking into account that $u^\veps_i$ is nonnegative we obtain that the constraint in \eqref{L2-Problem} holds for $\uu$.
To close the argument, note that
$$\|\nabla\uu\|^2=\lim_{\veps_n\rightarrow0}\|\nabla\uu^{\veps_n}\|^2\leq J_{\veps_n}(\uu^{\veps_n})\leq \Lambda.$$
So, $J(\uu)=\Lambda$ and $\uu$ is a minimizer.
\end{proof}

In the sequel we introduce an algorithm to solve problem \eqref{L2-Problem*} or its equivalent version  \eqref{PDE-penalized}.
The later is a nonlinear system of PDEs and is not easy to solve directly.
For an explanation of our algorithm, we define the following sequence which converges to the solution \eqref{PDE-penalized}.
First, consider the harmonic extension $u_{i,0}$ of boundary data given by
$$
\left\{\begin{array}{ll}
\cL u_{i,0}=0, & \text{ in }X\setminus\Gamma,\\[8pt]
u_{i,0}=\phi_i& \text{ on }\Gamma,
\end{array}\right.
$$
which is a nonnegative function according to the maximum principle.
Next, given nonnegative functions $\uu_m:=(u_{1,m},\cdots,u_{k,m})$, let $u_{i,m+1}$ be the solution of the following system
$$
\left\{\begin{array}{ll}
\cL u_{i,m+1}+\frac{u_{i,m+1}}{\veps}\sum\limits_{j\neq i}u_{j,m}^2=0, & \text{ in }X\setminus\Gamma,\\[10pt]
u_{i,m+1}=\phi_i,& \text{ on }\Gamma,
\end{array}\right.
$$
The following theorem shows that why our algorithm  works for solving problem \eqref{PDE-penalized}.

\begin{theorem}\label{thm-order-seq}
Suppose that the boundary data $\phi_i$ satisfy \eqref{phi-assumption}.
Then the  sequence $\uu_m$ makes the following order
\begin{equation}\label{order-sequence-u}
1\geq u_{i,0}\geq u_{i,2}\geq\cdots\geq u_{i,2m}\geq\cdots\geq u_{i,2m+1}\geq\cdots\geq u_{i,3}\geq u_{i,1}\geq0.
\end{equation}
Moreover, the limit of this sequence is the solution of \eqref{PDE-penalized}.
\end{theorem}

\begin{proof}
{\bf Step 1: } We show that $u_{i,m}$ is nonnegative. \\
This is a matter of maximum principle, Proposition  \ref{Maximum principle}.
For $m=0$, it is obvious due to maximum principle and taking account that $u_{i,0}$ is a harmonic function. In fact, we consider $p(x)\equiv0$ in Proposition \ref{Maximum principle}.
To show $u_{i,m+1}\geq0$, again apply Proposition  \ref{Maximum principle} for nonnegative function $p(x)=\frac1{\veps}\sum_{j\neq i}u_{j,m}^2$.

\medskip
{\bf Step 2: } $u_{i,0}\leq 1$. \\
Apply the maximum principle for harmonic function $1 - u_{i,0}$ and recall the assumption \eqref{phi-assumption}.

\medskip
{\bf Step 3: } In this step we show that $u_{i,m}\leq u_{i,0}$.\\
We just need to note that $\cL u_{i,m+1}\leq 0=\cL u_{i,0}$. Then maximum principle yields that $u_{i,m+1}\leq u_{i,0}$.

\medskip
{\bf Step 4:} Here, we claim that $u_{i,2}\geq u_{i,1}$.\\
By the result of Step 2, we can write
\begin{align*}
0&=\cL u_{i,2}+\frac{u_{i,2}}{\veps}\sum_{j\neq i}u_{j,1}^2\\
&\leq \cL u_{i,2}+\frac{u_{i,2}}{\veps}\sum_{j\neq i}u_{j,0}^2.
\end{align*}
This together with the equation of $u_{i,1}$,  the maximum principle yields that $u_{i,2}\geq u_{i,1}$.

\medskip
{\bf Step 5:} Now we close the argument with the induction. Assume that
\begin{equation}\label{order-u-induction}
u_{i,0}\geq u_{i,2}\geq\cdots\geq u_{i,2m}\geq u_{i,2m-1}\geq\cdots\geq u_{i,1}\geq0.
\end{equation}
for some $m\geq1$, we will extend the string for $m+1$.
By the following inequality
$$
0=\cL u_{i,2m+1}+\frac{u_{i,2m+1}}{\veps}\sum_{j\neq i}u_{j,2m}^2\geq \cL u_{i,2m+1}+\frac{u_{i,2m+1}}{\veps}\sum_{j\neq i}u_{j,2m-1}^2
$$
we can apply Proposition  \ref{Maximum principle} for function $(u_{i,2m}-u_{i,2m+1})$ when $p=\frac{1}{\veps}\sum_{j\neq i}u_{j,2m-1}^2$
to deduce that $u_{i,2m}\geq u_{i,2m+1}$.
Similarly, we have
$$
0=\cL u_{i,2m+1}+\frac{u_{i,2m+1}}{\veps}\sum_{j \neq i}u_{j,2m}^2\leq \cL u_{i,2m+1}+\frac{u_{i,2m+1}}{\veps}\sum_{j \neq i}u_{j,2m-2}^2
$$
according the induction assumption $u_{j,2m-2}\geq u_{j,2m}$.
Comparing with the equation for $u_{i,2m-1}$ we obtain $u_{i,2m+1}\geq u_{i, 2m-1}$.
Now repeat this argument to show that $u_{i,2m}\geq u_{i,2m+2}\geq u_{i,2m+1}$.

%

\medskip
{\bf Step 6:}
From \eqref{order-sequence-u}, we know that  there are $\overline u_i$ and $\underline u_i$ as the limit of
\begin{align*}
u_{i,2m}\longrightarrow \overline u_i,\\
u_{i,2m+1}\longrightarrow \underline u_i.
\end{align*}
These limits satisfy
\begin{align}\label{limit-equation}
\left\{\begin{array}{ll}
\cL\overline u_i+\frac{\overline u_i}{\veps}\sum_{j\neq i}\underline u_j^2=0, \quad\text{ in }\cX\setminus\Gamma\\[10pt]
\cL\underline u_i+\frac{\underline u_i}{\veps}\sum_{j\neq i}\overline u_j^2=0, \quad\text{ in }\cX\setminus\Gamma\\[10pt]
\overline u_i=\underline u_i=\phi_i\quad\text{ on }\Gamma.
\end{array}\right.
\end{align}
Multiply in inner product $\ell^2(\cX)$ both equations by $\overline u_i - \underline u_i$ and subtract to get
\[
\veps\|\nabla (\overline u_i - \underline u_i)\|_{\ell^2(\cX^2)}^2  = \Big( \underline u_i (\overline u_i - \underline u_i), \sum_{j\neq i}\overline u_j^2\Big)-\Big( \overline u_i (\overline u_i - \underline u_i), \sum_{j\neq i}\underline u_j^2\Big).
\]
It is worthwhile noticing that although the equation holds in $\cX\setminus\Gamma$, since $\overline u_i - \underline u_i=0$ on $\Gamma$ we are able to utilize the relation \eqref{green-formula}.
Now sum over $i$, we obtain
\begin{align*}
\veps\|\nabla (\overline \uu - \underline \uu)\|_{\ell^2(\cX^2)}^2
= & \sum_i \Big( \underline u_i \overline u_i , \sum_{j\neq i}(\overline u_j^2+\underline u_j^2)\Big)
-\sum_i \Big(  \underline u_i^2, \sum_{j\neq i}\overline u_j^2 \Big)
-\sum_i \Big(  \overline u_i^2, \sum_{j\neq i}\underline u_j^2 \Big) \\
= & \sum_{i,j} \big( \underline u_i \overline u_i , \overline u_j^2+\underline u_j^2\big) - \sum_i \big( \underline u_i \overline u_i , \overline u_i^2+\underline u_i^2\big) \\
& - 2 \sum_{i,j} \big(  \underline u_i^2, \overline u_j^2 \big) +  2 \sum_i  \big(  \underline u_i^2, \overline u_i^2 \big)\\
= & \sum_{x\in \cX} \sum_{i,j}  \underline u_i(x) \overline u_i(x) (\overline u_j^2(x)+\underline u_j^2(x)) \\
& - \sum_{x\in\cX} \sum_i  \underline u_i(x) \overline u_i(x) (\overline u_i(x)-\underline u_i(x))^2 \\
& -2 \sum_{x\in \cX} \Big(\sum_i (\underline u_i(x))^2 \Big)\Big( \sum_i (\overline u_i(x))^2 \Big) \\
\leq & 2 \sum_{x\in \cX}  \left(\sum_{i}  \underline u_i(x) \overline u_i(x)  -  \Big(\sum_i (\underline u_i(x))^2 \Big)\Big( \sum_i (\overline u_i(x))^2 \Big) \right) \leq 0,
\end{align*}
where we have used the relation $0\leq \underline u_i \leq  \overline u_i \leq 1$ and in the last line the Cauchy-Schwartz inequality has been applied.

Then $\|\nabla (\overline \uu - \underline \uu)\|_{\ell^2(\cX^2)}^2\leq 0$ and so, $\overline \uu - \underline \uu$ is  constant in $\cX$. Taking  this along to the boundary condition  implies that $\overline \uu = \underline \uu$ in $\cX$.
Recall \eqref{limit-equation}, $\underline u_i =\overline u_i$ is a solution of \eqref{PDE-penalized}.
\end{proof}

 \section{The main algorithm for clustering}\label{cluste_seg}
 In the previous sections of the paper we considered a minimization problem \eqref{L2-Problem}, which unfortunately has no unique solution over connected graphs. In the current section in order to overcome the lack of uniqueness we consider different functional and prove the existence and uniqueness of the minimizer. The definition of a new functional is inspired from the numerical results of the spatial segregation of reaction-diffusion systems (see \cite{Av1}).
\subsection{Existence and  uniqueness of a minimizer}

We introduce the discrete counterpart of the spatial segregation problem defined on connected graphs.
In the rest of the paper the following notation
\[
\hat{z}_q=z_q-\sum_{j\neq q}z_j,
\]
for elements  $(z_1,z_2,\dots,z_k)$, will play a crucial role.
Let    $\overline{u}_{i}(x_{l} )$ for $ i=1,2\cdots k$  denote  the average value  of $u_{i}$ for all neighbor points of $x_{l}:$
\[
\overline{u}_{i}(x_{l})= \frac{1}{\deg(x_{l})} \sum_{p  \sim l} w_{lp} u_{i} (x_p),
\]
where
\[
\deg(x_{l})=  \sum_{(x_l,y)\in E} w_{x_ly},
\]
and $V$ and $E$ stand for a set of vertices and edges respectively. We will set a graph $\cX$ to be a tuple $(V, E)$ in the rest of this section.

%

When $\cX$ is a connected graph and also consist of discrete and finite number of points, it turns out that we have to consider slightly different functional (see \cite[Section $2$]{Av1}). Since $\mathcal L$ is a self-adjoint operator, then we set:
 \begin{equation}
\label{sun2}
J(u_1,\dots,u_k)=\frac{1}{2}  \sum_{i=1}^{k} \|\nabla u_i\|^2_{\ell^{2}(X^2)}- \sum_{i\neq j} \left(\nabla u_i, \nabla u_j\right)_{\ell^{2}(\cX^2)}, \end{equation}
over the set
\begin{equation}\label{disc_min_set}
	\mathcal K=\left\{\uu=(u_1,\cdots, u_k)\in \left(\ell^2(\cX)\right)^k: u_i=\phi_i\text{ on } \Gamma, u_i\geq0, u_i\cdot u_j=0 \text{ for }i\neq j\right\}.
\end{equation}


\begin{theorem}
	The following minimization problem
	\begin{equation}\label{disc_min_problem}
	\inf_{\mathcal K}J(u_1,u_2,\dots,u_k)\
	\end{equation}
	has a solution.
\end{theorem}

\begin{proof}
The proof repeats the same lines as in Theorem $2$ in \cite{Av1}. In \cite{Av1} the functional is defined for standard difference scheme, but it can be easily concluded for the connected graphs as well.
\end{proof}

Now, by following the proofs of Proposition $1$ and Lemma $2$  for $F_l(x,s)=0$ in \cite{Av1}, We can observe that the similar results can be obtained for connected graphs instead of finite difference discretization. Although, it is worth to notice that the standard finite difference grid is itself a particular case of connected graphs.

Thus, we conclude the following  result:

\begin{theorem}\label{limit_pde_discrete}
	For every minimizer $(u_1,\dots,u_k)\in \mathcal K,$ the following properties hold:
\begin{itemize}
	\item $\mathcal{L}\hat{u}_i(x) =0 \;\; whenever \;\; u_i(x)>0.$
	\item $\mathcal{L}\hat{u}_i(x)\geq 0 \;\; whenever \;\; x\in \cX \setminus\Gamma.$
\end{itemize}
\end{theorem}


To prove the uniqueness of the minimizer  $(u_1,\dots,u_k)\in \mathcal K$ one needs some technical lemmas.
\begin{lemma}\label{lemma1}
	Let $\cX=(V,E)$ be a connected graph. If any two vectors $(u_{1},u_{2},\dots,u_{k})$ and $(v_{1},v_{2},\dots,v_{k})$ are  minimizers to the \eqref{disc_min_problem}, then the following equation holds:
	\[
	\max_{x\in \cX}\left(\hat{u}_l(x)-\hat{v}_l(x)\right)=\max_{\{ x\in \cX\;:\; u_l(x)\leq v_l(x)\}}\left(\hat{u}_l(x)-\hat{v}_l(x)\right),
	\]
	for all $l=1,2,\dots,k$.
\end{lemma}

\begin{proof}
	We argue by contradiction. Suppose for some $l_0$ we have
	\begin{equation}\label{init_assmp}
		\begin{multlined}
			\hat{u}_{l_0}(x_0)-\hat{v}_{l_0}(x_0)=
			\max_{x\in  \cX}(\hat{u}_{l_0}(x)-\hat{v}_{l_0}(x))=\\=
			\max_{\{ x\in \cX\;:\; u_{l_0}(x)> v_{l_0}(x)\}}(\hat{u}_{l_0}(x)-\hat{v}_{l_0}(x))>
			\max_{\{ x\in \cX\;:\; u_{l_0}(x)\leq v_{l_0}(x)\}}(\hat{u}_{l_0}(x)-\hat{v}_{l_0}(x)).
		\end{multlined}
	\end{equation}	
It is easy to observe that  the following simple chain of inclusions hold:
	\begin{equation}\label{incl_chain}
		\{u_l(x)> v_l(x)\}\subset\{\hat{u}_l(x)> \hat{v}_l(x)\}\subset\{u_l(x)\geq v_l(x)\}.
	\end{equation}
	We obviously see that $ u_{l_0}(x_0)> v_{l_0}(x_0)\geq 0 $ implies
	$\hat{u}_{l_0}(x_0)>\hat{v}_{l_0}(x_0)$. On the other hand,  Theorem \ref{limit_pde_discrete}   gives us
	$$
     \mathcal{L}\hat{u}_i(x_0)= 0,
	$$
	and	
	$$
     \mathcal{L}\hat{v}_i(x_0)\geq 0.	 	
	$$
	Therefore 	
	\begin{equation*}
	\mathcal{L}(\hat{u}_i-\hat{v}_i)(x_0)\leq 0.
	\end{equation*}
	Thus,
	
	\begin{align}
		0  < \left(\hat{u}_{l_0}(x_0)- \hat{v}_{l_0}(x_0)\right)&\leq
	 \frac{1}{\deg(x_0)}\sum_{y\in \cX}w_{x_{0}y}\left(\hat{u}_{l_0}(y)-\hat{v}_{l_0}(y)\right)\nonumber,
	\end{align}
	which implies that $\hat{u}_{l_0}(x_0)-\hat{v}_{l_0}(x_0)=\hat{u}_{l_0}(y)-\hat{v}_{l_0}(y)>0,$ when $w_{x_{0}y}\ne0$. Due to the chain  \eqref{incl_chain}, we apparently have  ${u}_{l_0}(y)\geq {v}_{l_0}(y)$. According to our assumption \eqref{init_assmp}, the only possibility is ${u}_{l_0}(y)>{v}_{l_0}(y)$  for all $y\in \cX$.
	Now we can proceed the previous steps for all $y\in V$ such that $(x_0,y)\in E,$ and then for each one we will get corresponding neighbours with the same strict inequality and so on. Since the graph $\cX$ is connected, then  one can always find the shortest path from a given vertex $y$ to the vertex belonging $\Gamma.$   Continuing above procedure along this path we will finally approach to a vertex on $\Gamma,$ where as we know ${u}_{l_0}(x)={v}_{l_0}(x)={\phi}_{l_0}(x)$ for all $x\in\Gamma$. Hence, the strict inequality  fails, which implies that our initial assumption \eqref{init_assmp} is false. Observe that the same arguments can be applied if we interchange the roles of ${u}_{l}(x)$ and ${v}_{l}(x)$. Thus, we also have
	\[
	\max_{V}\left(\hat{v}_l(x)-\hat{u}_l(x)\right)=\max_{\{ v_l(x)\leq u_l(x)\}}\left(\hat{v}_l(x)-\hat{u}_l(x)\right),
	\]
	for every $l=1,2,\dots,m$.

	Particularly, for every fixed $l=1,2.\dots,m$ and $x\in V$ we have	\begin{multline}\label{double_ineq}
		-\max_{\{v_l(x)\leq u_l(x)\}}(\hat{v}_l(x)-\hat{u}_l(x))=	-\max\limits_{x\in V}(\hat{v}_l(x)-\hat{u}_l(x))\leq\\\leq \hat{u}_l(x)-\hat{v}_l(x) \leq 	\max\limits_{x\in V}(\hat{u}_l(x)-\hat{v}_l(x))=\max_{\{u_l(x)\leq v_l(x)\}}(\hat{u}_l(x)-\hat{v}_l(x)).
	\end{multline}
\end{proof}

Thanks to Lemma \ref{lemma1} in the sequel we will use the following notations:
$$
A:=\max_l\;\left(\max\limits_{x\in V}(\hat{u}_l(x)-\hat{v}_l(x))\right)=\max_l\;\left(\max\limits_{\{u_l(x)\leq v_l(x)\}}(\hat{u}_l(x)-\hat{v}_l(x))\right),
$$
and
$$
B:=\max_l\;\left(\max\limits_{x\in V}(\hat{v}_l(x)-\hat{u}_l(x))\right)=\max_l\;\left(\max\limits_{\{v_l(x)\leq u_l(x)\}}(\hat{v}_l(x)-\hat{u}_l(x))\right).
$$
Next lemma we write down without a proof. The proof can be easily adapted from  \cite{Av1}[Lemma $4$].
\begin{lemma}\label{lemma2}
	Let $\cX=(V,E)$ be a connected graph. Assume given two vectors $(u_{1},u_{2},\dots,u_{k})$ and $(v_{1},v_{2},\dots,v_{k})$ are  minimizers to the \eqref{disc_min_problem}. For them  we set $A$ and $B$ as defined above. If $A>0$ and it is attained for some $l_0$, then $A=B>0$  and there exists some $t_0\neq l_0,$ and $y_0\in V,$ such that
	$$
	0<A=\max_{\{u_{l_0}(x)\leq v_{l_0}(x)\}}(\hat{u}_{l_0}(x)-\hat{v}_{l_0}(x))=
	\max_{\{u_{l_0}(x)= v_{l_0}(x)=0\}}(\hat{u}_{l_0}(x)-\hat{v}_{l_0}(x))
	=\hat{v}_{t_0}(y_0)-\hat{u}_{t_0}(y_0).
	$$	
\end{lemma}
Now, we are ready to proof the uniqueness of the minimizer. The following Theorem is true.
\begin{theorem}[Uniqueness]
	Let $\cX=(V,E)$ be a connected graph. Then
	there exists  a unique minimizer  $(u_{1},u_{2},\dots,u_{k})\in \mathcal K,$ to minimization problem  \eqref{disc_min_problem}.

\end{theorem}
\begin{proof}
	
	Let two vectors $(u_{1},u_{2},\dots,u_{k})$ and $(v_{1},v_{2},\dots,v_{k})$ are minimizers to \eqref{disc_min_problem}. For these vectors we set the definition of $A$ and $B$. Then, we consider two cases $A\leq 0$ and $A>0$. If we assume that $A\leq 0,$ then according to Lemma \ref{lemma2}, we get  $B\leq 0$. But if  $A$ and $B$ are non-positive, then the uniqueness follows. Indeed,  due to  \eqref{double_ineq} we have the following obvious inequalities
	$$
	0\leq -B\leq \hat{u}_l(x)-\hat{v}_l(x)\leq A\leq 0.
	$$
	This provides  for every $l=\overline{1,k}$ and $x\in V$ we have $\hat{u}_l(x)=\hat{v}_l(x),$ which in turn implies $$
	{u}_l(x)={v}_l(x).
	$$
	Now suppose  $A>0$. Our aim is to prove that this case leads to a contradiction. Let the value $A$ is attained for some $l_0\in\overline{1,k},$ then
	due to Lemma \ref{lemma2} there exist $y_0\in V$ and $t_0\neq l_0$ such that:
	\begin{align*}
		0<A=B=&\max_{\{u_{l_0}(x)\leq v_{l_0}(x)\}}(\hat{u}_{l_0}(x)-\hat{v}_{l_0}(x))\\&=\max_{\{u_{l_0}(x)=v_{l_0}(x)=0\}}(\hat{u}_{l_0}(x)-\hat{v}_{l_0}(x))=\hat{v}_{t_0}(y_0)-\hat{u}_{t_0}(y_0).
	\end{align*}
Thus, along with the fact that $\hat{v}_{t_0}(y_0)>\hat{u}_{t_0}(y_0)$ implies ${v}_{t_0}(y_0)\geq {u}_{t_0}(y_0),$ we can repeat the same steps as in the proof of Lemma \ref{lemma1} to obtain
	$$
	\left(\hat{v}_{t_0}(y_0)- \hat{u}_{t_0}(y_0)\right)\leq \frac{1}{\deg(y_0)}\sum_{(y_0,z)\in E}w_{y_{0}z}\left(\hat{v}_{t_0}(z)-\hat{u}_{t_0}(z)\right).
	$$
	This  implies  $A=\hat{v}_{t_0}(y_0)-\hat{u}_{t_0}(y_0)=\hat{v}_{t_0}(z)-\hat{u}_{t_0}(z)>0$ for all $(y_0,z)\in E$.  The chain \eqref{incl_chain} provides that  for all $(y_0,z)\in E,$ we have ${v}_{t_0}(z)\geq{u}_{t_0}(z)$. Since a graph $\cX$ is connected, then one can always find a shortest path from $y_0$ to some vertex $w\in\Gamma.$ Assume the vertices along this path are $y_0;y_1;\dots;y_{k-1};y_q=w.$ Hence, for every $0 \leq j\leq q-1,$ we have $(y_j,y_{j+1})\in E,$ i.e. every vertex  $y_{j+1}$ is a closest neighbor for $y_{j}$ and $y_{j+2}.$

	According to the above arguments for the neighbor vertex $y_1\in V$ we proceed as follows: If  ${v}_{t_0}(y_1)>{u}_{t_0}(y_1),$ then obviously
	$$
	\left(\hat{v}_{t_0}(y_1)- \hat{u}_{t_0}(y_1)\right)\leq \frac{1}{\deg(y_1)}\sum_{(y_1,z)\in E}w_{y_{1}z}\left(\hat{v}_{t_0}(z)-\hat{u}_{t_0}(z)\right).
	$$
	This, as we saw a few lines above, leads to  $A=\hat{v}_{t_0}(y_1)-\hat{u}_{t_0}(y_1)=\hat{v}_{t_0}(z)-\hat{u}_{t_0}(z)>0$ for all $(y_1,z)\in E$.  In particular,
	$A = \hat{v}_{t_0}(y_2)-\hat{u}_{t_0}(y_2)> 0.$
	
	If  ${v}_{t_0}(y_1)={u}_{t_0}(y_1),$  then due to $\hat{v}_{t_0}(y_1)-\hat{u}_{t_0}(y_1)=A=B>0,$ there exists some $\lambda_0\neq t_0,$ such that
	$$
	0<A=\hat{v}_{t_0}(y_1)-\hat{u}_{t_0}(y_1)=
	\sum_{l\neq t_0}\left({u}_{l}(y_1)-{v}_{l}(y_1) \right)={u}_{\lambda_0}(y_1)-\sum_{l\neq t_0}{v}_{l}(y_1).
	$$
	Note that ${u}_{\lambda_0}(y_1)>0$ implies ${u}_{l}(y_1)=0$ for all $l\neq \lambda_0,$ and particularly  ${v}_{t_0}(y_1)={u}_{t_0}(y_1)=0.$
	Following the definition of $A,$ we get
	$$
	A={u}_{\lambda_0}(y_1)-\sum_{l\neq t_0}{v}_{l}(y_1)\geq \hat{u}_{\lambda_0}(y_1)-\hat{v}_{\lambda_0}(y_1),
	$$
	which in turn gives $2\sum\limits_{l\neq \lambda_0}{v}_{l}(y_1)\leq 0,$ and therefore
	${v}_{l}(y_1)=0$ for all $l\neq \lambda_0$. Hence
	$$
	A={u}_{\lambda_0}(y_1)-\sum_{l\neq t_0}{v}_{l}(y_1)= \hat{u}_{\lambda_0}(y_1)-\hat{v}_{\lambda_0}(y_1),
	$$
	This suggests us to apply the same approach as above to arrive at
	$$
	\left(\hat{v}_{\lambda_0}(y_1)- \hat{u}_{\lambda_0}(y_1)\right)\leq \frac{1}{\deg(y_1)}\sum_{(y_1,z)\in E}w_{y_{1}z}\left(\hat{v}_{\lambda_0}(z)-\hat{u}_{\lambda_0}(z)\right),
	$$
	which leads to
	$A=\hat{u}_{\lambda_0}(y_1)-\hat{v}_{\lambda_0}(y_1)=\hat{u}_{\lambda_0}(z)-\hat{v}_{\lambda_0}(z)>0,$ for all $(y_1,z)\in E$. In particular,
	$A = \hat{u}_{\lambda_0}(y_2)-\hat{v}_{\lambda_0}(y_2)> 0.$
	Thus, combining two cases we observe that for $y_2\in V$ there exist an index $1\leq l_{y_2}\leq m$ (in our case $l_{y_2}=t_0$ or $l_{y_2}=\lambda_0$) such that
	\begin{equation}\label{contradiction}
		\mbox{either}\;\;	\hat{u}_{l_{y_2}}(y_2)-\hat{v}_{l_{y_2}}(y_2)=A, \;\;\mbox{or}\;\;
		\hat{u}_{l_{y_2}}(y_2)-\hat{v}_{l_{y_2}}(y_2)=-A.
	\end{equation}
	It is not hard to understand that the same procedure can be repeated for a vertex $y_2$ instead of $y_1$ and come to the same conclusion \eqref{contradiction} for $y_3\in V$ and some index $l_{y_3}$ and so on. This allows to claim that for every $y_j\in V$ along the path $(y_0,\dots,y_q)$  there exist some $l_{y_j}$ such that
	\[
	|\hat{u}_{l_{y_j}}(y_j)-\hat{v}_{l_{y_j}}(y_j)|=A>0.
	\]
	But this means that above equality holds also for $y_q=w\in \Gamma,$ which will lead to a contradiction, because for every $z\in\Gamma,$ and $l=1,\cdots,k$ one has $\hat{u}_{l}(z)-\hat{v}_{l}(z)=0$.  This completes the proof of uniqueness.
\end{proof}

\subsection{Semi-supervised learning algorithm}

Using definition of graph Laplacian in  $$\cL (u_{i}- \sum\limits_{j\neq i}u_j )\ge 0,$$ yields
\begin{equation}\label{17}
	\cL (u_{i} - \sum\limits_{j\neq i}u_{j})(x_l) =
	\sum\limits_{s=1}^{n} w_{ls}\, \left( u_{i}(x_{l})- u_{i}(x_s) -  \sum\limits_{j\neq i}(u_{j}(x_l) - u_{j}(x_s))  \right).
\end{equation}
To  obtain   $u_{i} (x_{l})$    from (\ref{17})  we impose   the following  conditions
\[
u_{i} (x_{l}) \cdot u_{j} (x_{l}) =0  \text{ and }    u_{i}(x_l)\geq 0.
\]
From these
\[
\deg(x_{l}) u_{i} (x_{l})-\sum\limits_{s=1}^{n} w_{ls}\,u_{i} (x_{s})+ \sum\limits_{s=1}^{n}\sum\limits_{j\neq i}w_{ls} \,u_{j} (x_{s})=0
\]
Then
\[
u_{i} (x_{l})=\overline{u}_{i}(x_{l})-\sum\limits_{j\neq i}\overline{u}_{j}(x_{l}).
\]

According to the above ideas and following the Theorem \ref{limit_pde_discrete} we can easily check that if $(u_1,u_2,\dots,u_k)\in \mathcal K$ is a unique minimizer to \eqref{disc_min_problem}, then  it satisfies the following system of equations:
\begin{equation}\label{scheme_sys}
	\begin{cases}
		u_{1}(x) =\max \left(\overline{u}_1(x) - \sum\limits_{p \neq 1}  \overline{u}_p(x), \,  0\right),\;\;x\in \cX\setminus\Gamma,\\
		u_{2}(x) =\max \left(\overline{u}_2(x) - \sum\limits_{p \neq 2}  \overline{u}_p(x), \,  0\right),\;\;x\in \cX\setminus\Gamma,\\
		\dots\dots\dots\dots\\
		u_{k}(x) =\max \left(\overline{u}_k(x) - \sum\limits_{p \neq k}  \overline{u}_p(x), \,  0\right),\;\;x\in \cX\setminus\Gamma,\\
		u_{i}(x) =\phi_{i}(x),\;\;x\in \Gamma,\; \mbox{for all}\; i=1,2,\dots,k.
	\end{cases}
\end{equation}
\begin{remark}\label{remark_1}
	We remark that the system \eqref{scheme_sys} itself implies the disjointness  property, i.e. it is easy to see that  if a vector $(u_1,u_2,\dots,u_k)$ satisfies the system \eqref{scheme_sys}, then  $u_i(x)\cdot u_j(x)=0,$ for every $x\in V$ and $i\neq j.$
\end{remark}

In order to approximate   the solution of   system \eqref{scheme_sys} we propose the iterative  scheme which is  easy to implement   as follows:
For  $i=1,\cdots k,$ and $x_l\in\cX\setminus\Gamma$ we set
\begin{equation*}
u_{i}^{(t+1)} (x_{l}) =\max\left(\overline{u}^{(t)}_{i}(x_{l})-\sum\limits_{j\neq i}\overline{u}^{(t)}_{j}(x_{l}),\,  0\right).
\end{equation*}

In the lite of Remark \ref{remark_1} it can be seen that for every iteration the disjointness property is preserved. In other words the following lemma is true.
\begin{lemma}\label{lemma}
		Let $\cX=(V,E)$ be a connected graph. The above iterative method  satisfies
	\[u_i^{(t)}(x)\cdot u_j^{(t)}(x) =0,\; \forall x\in V,\; i\neq j.\]
\end{lemma}

The label decision for vertex $x_l$ is determined by the strictly positive component
$u_i(x_l),$ i.e. find an index $i_0$ such that $u_{i_0}(x_l)>0$. Thus, in this case the label corresponding to a vertex $x_l$ will be $i_0.$

\

\section{Experimental results}

In this section we are going to test and compare two well-known semi-supervised learning algorithms to the one we have developed based on segregation theory. We note that our taken dataset for visual implementations will be the random generated half-moons and for the statistic analysis we will use the well-known MNIST. Thus, we will depict the predictions for Laplace learning, Poisson learning and our learning (We call it Segregation learning) algorithms.

We  run the learning algorithms for different number of initial label of classes and for different number of classes (basically we will run for $3, 4$ and $5$ classes). For each  implementation all the classes have the same number of nodes, i.e, either all classes have   $200$ or $300$ nodes. The reader can also observe the red nodes on every figure. They correspond to the randomly chosen initial known labels.
In the figures \ref{fig1}--\ref{fig15} below one can observe that, when the initial number of labels per class is small, i.e. $2, 3$ or $5$ labels, then the Laplace learning algorithm is performing poorly, whilst both the Poisson and our Segregation learning algorithms are performing  much better and have more or less  the same accuracy.

When the initial number of labels per class is $10$ or $20$ labels, then the performance of the Laplace learning becomes more accurate and is getting close to the results depicted for Poisson and Segregation learning algorithms.

Tables \ref{table_1} and \ref{table_2}  show the average accuracy over all $100$ trials for various low and high label rates. The implementations have been done on MNIST dataset only for $3$ classes.  We see that for low label rates Laplace learning performs poor as we noted in the depicted figures. On the other hand Poisson and Segregation learning perform better and predicted  more or less with the same accuracy. For high label rates Laplace learning performs much better  and gets close to Poisson and Segregation learning results.
\begin{figure}[!htb]
	\includegraphics[width=.32\linewidth,valign=m]{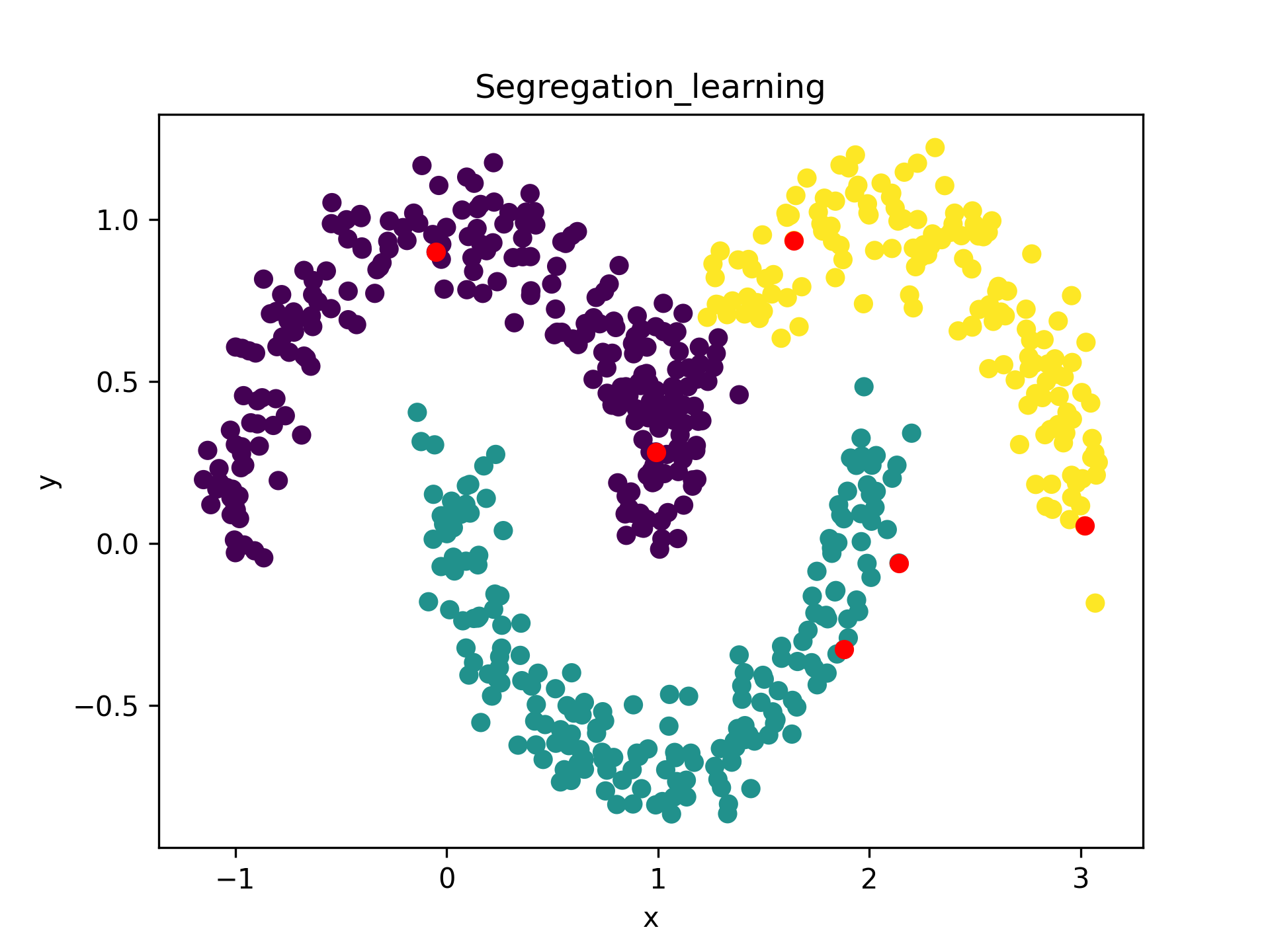}
	\includegraphics[width=.32\linewidth,valign=m]{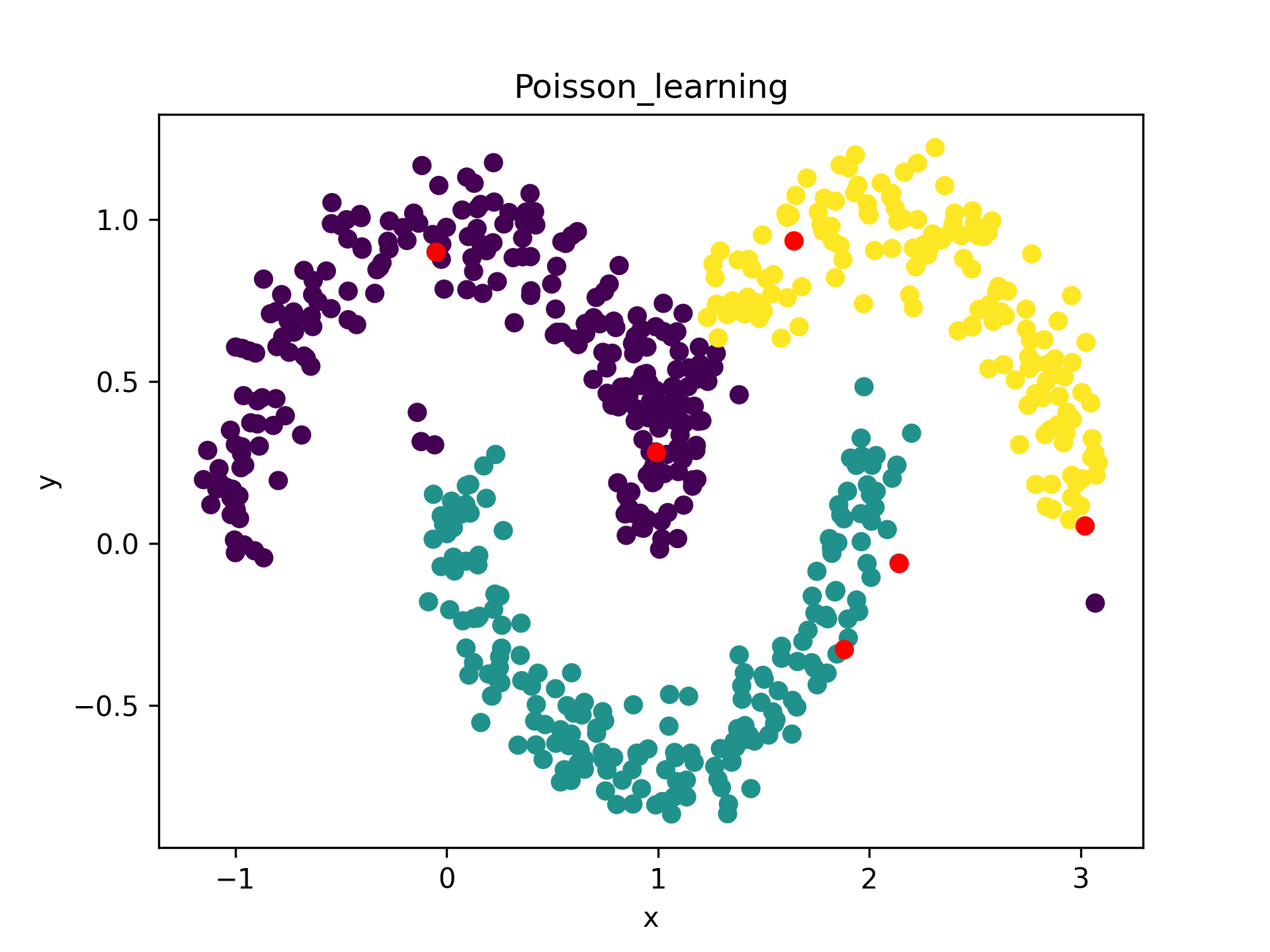}
	\includegraphics[width=.32\linewidth,valign=m]{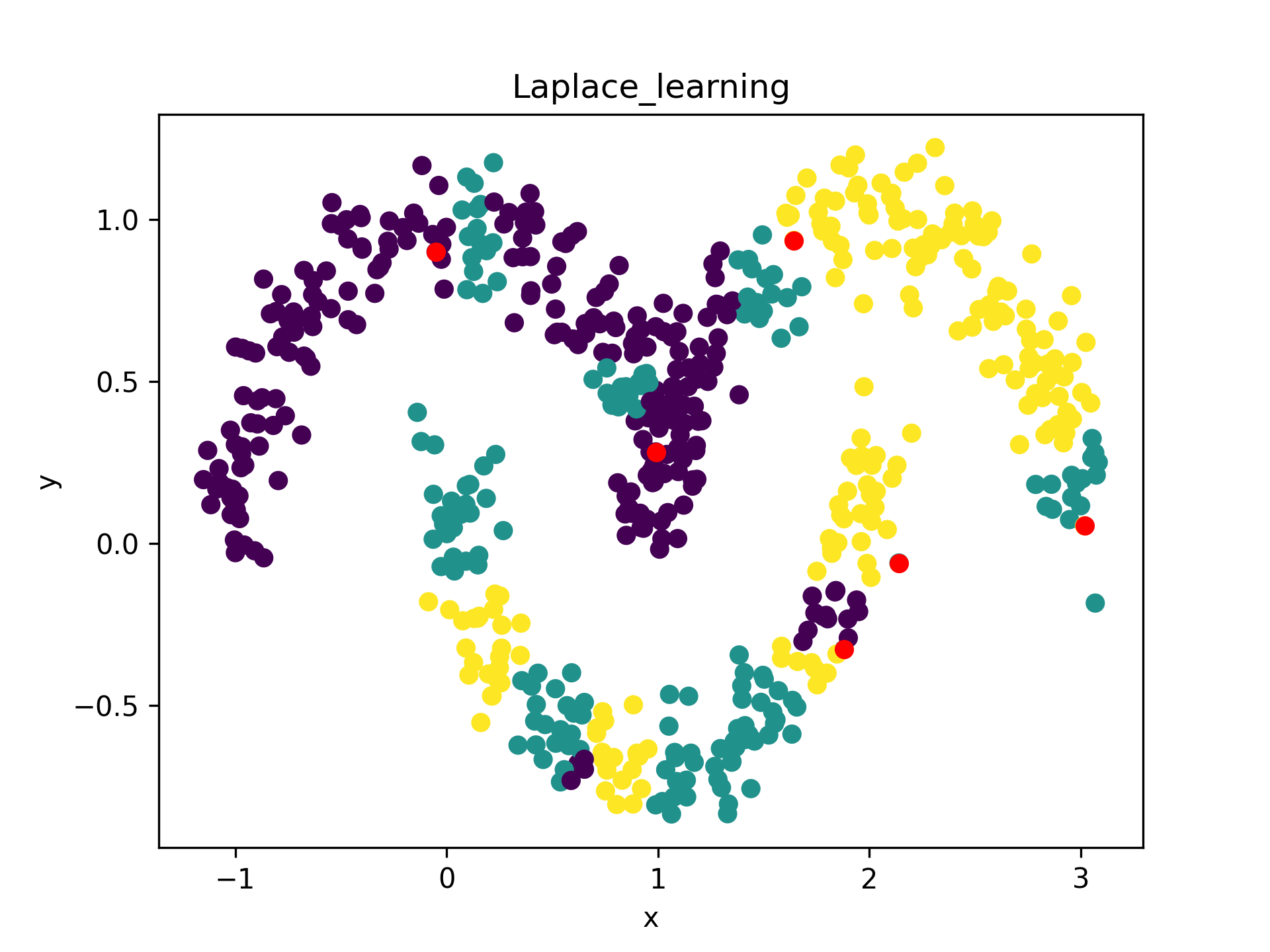}
	\caption{Comparison of Laplace, Poisson and Segregation learning algorithms for 3 classes and initial 2 labels per class.}	
	\label{fig1}
\end{figure}

\begin{figure}[!htb]
	\centering
	\includegraphics[width=.32\linewidth,valign=m]{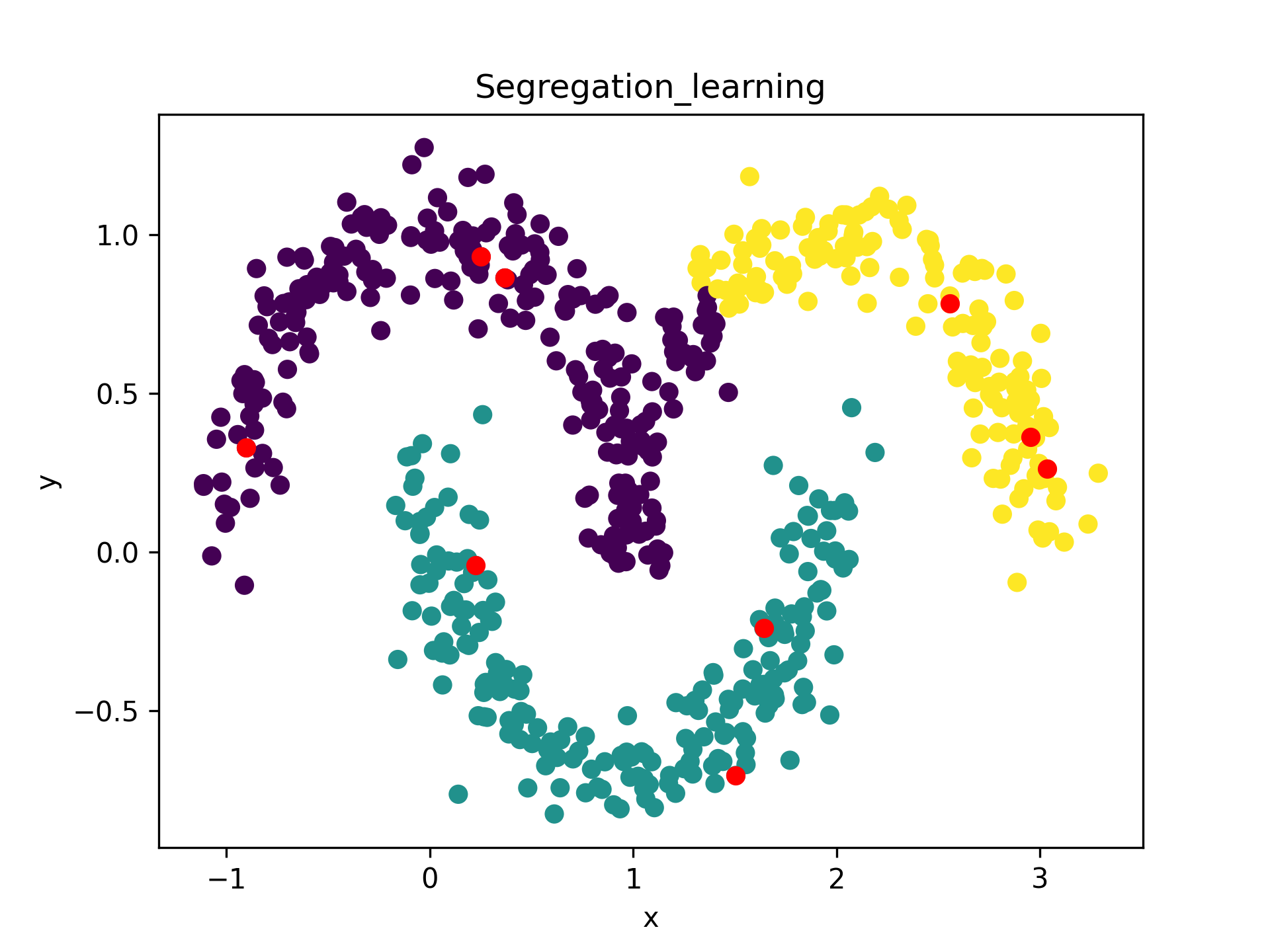}  \includegraphics[width=.32\linewidth,valign=m]{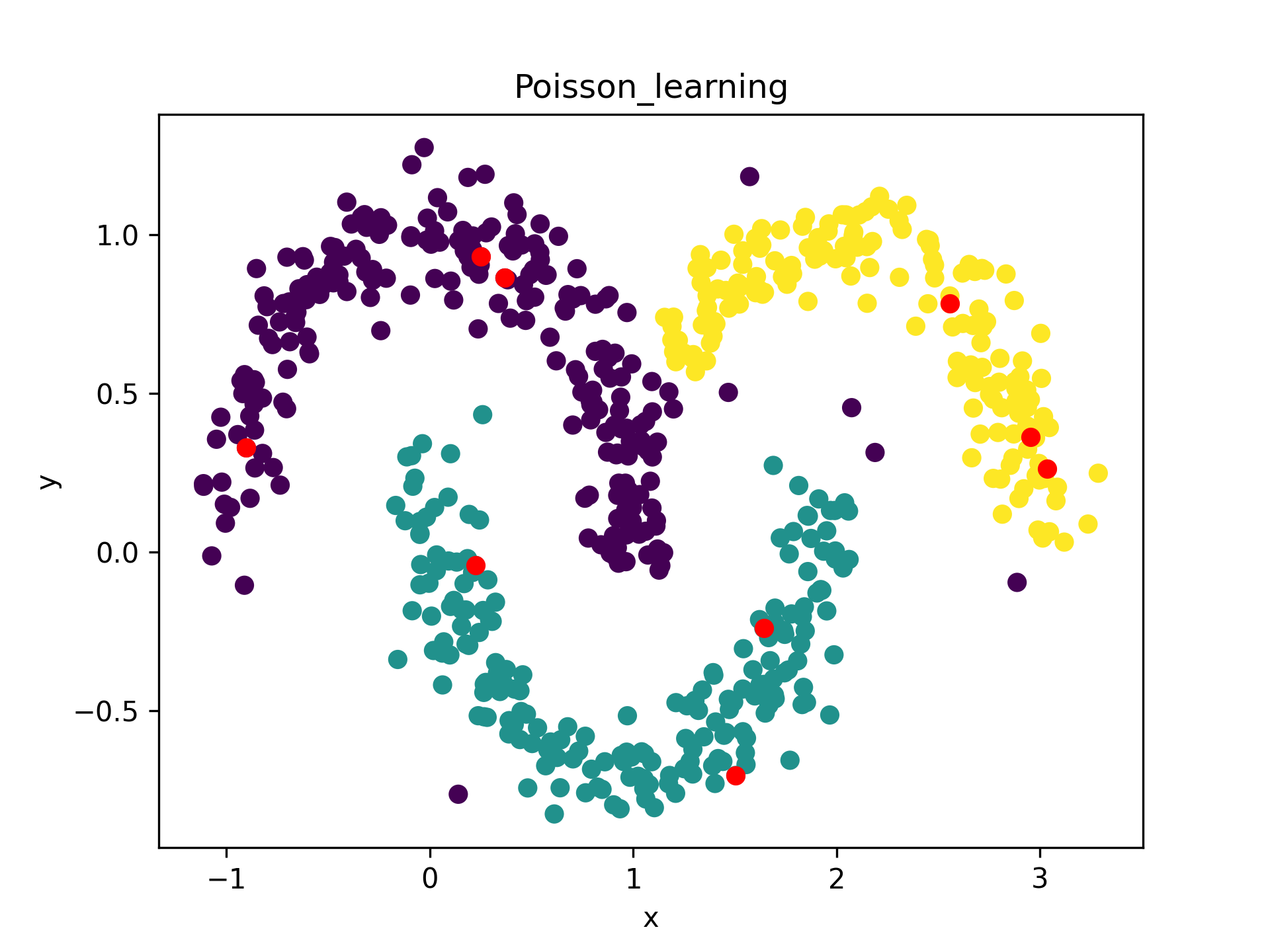}  \includegraphics[width=.32\linewidth,valign=m]{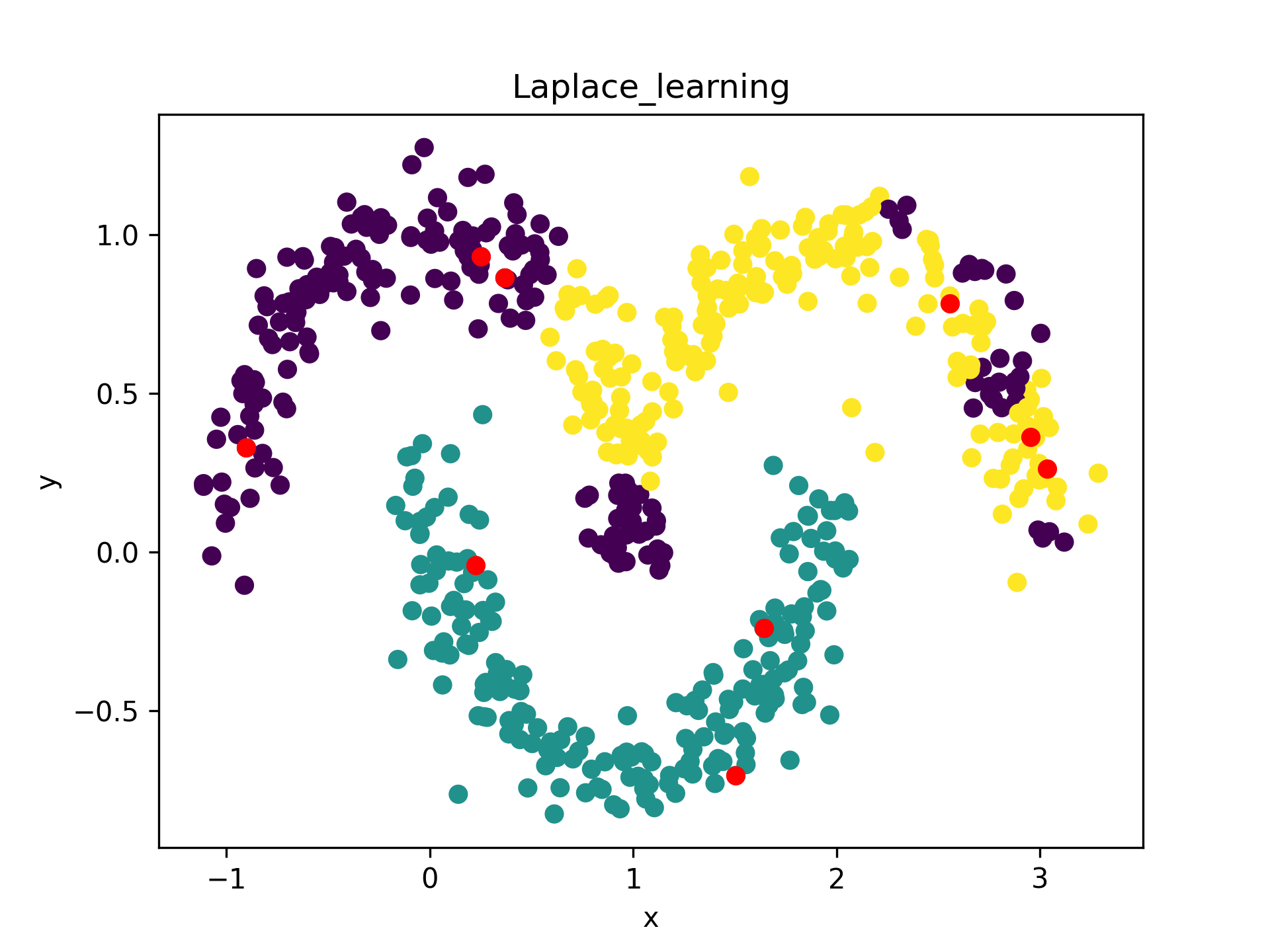}
	\caption{Comparison of Laplace, Poisson and Segregation learning algorithms for 3 classes and initial 3 labels per class.}
	\label{fig2}	
\end{figure}
\begin{figure}[!htb]\label{fig3}
	\includegraphics[width=.32\linewidth,valign=m]{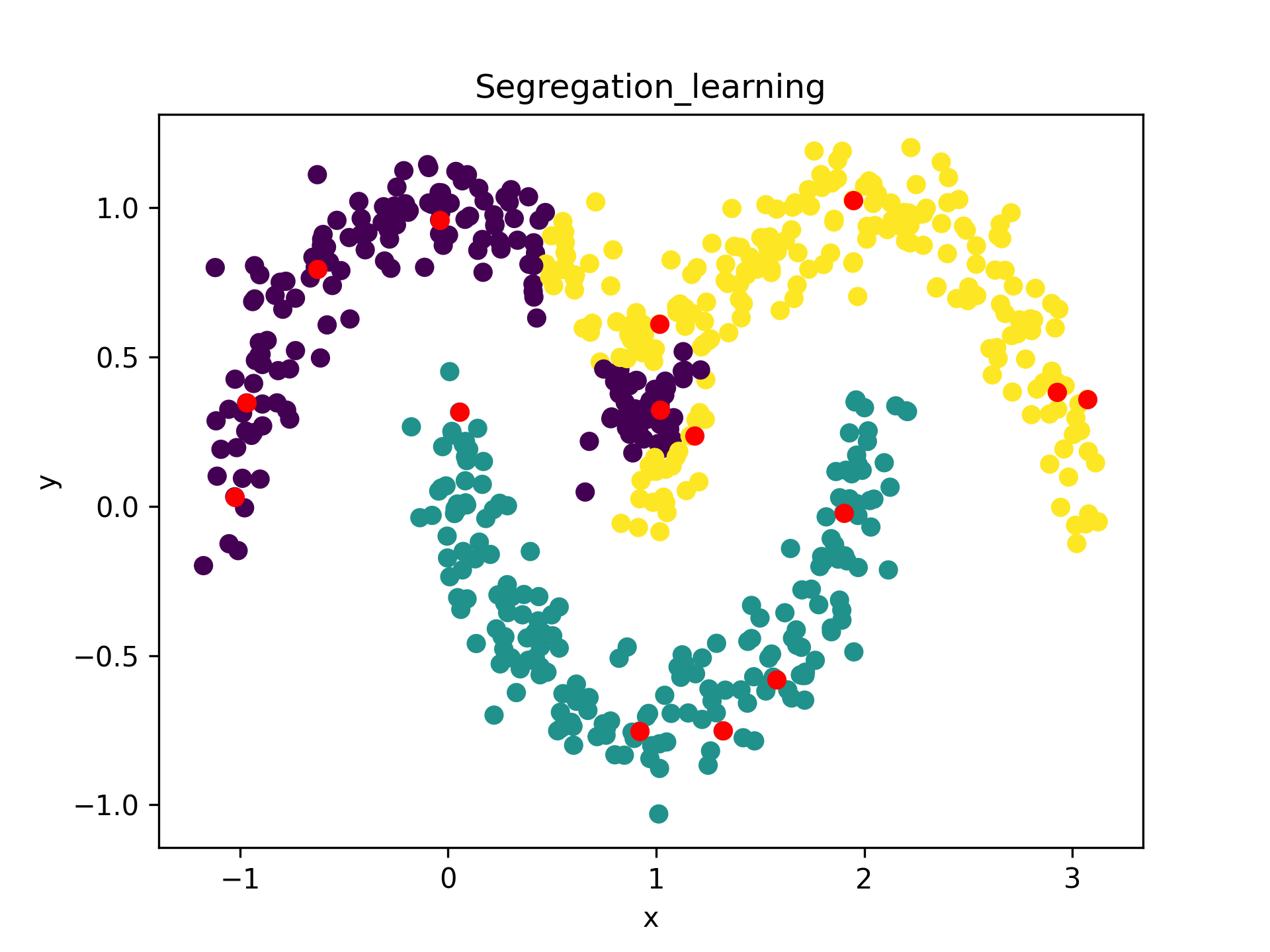}  \includegraphics[width=.32\linewidth,valign=m]{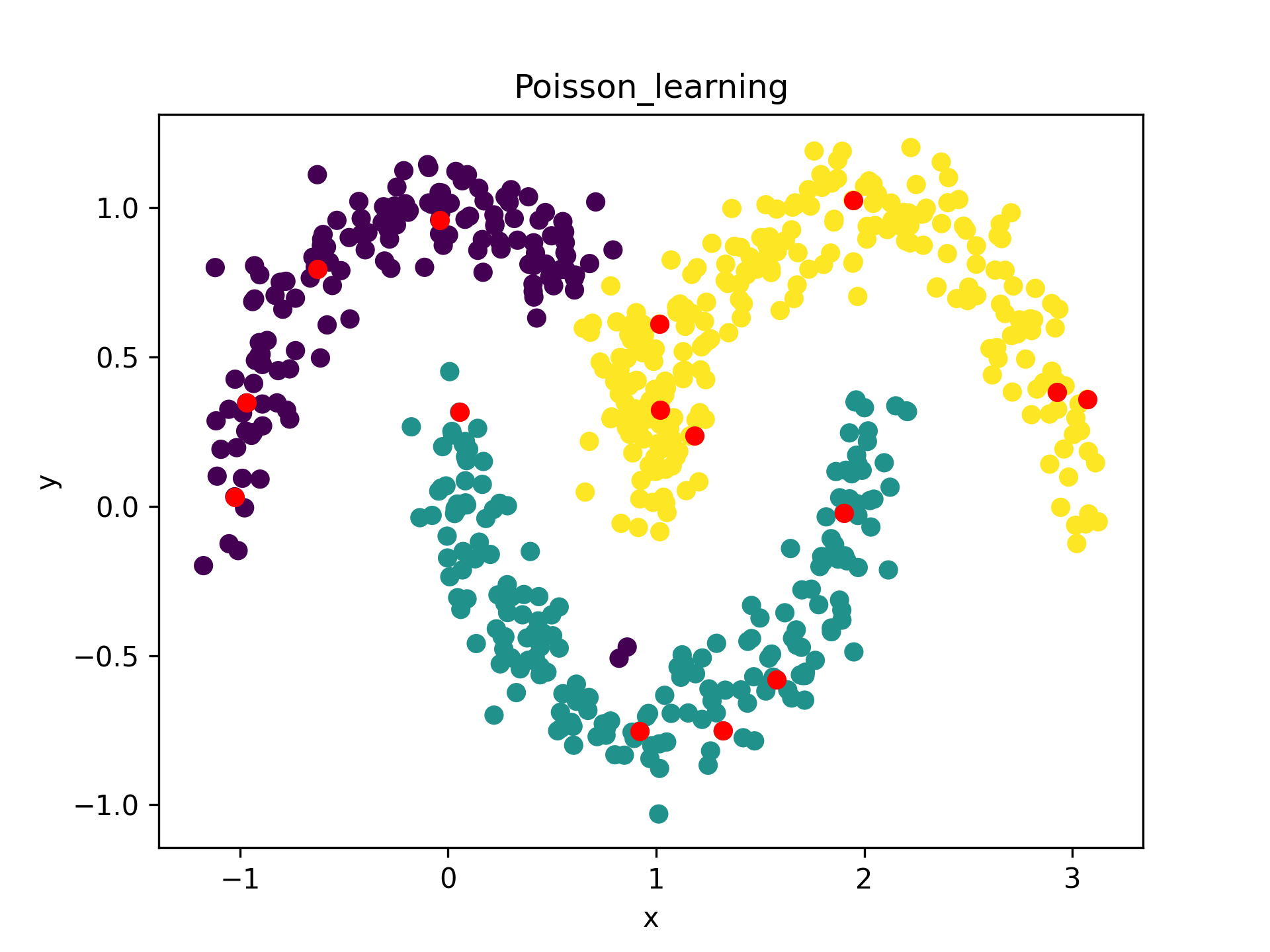}  \includegraphics[width=.32\linewidth,valign=m]{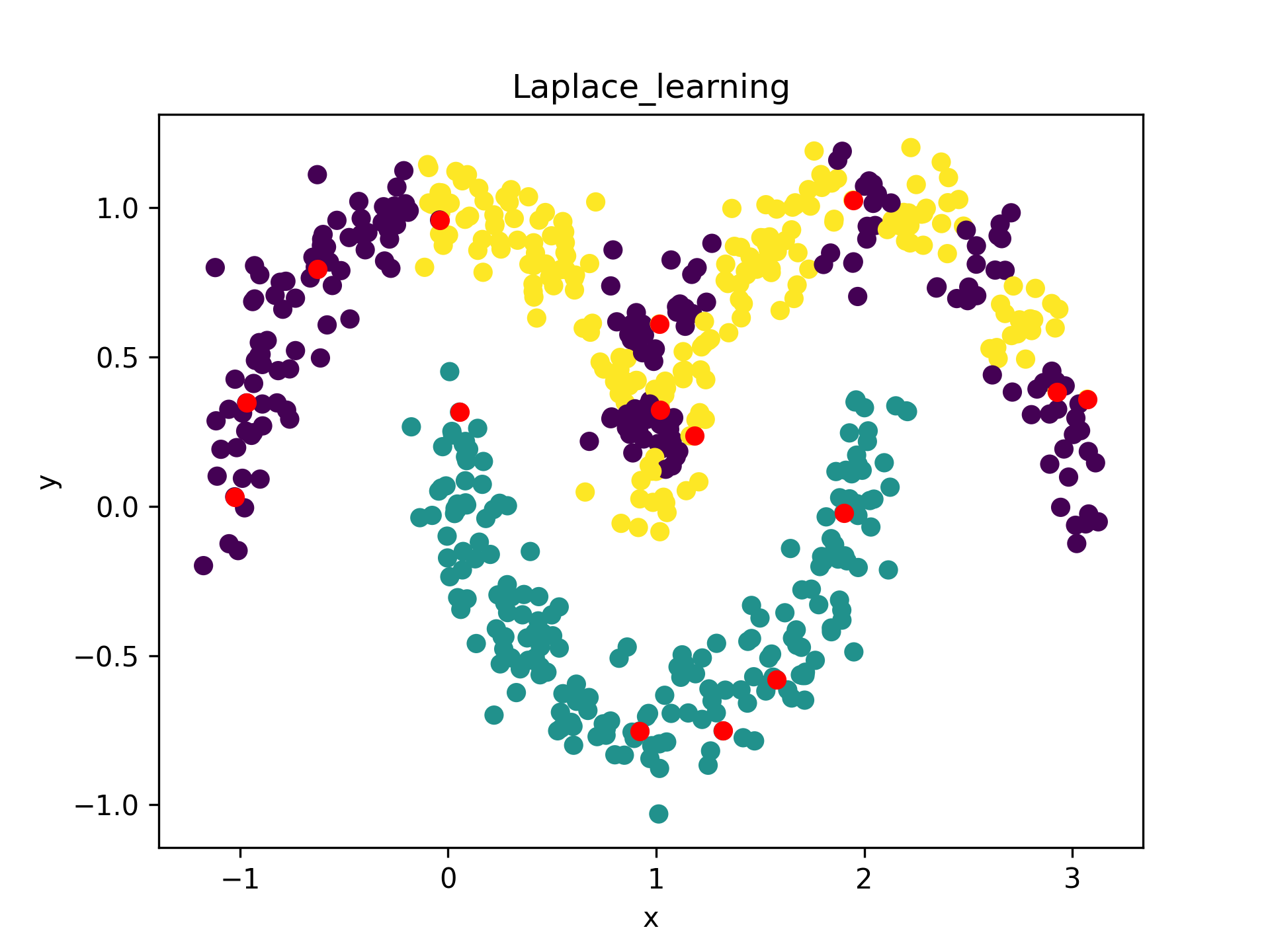}
	\caption{Comparison of Laplace, Poisson and Segregation learning algorithms for 3 classes and initial 5 labels per class.}	
\end{figure}
\begin{figure}[!htb]\label{fig4}
	\includegraphics[width=.32\linewidth,valign=m]{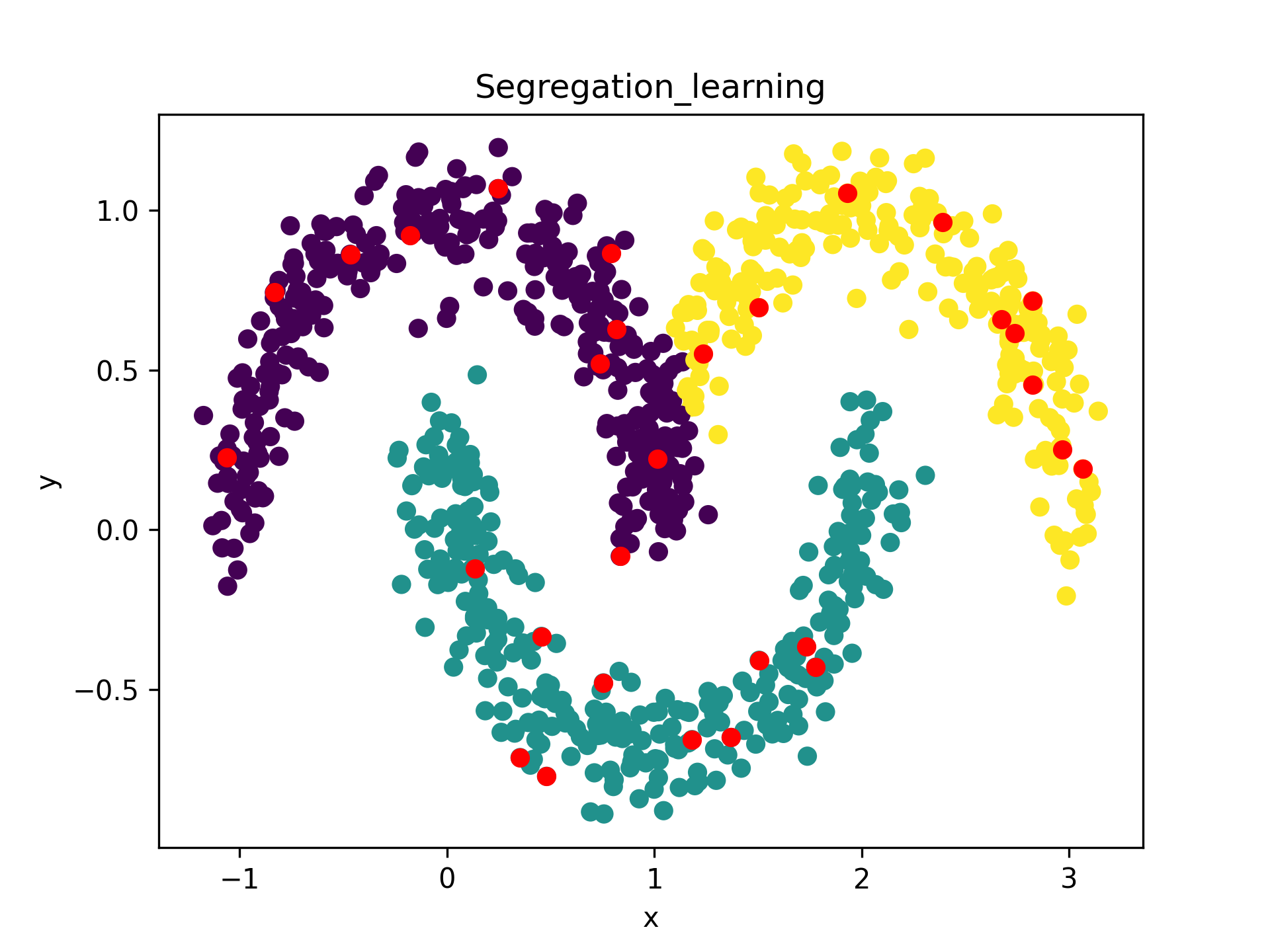}  \includegraphics[width=.32\linewidth,valign=m]{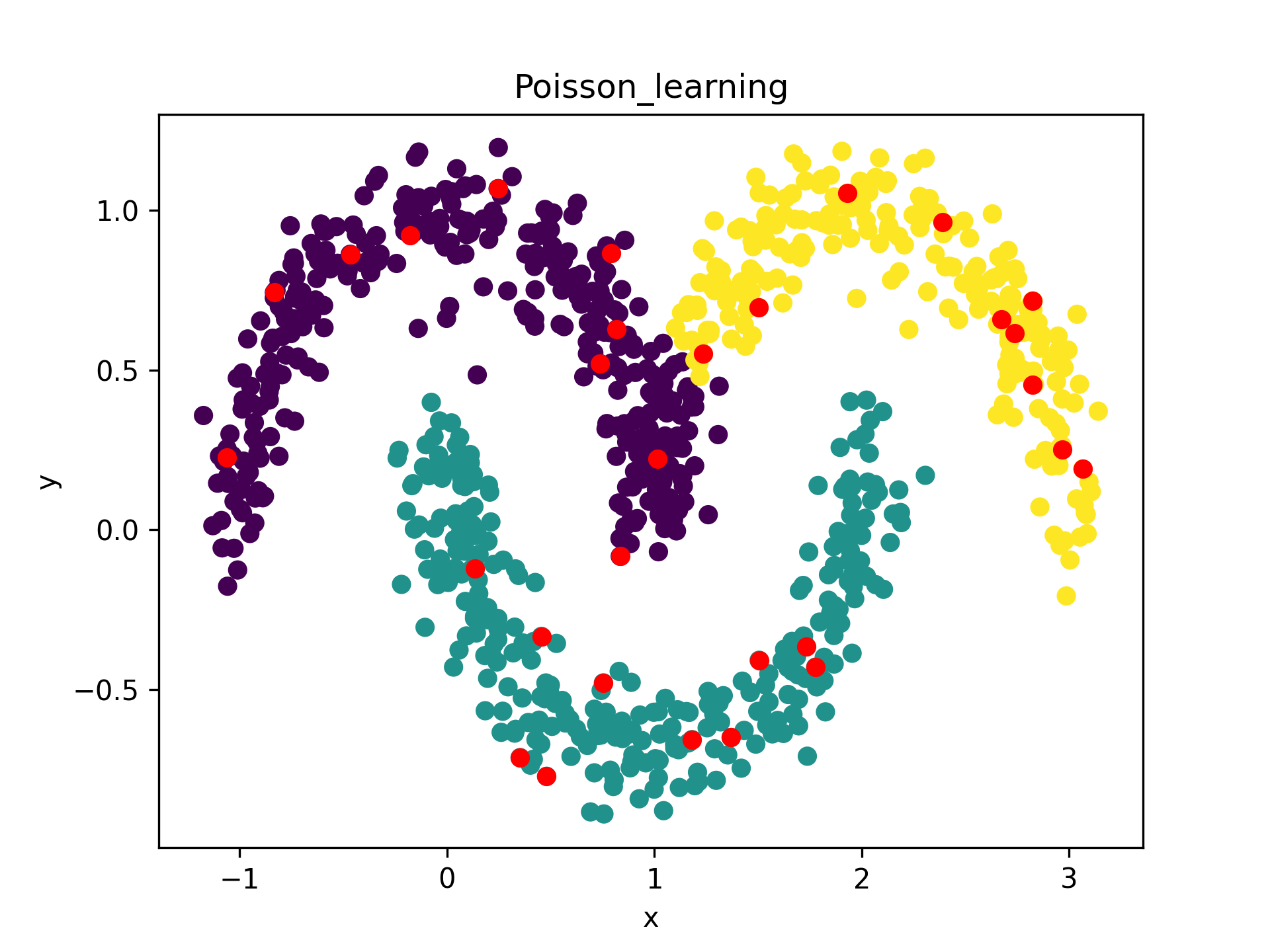}  \includegraphics[width=.32\linewidth,valign=m]{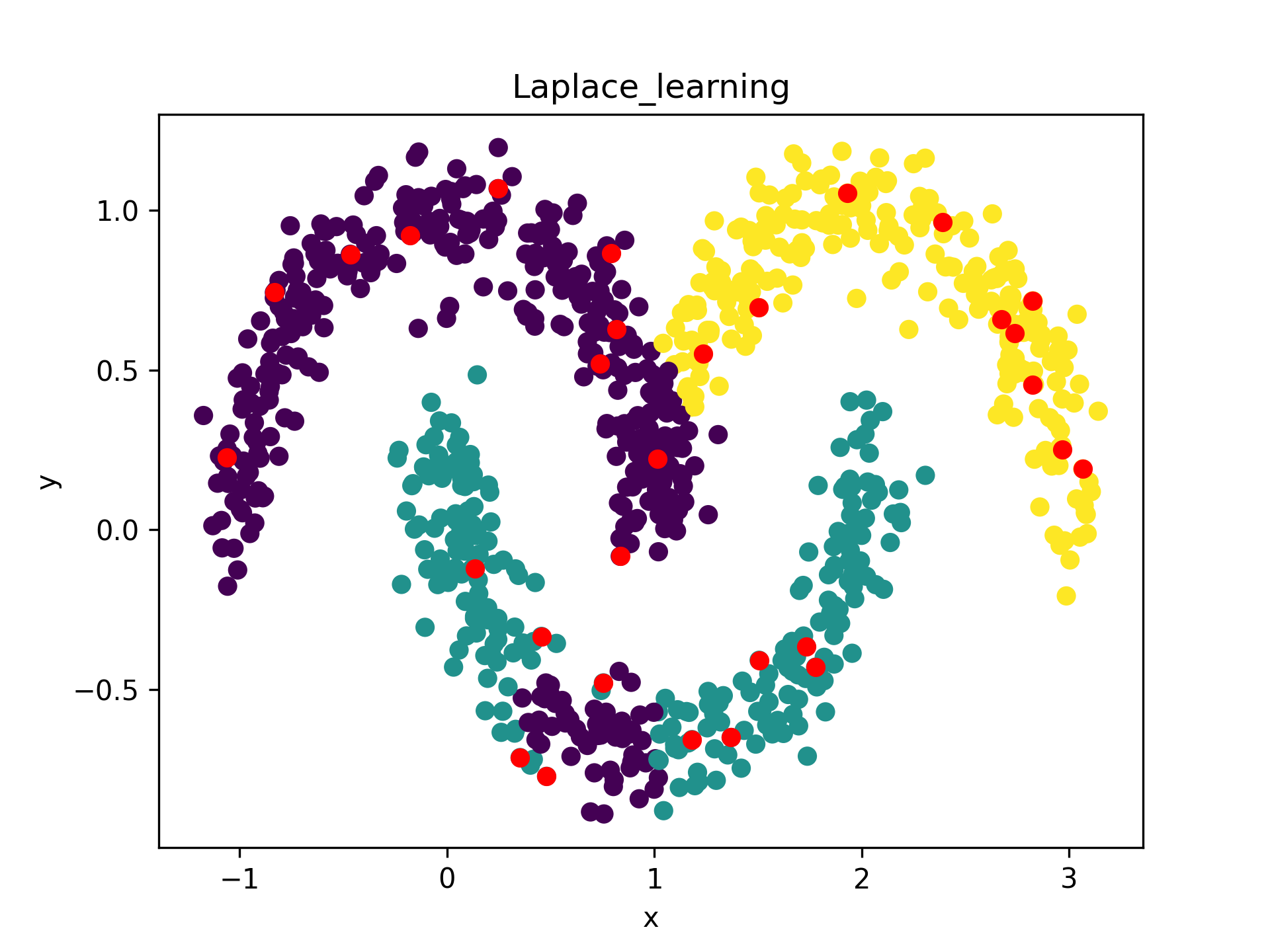}
	\caption{Comparison of Laplace, Poisson and Segregation learning algorithms for 3 classes and initial 10 labels per class.}			
\end{figure}
\begin{figure}[!htb]\label{fig5}
	\includegraphics[width=.32\linewidth,valign=m]{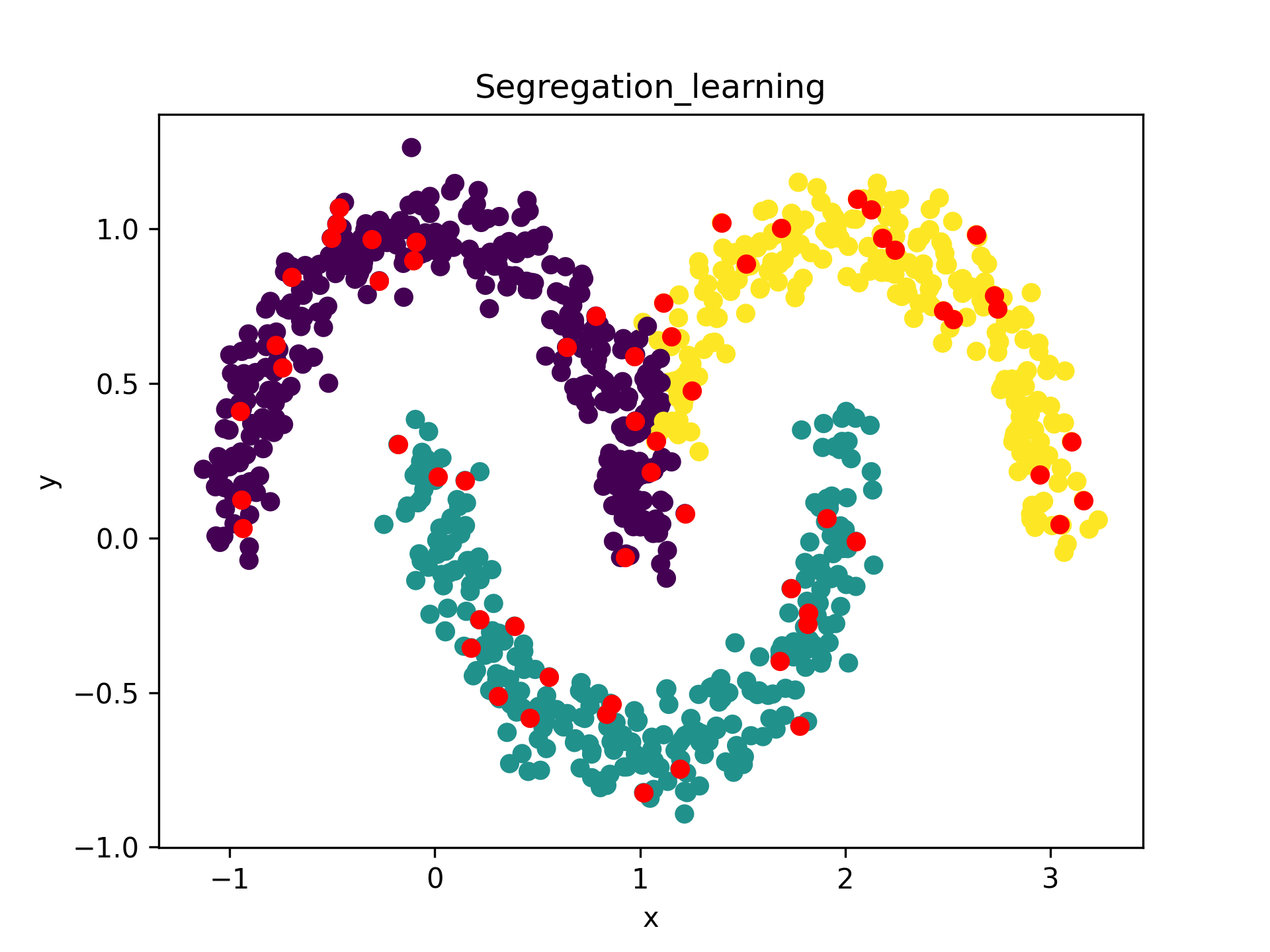}  \includegraphics[width=.32\linewidth,valign=m]{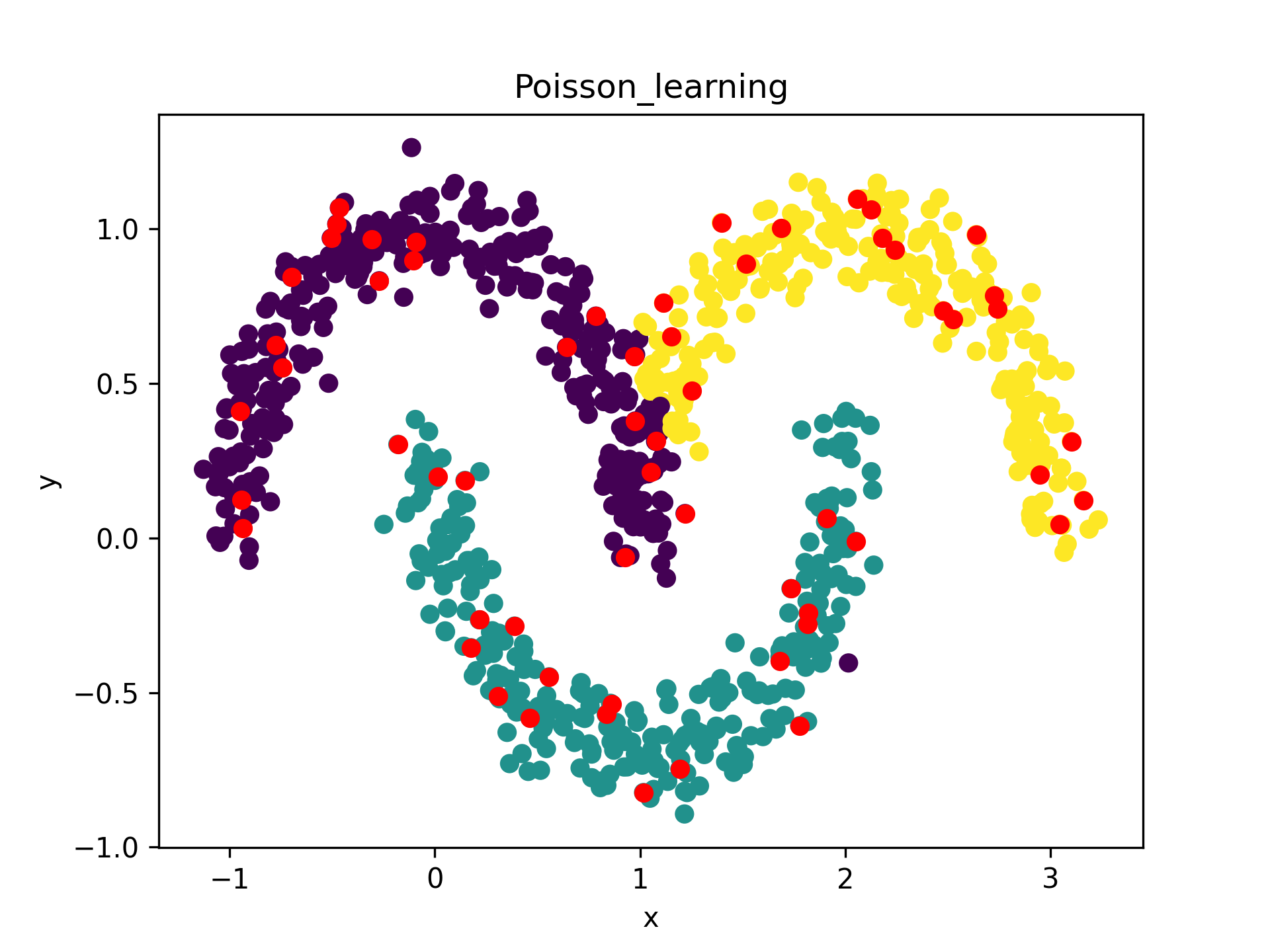}  \includegraphics[width=.32\linewidth,valign=m]{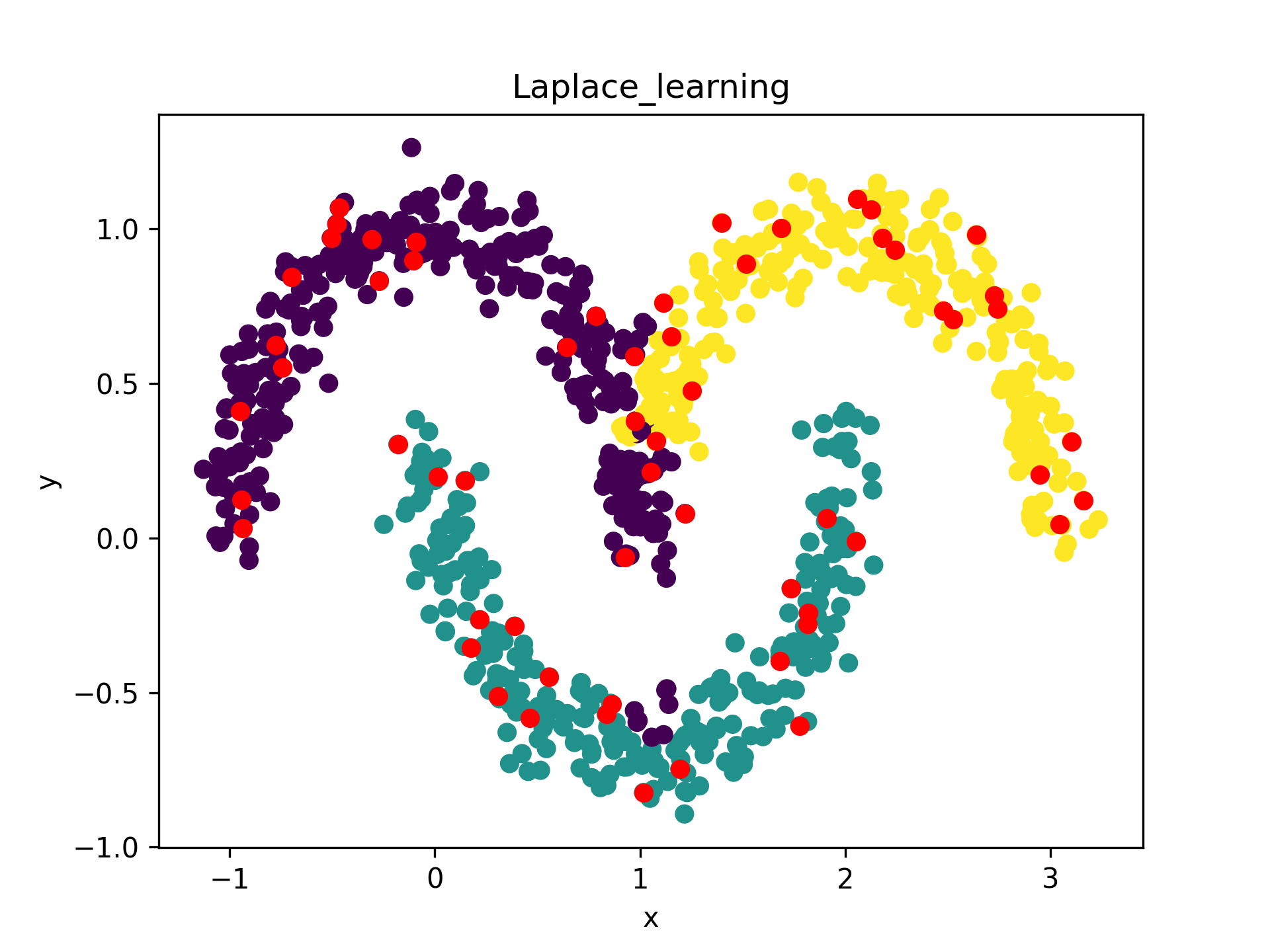}
	\caption{Comparison of Laplace, Poisson and Segregation learning algorithms for 3 classes and initial 20 labels per class.}			
\end{figure}

\begin{figure}[!htb]\label{fig6}
	\includegraphics[width=.32\linewidth,valign=m]{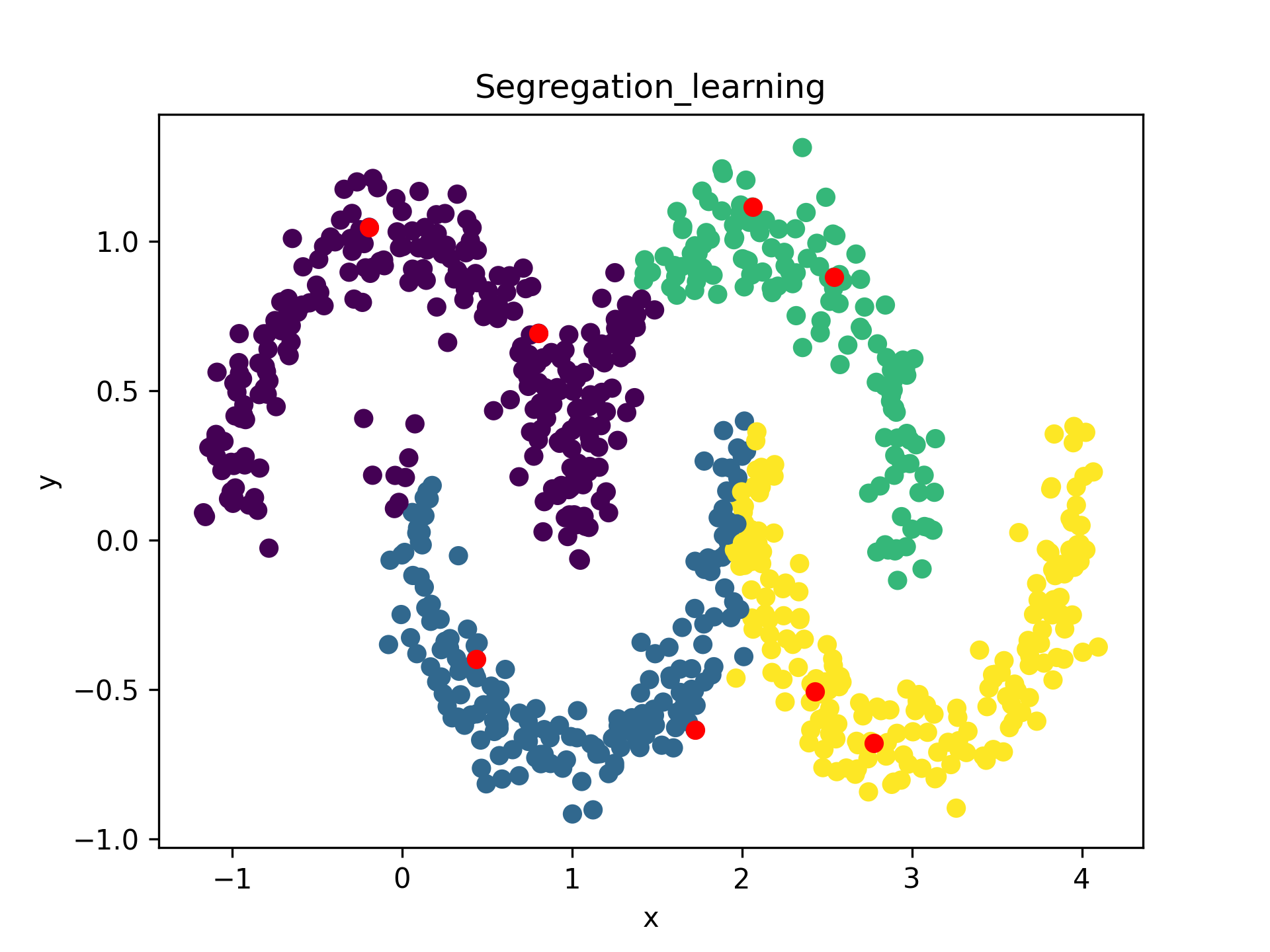}  \includegraphics[width=.32\linewidth,valign=m]{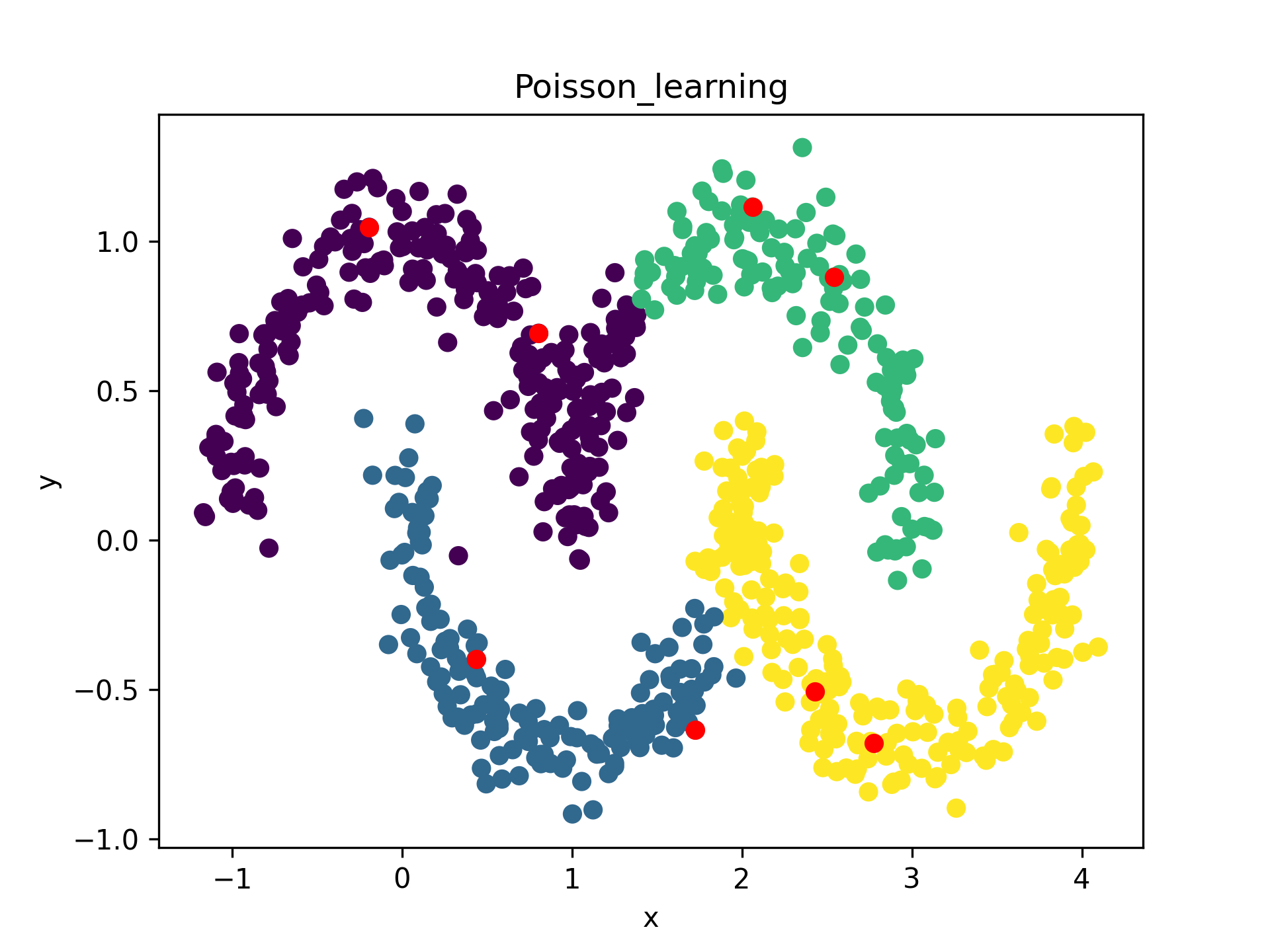}  \includegraphics[width=.32\linewidth,valign=m]{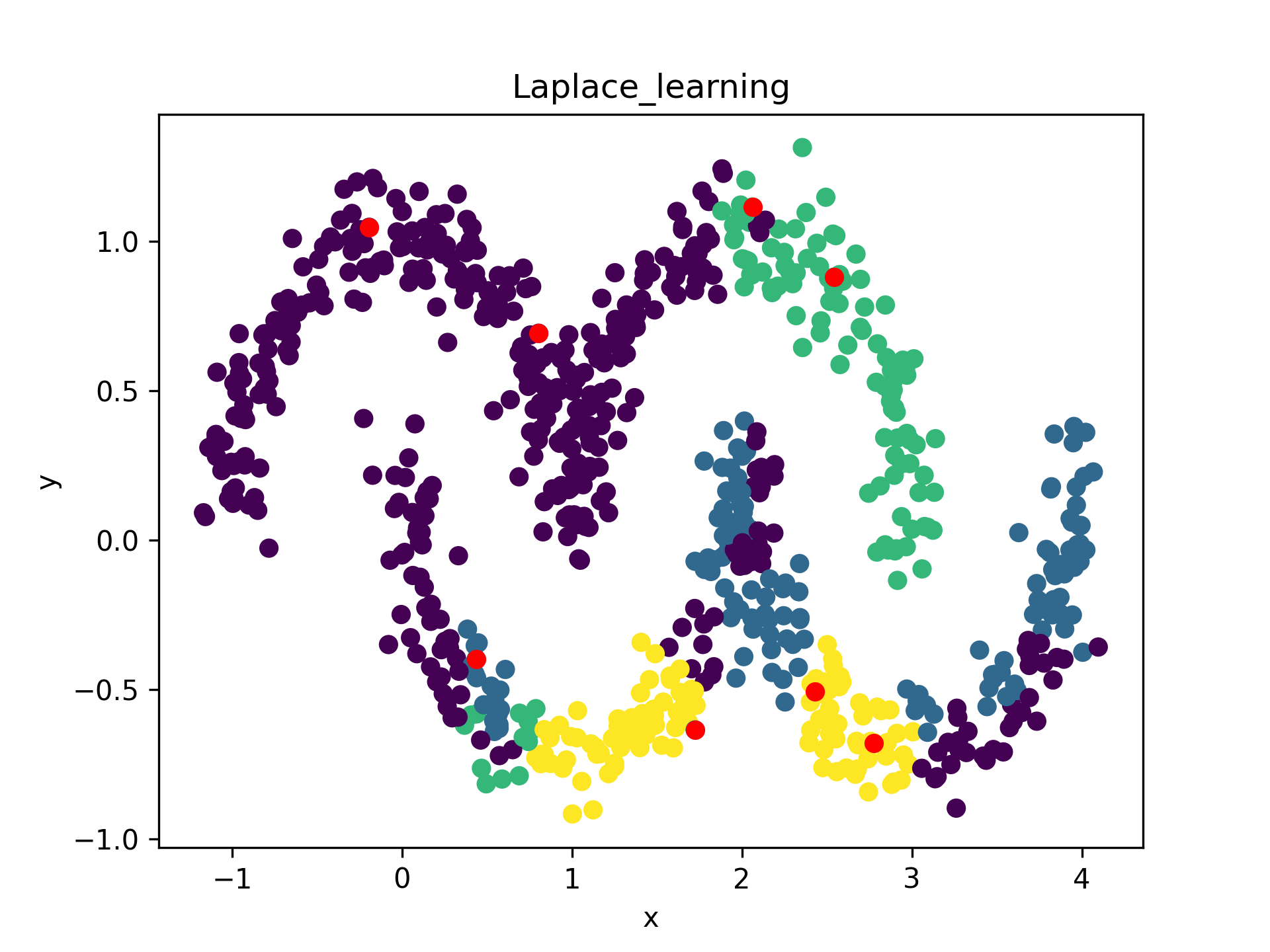}
	\caption{Comparison of Laplace, Poisson and Segregation learning algorithms for 4 classes and  initial 2 labels per class.}			
\end{figure}
\begin{figure}[!htb]\label{fig7}
	\includegraphics[width=.32\linewidth,valign=m]{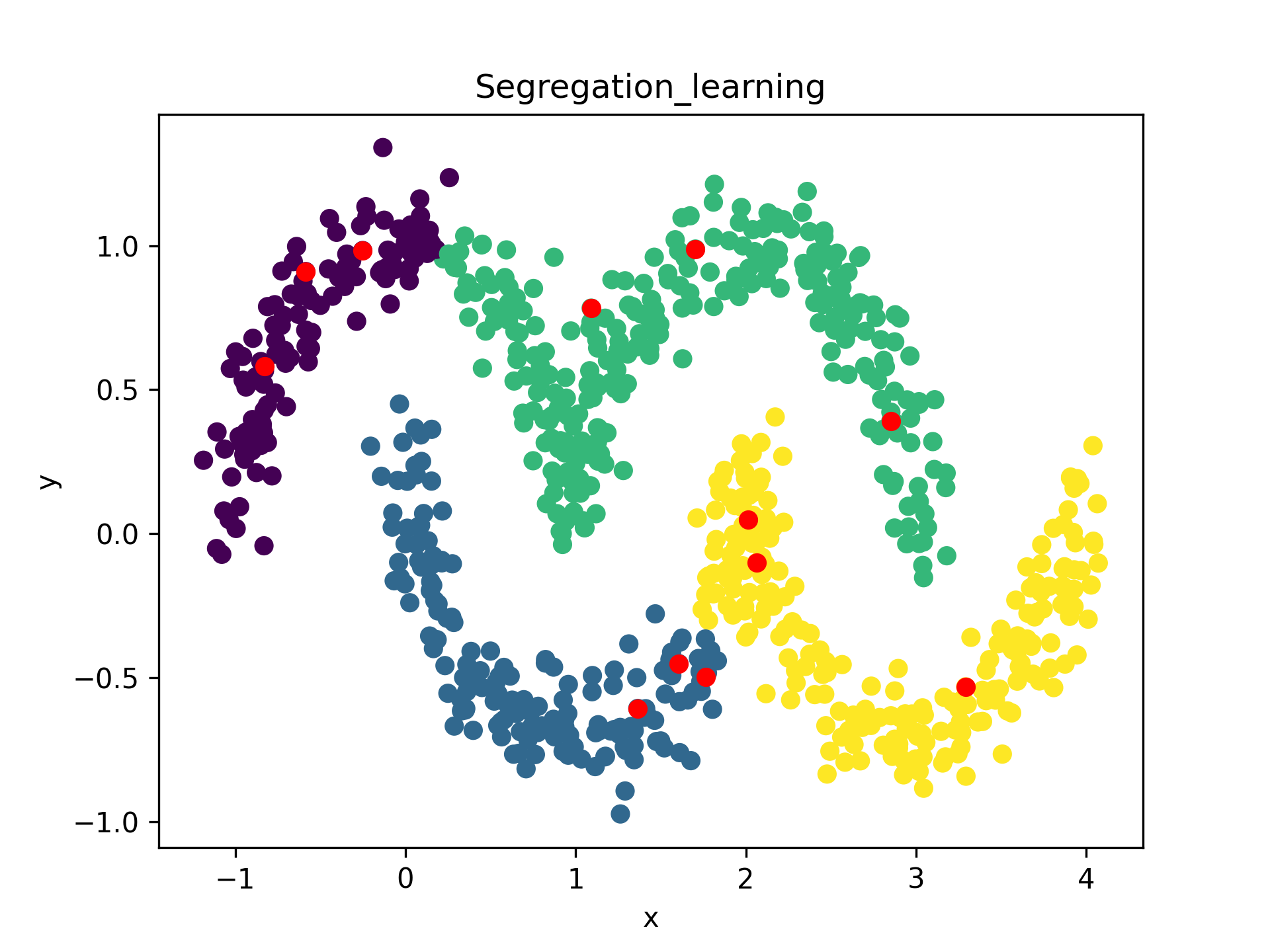}  \includegraphics[width=.32\linewidth,valign=m]{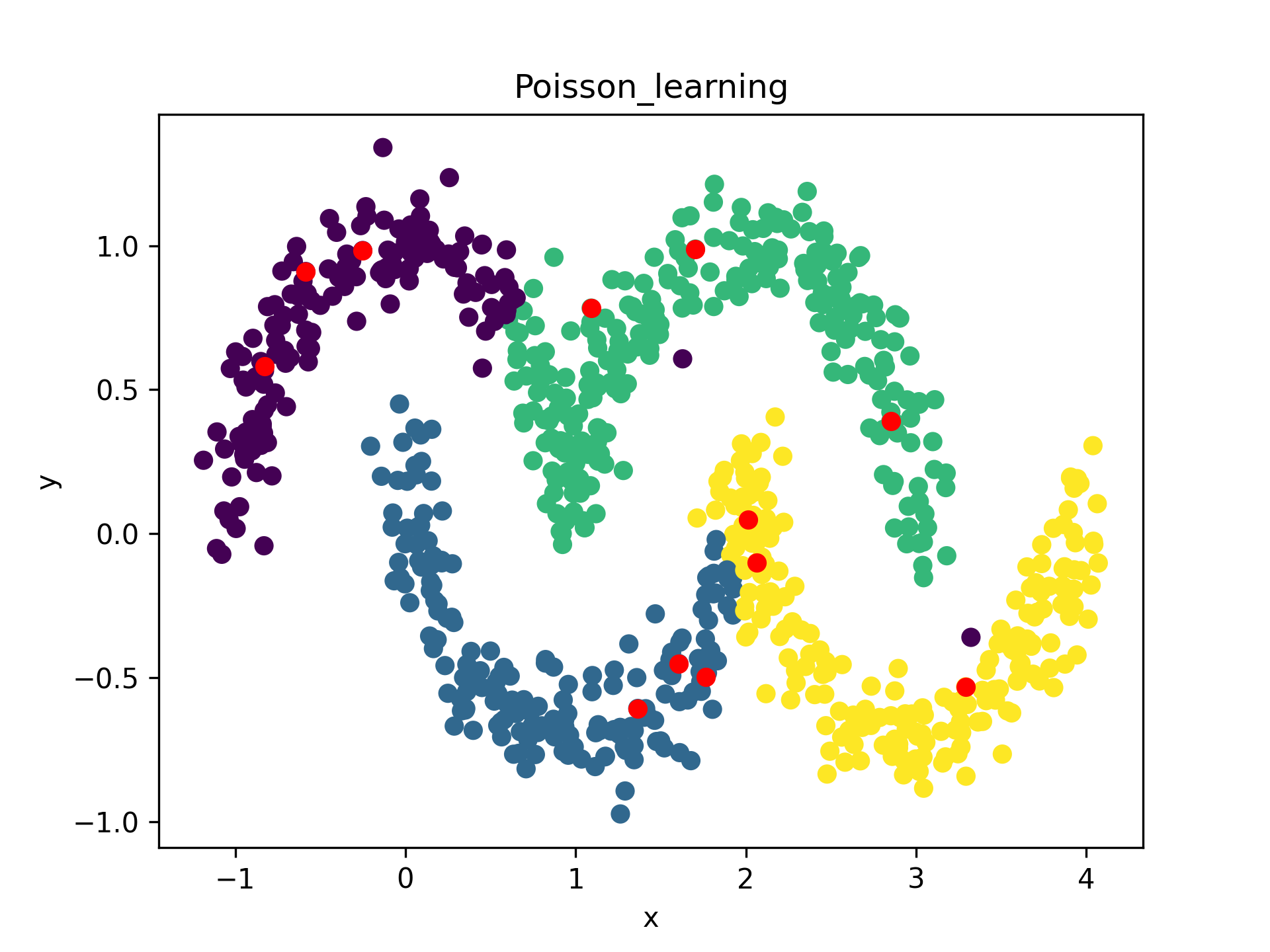}  \includegraphics[width=.32\linewidth,valign=m]{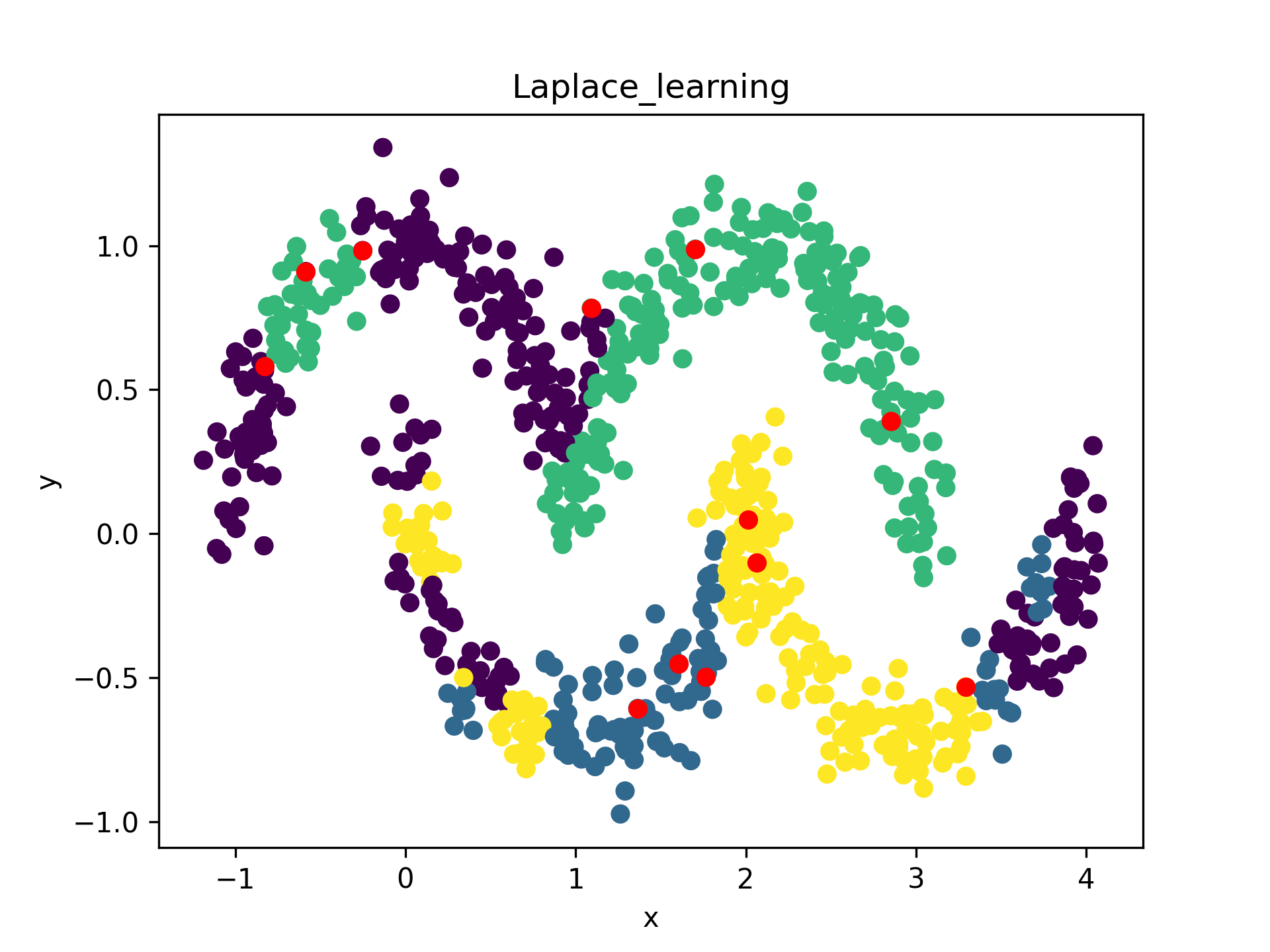}
	\caption{Comparison of Laplace, Poisson and Segregation learning algorithms for 4 classes and  initial 3 labels per class.}			
\end{figure}
\begin{figure}[!htb]\label{fig8}
	\includegraphics[width=.32\linewidth,valign=m]{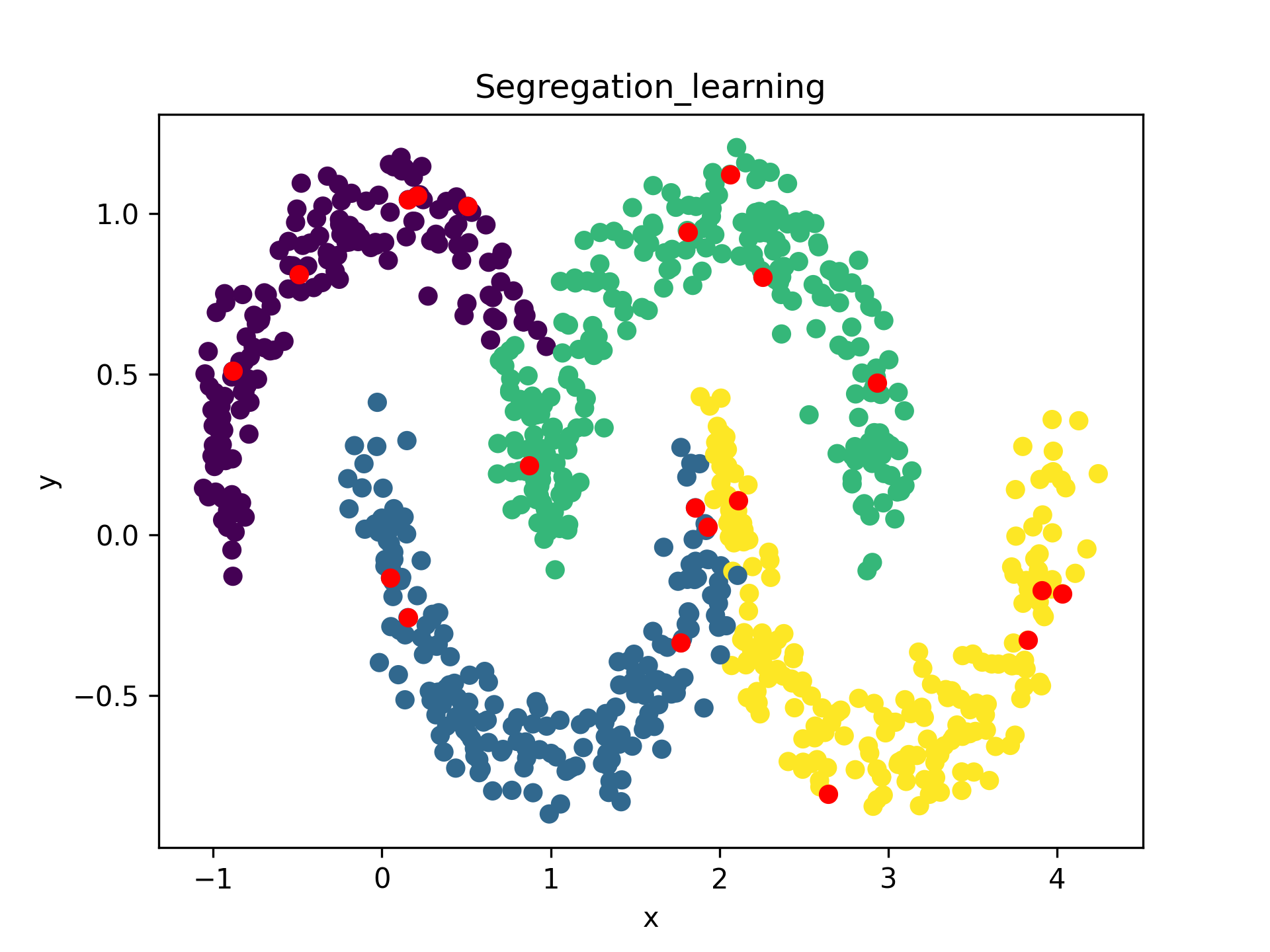}  \includegraphics[width=.32\linewidth,valign=m]{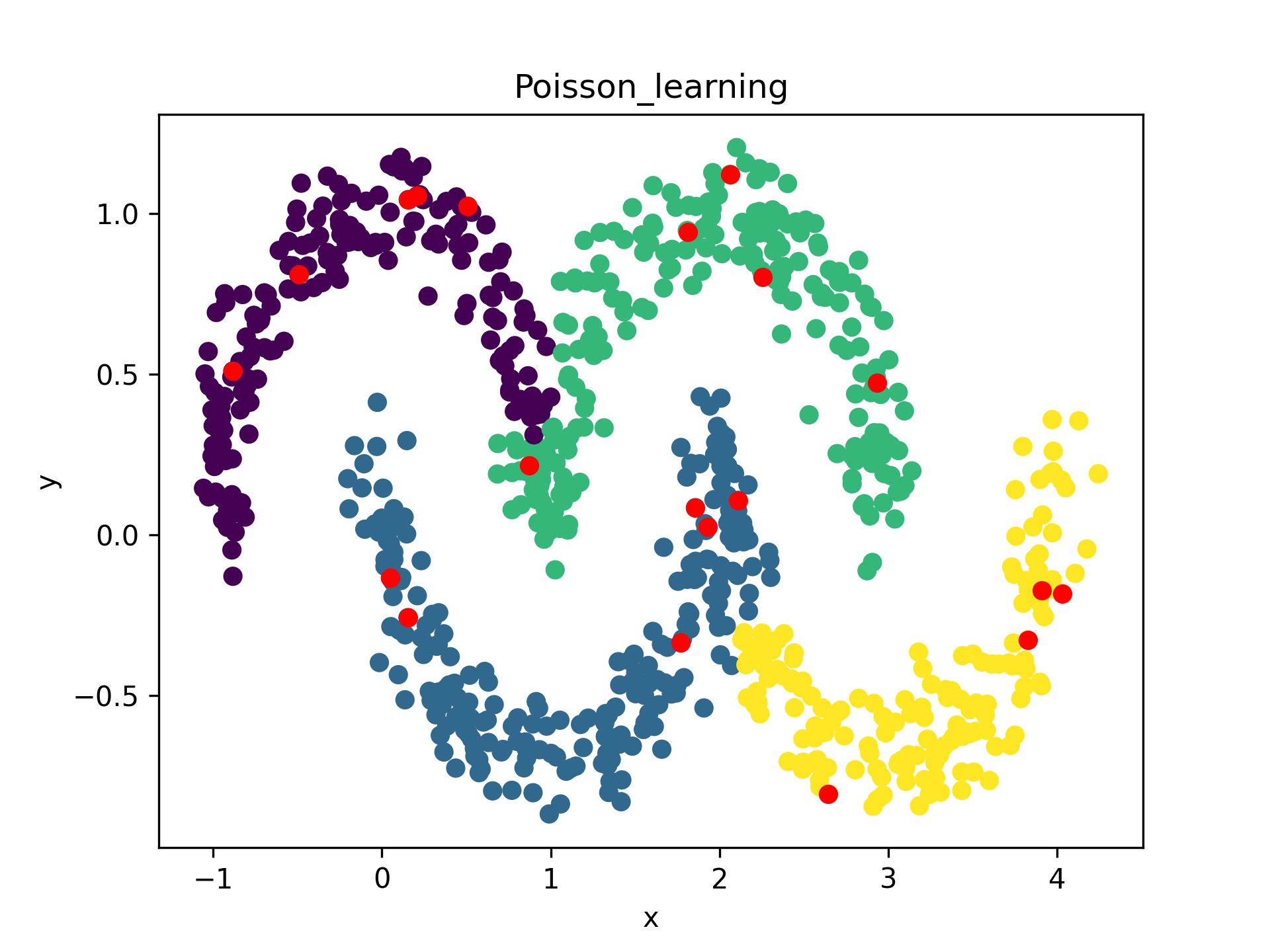}  \includegraphics[width=.32\linewidth,valign=m]{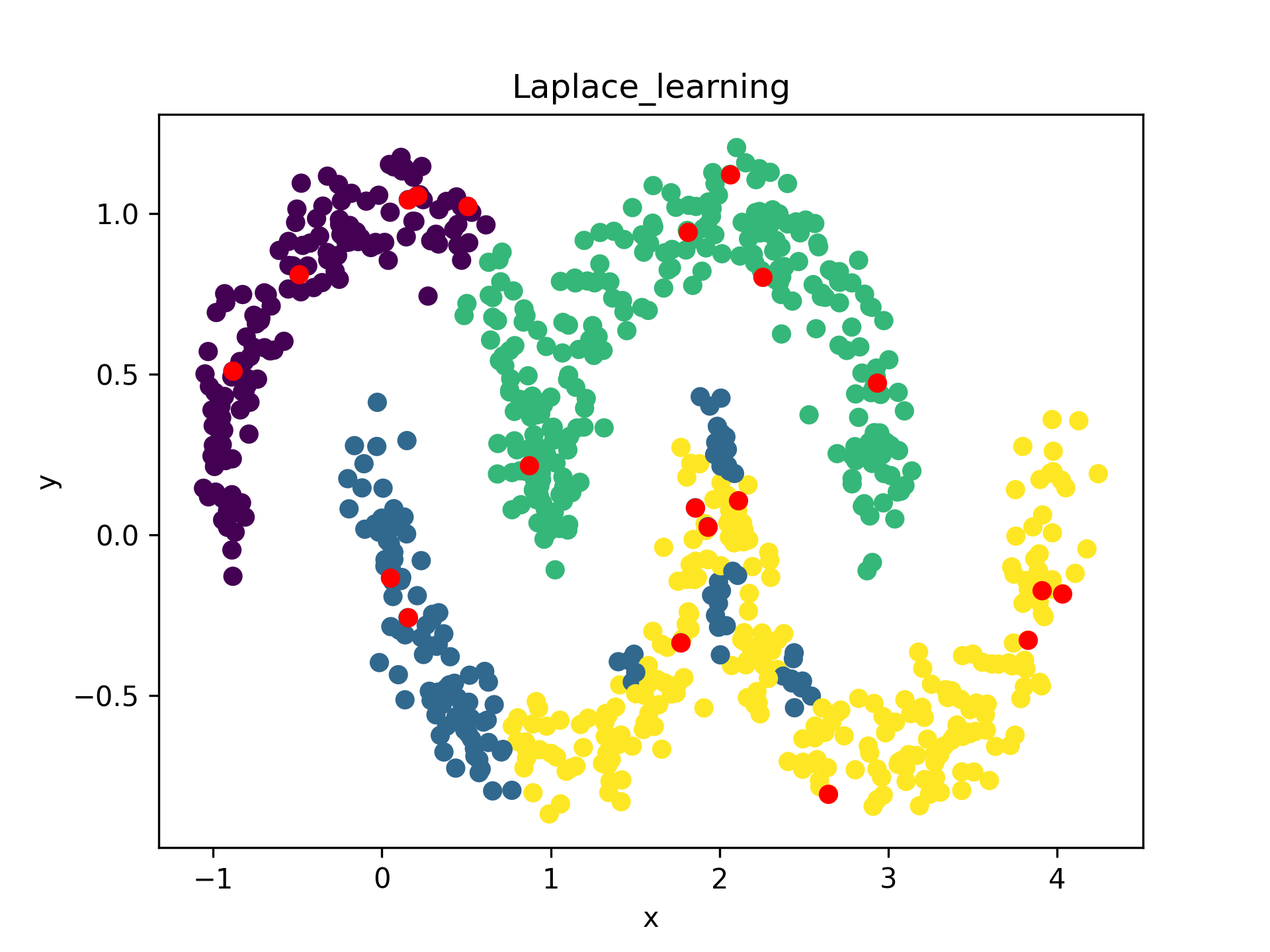}
	\caption{Comparison of Laplace, Poisson and Segregation learning algorithms for 4 classes and  initial 5 labels per class.}			
\end{figure}
\begin{figure}[!htb]\label{fig9}
	\includegraphics[width=.32\linewidth,valign=m]{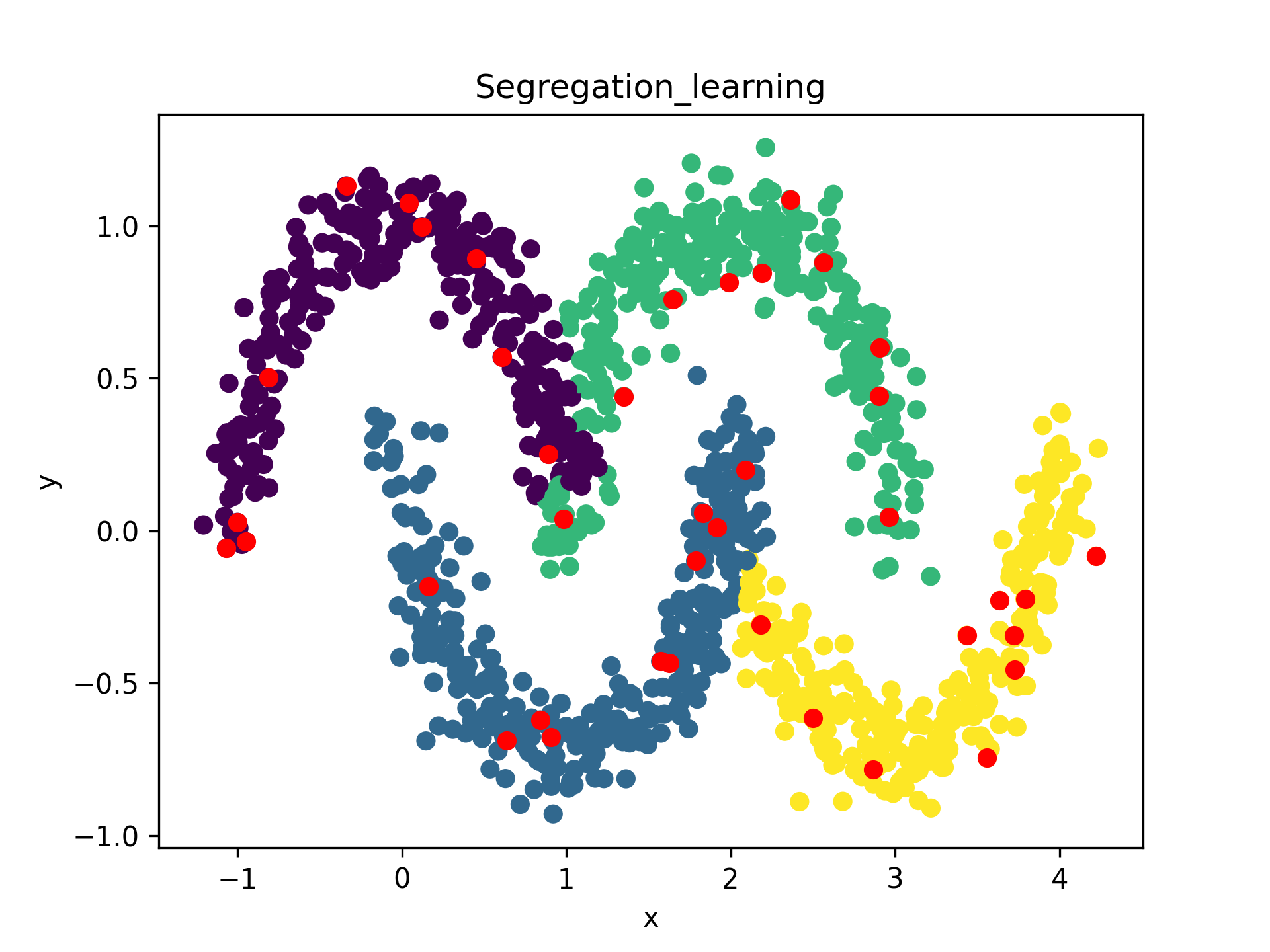}  \includegraphics[width=.32\linewidth,valign=m]{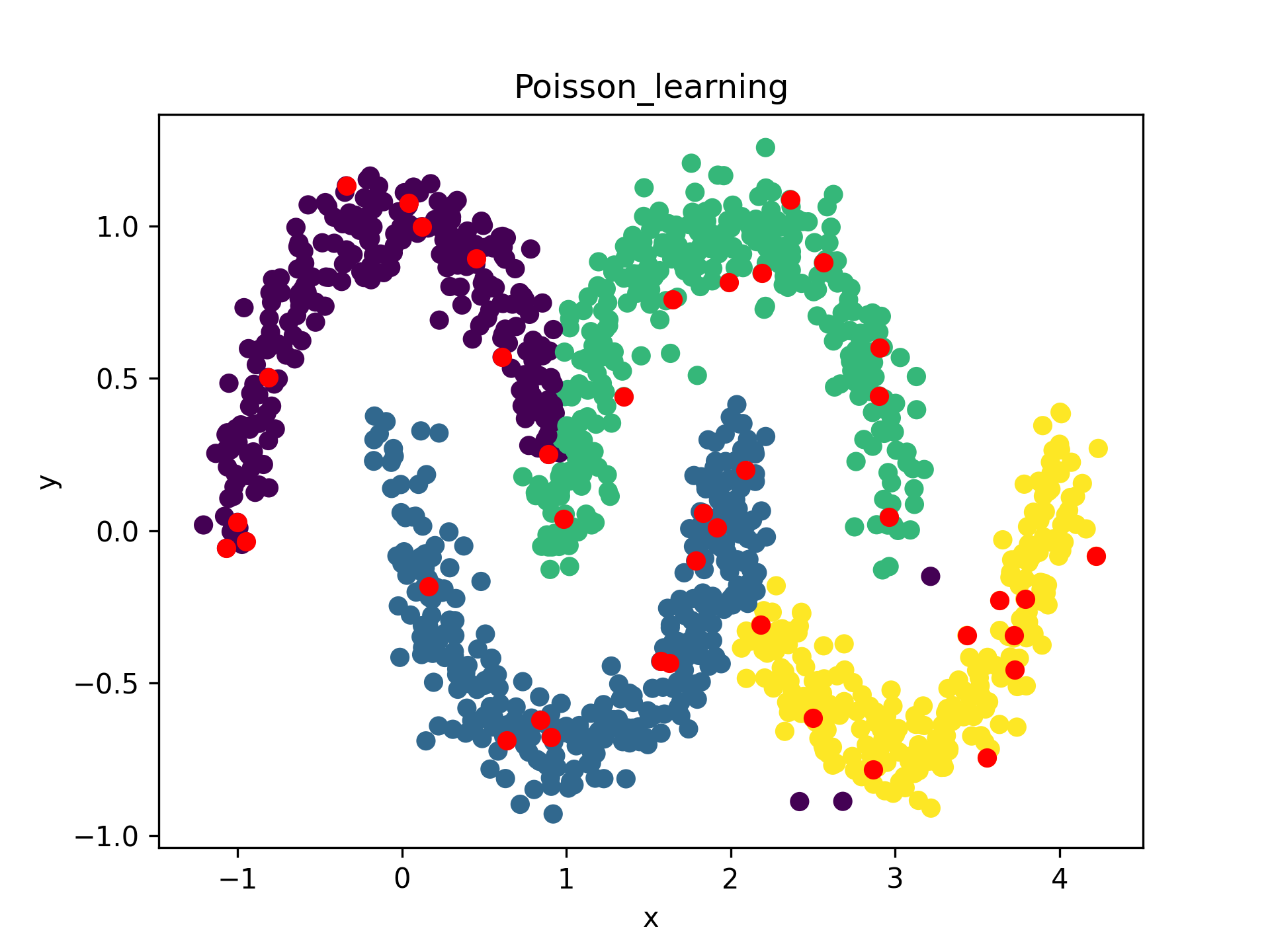}  \includegraphics[width=.32\linewidth,valign=m]{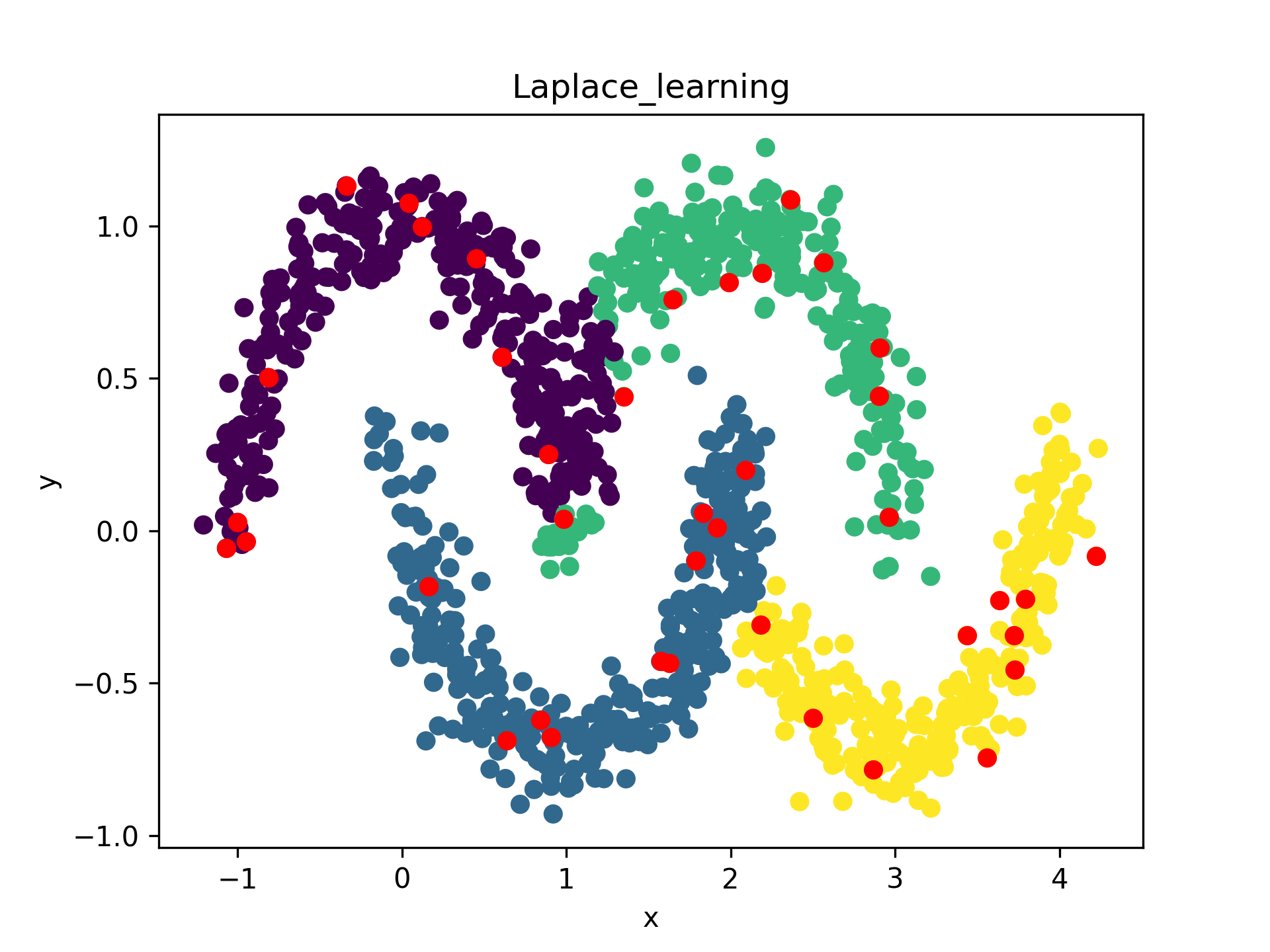}
	\caption{Comparison of Laplace, Poisson and Segregation learning algorithms for 4 classes and  initial 10 labels per class.}			
\end{figure}
\begin{figure}[!htb]\label{fig10}
	\includegraphics[width=.32\linewidth,valign=m]{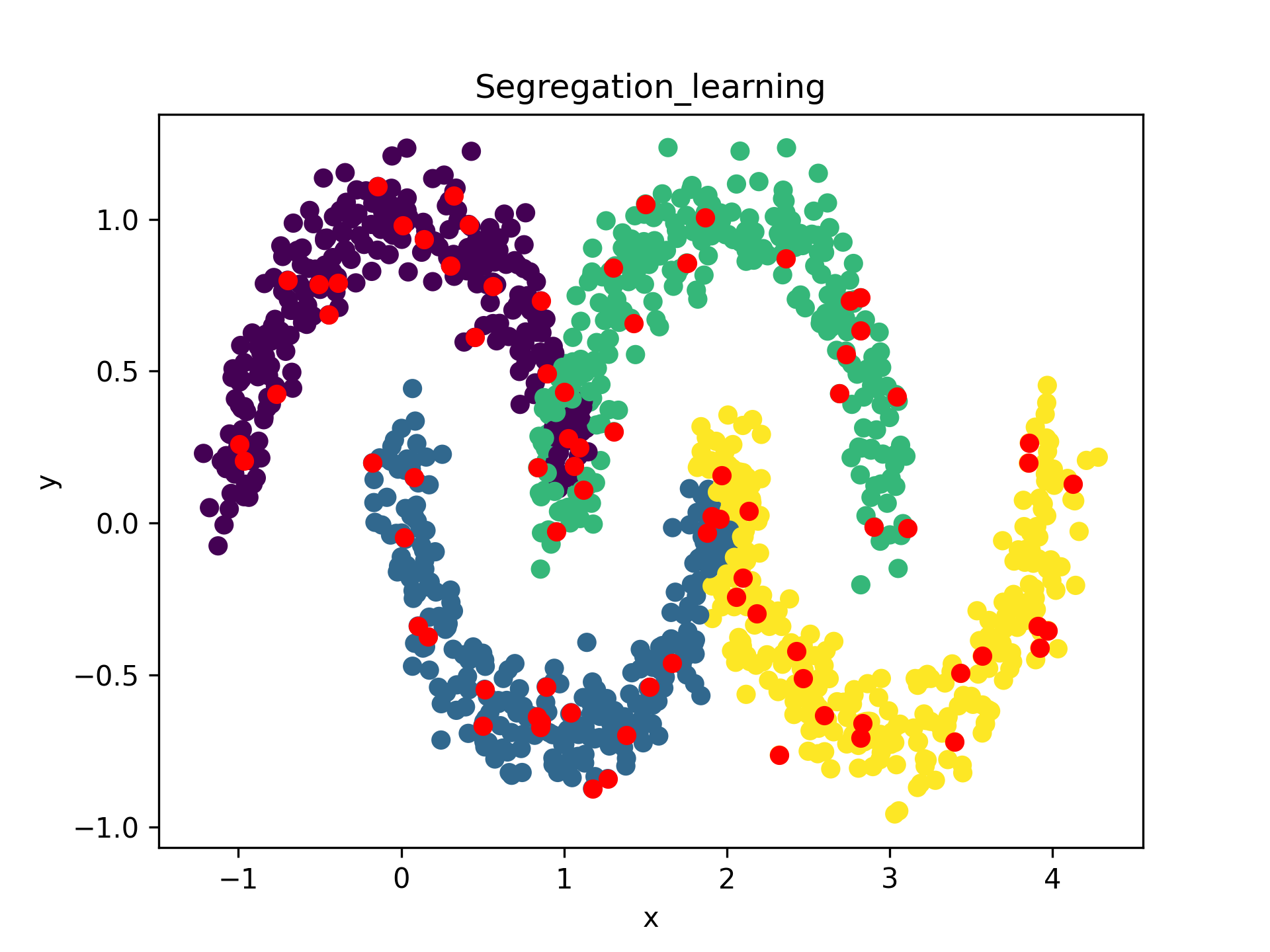}  \includegraphics[width=.32\linewidth,valign=m]{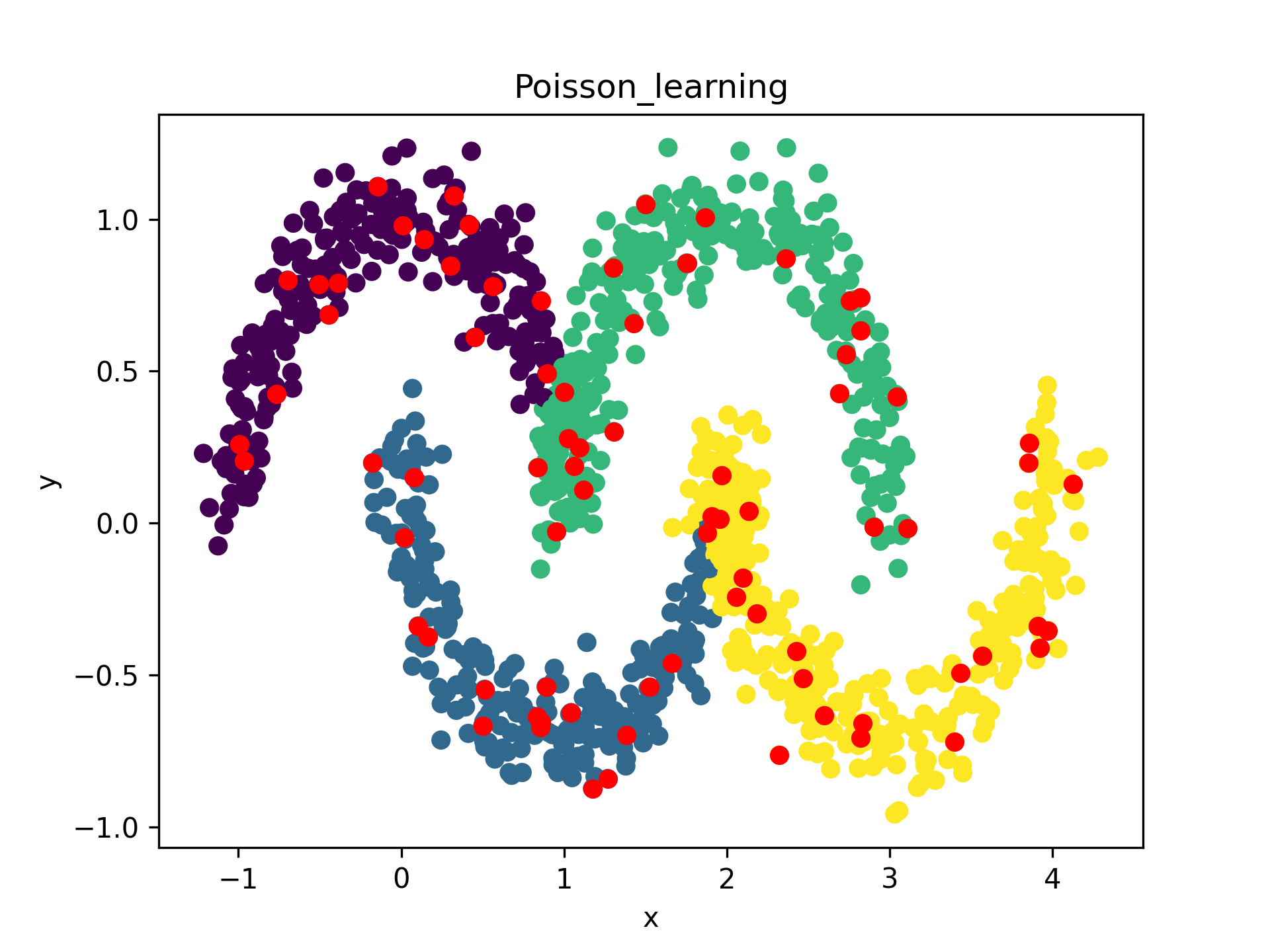}  \includegraphics[width=.32\linewidth,valign=m]{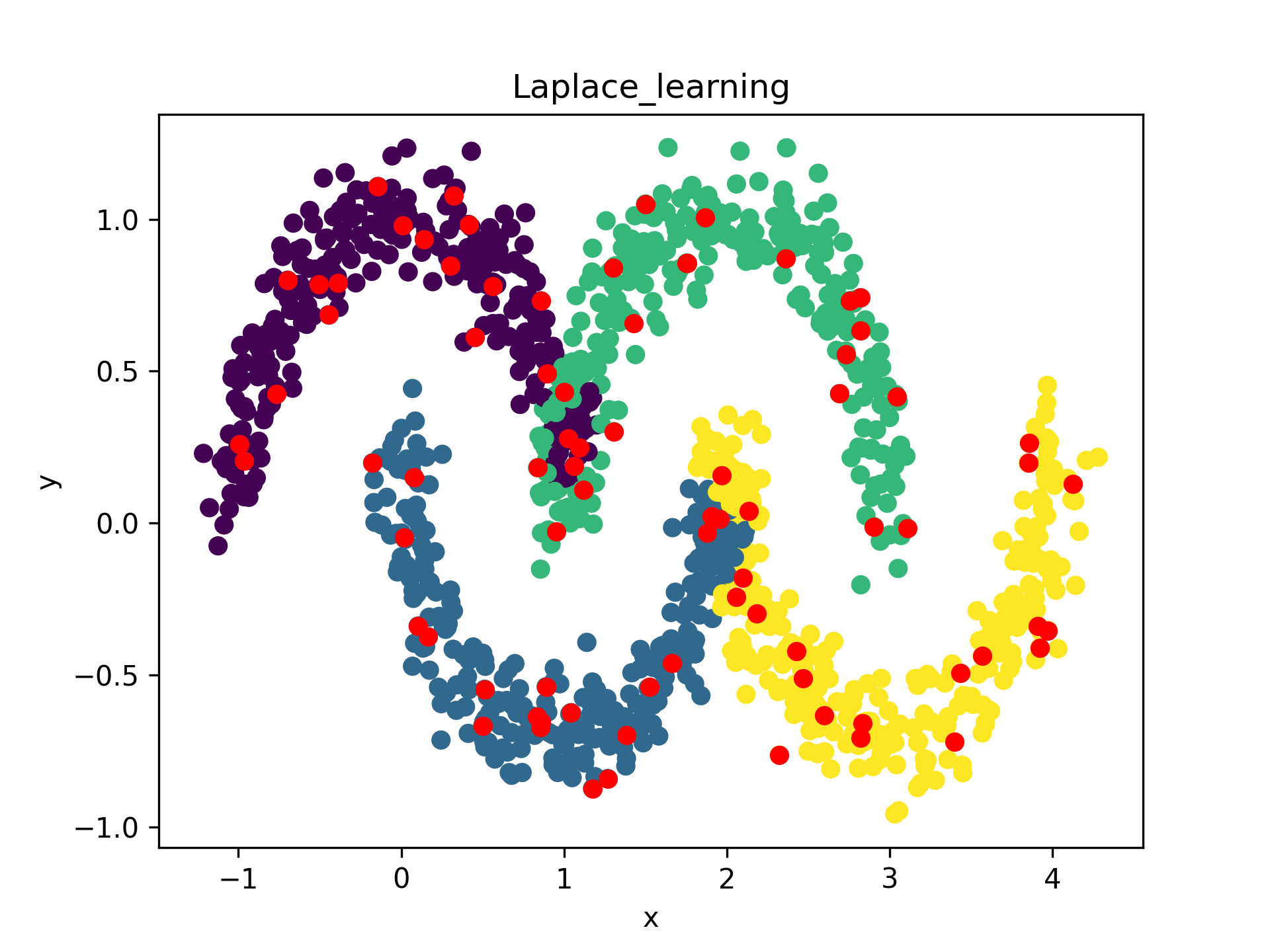}
	\caption{Comparison of Laplace, Poisson and Segregation learning algorithms for 4 classes and  initial 20 labels per class.}			
\end{figure}

\begin{figure}[!htb]\label{fig11}
	\includegraphics[width=.32\linewidth,valign=m]{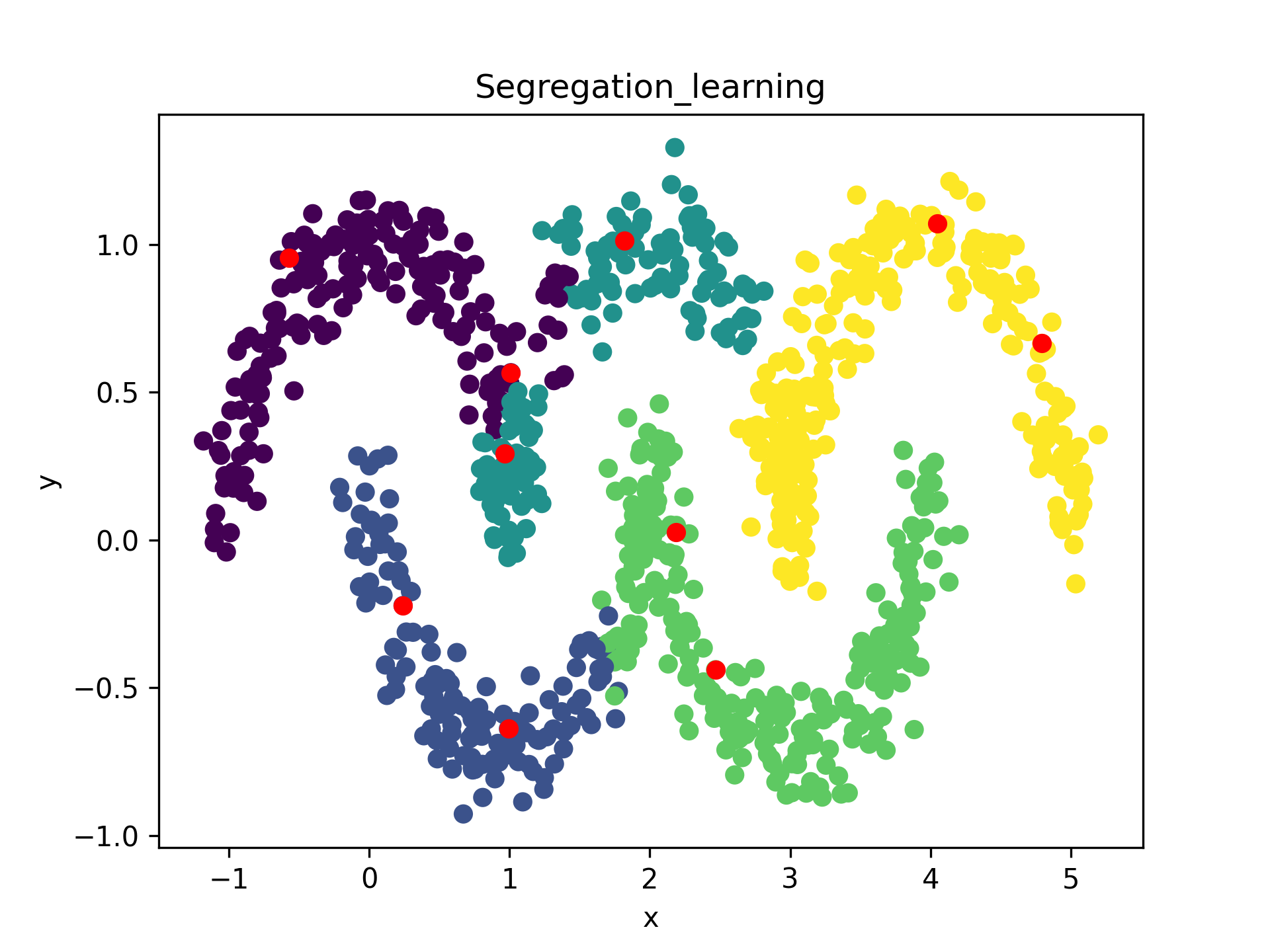}  \includegraphics[width=.32\linewidth,valign=m]{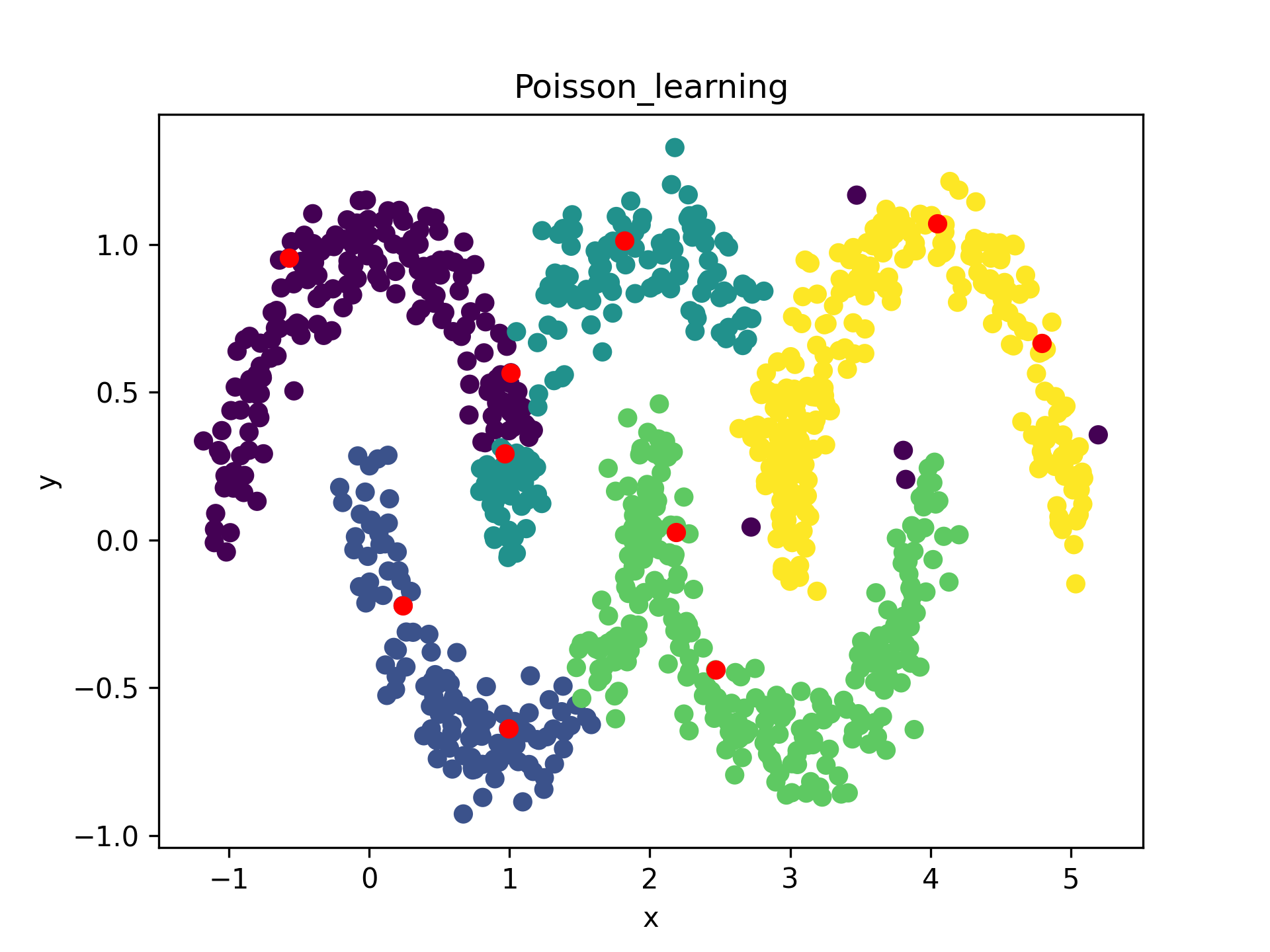}  \includegraphics[width=.32\linewidth,valign=m]{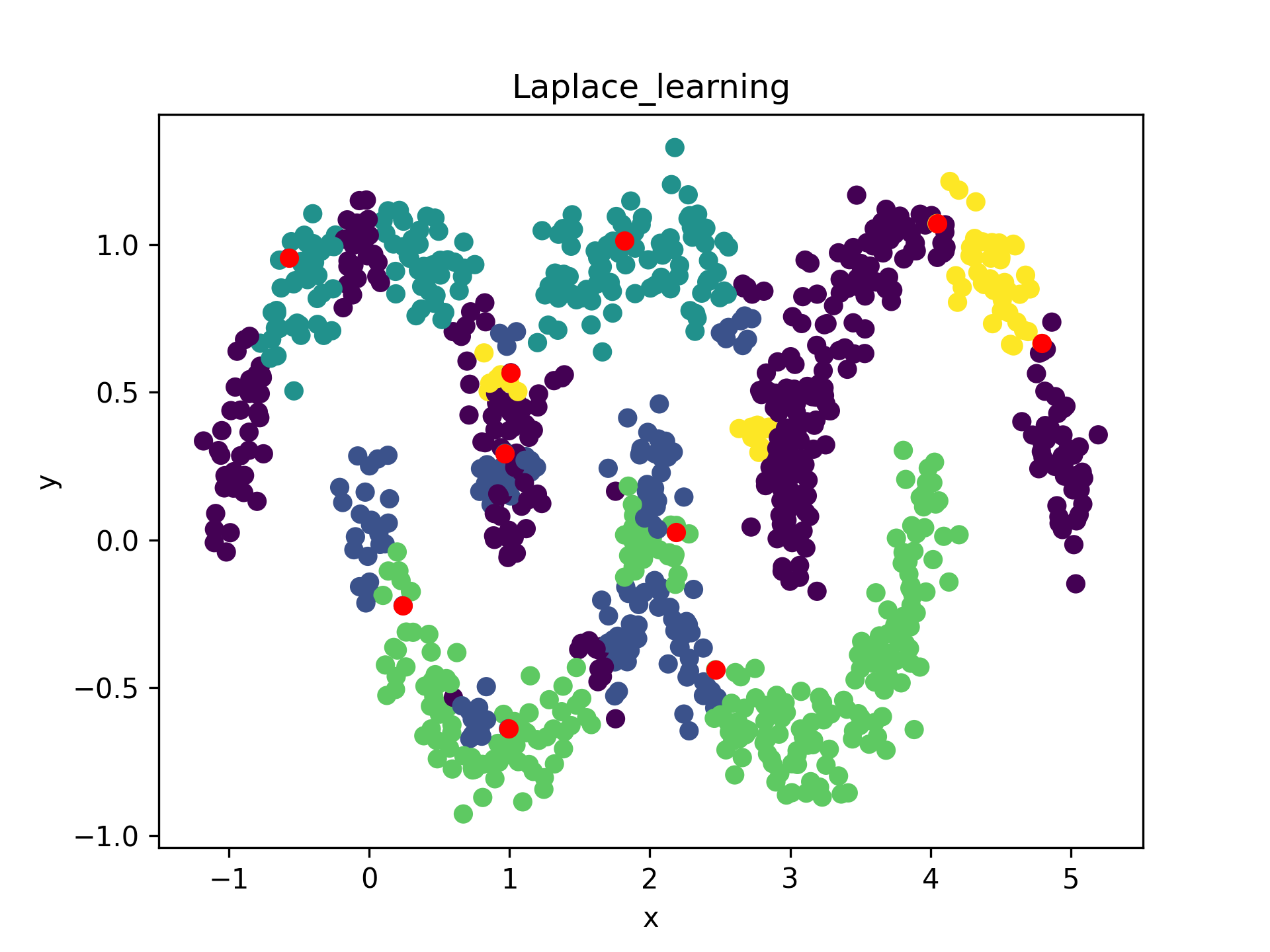}
	\caption{Comparison of Laplace, Poisson and Segregation learning algorithms for 5 classes and  initial 2 labels per class.}			
\end{figure}
\begin{figure}[!htb]\label{fig12}
	\includegraphics[width=.32\linewidth,valign=m]{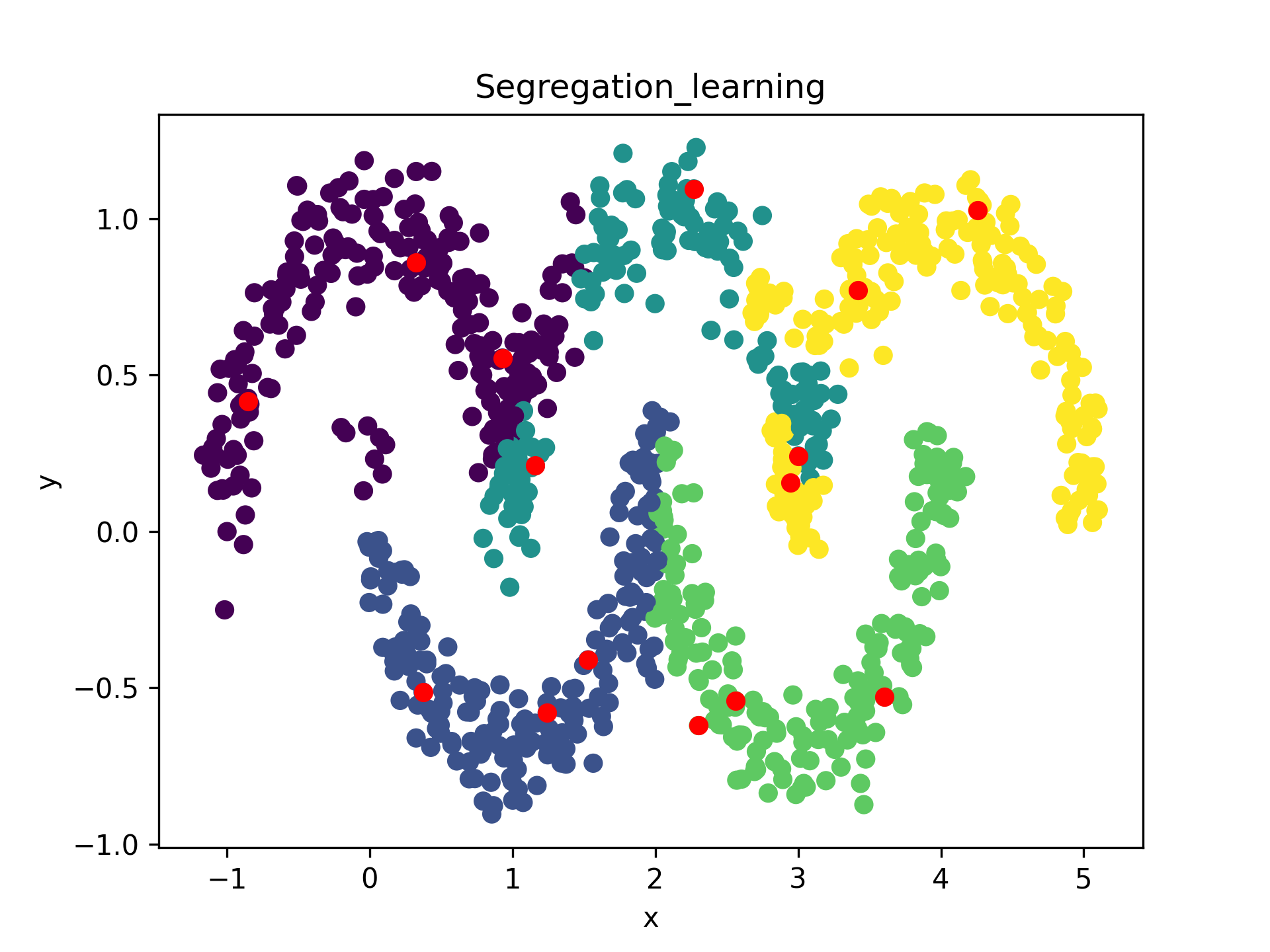}  \includegraphics[width=.32\linewidth,valign=m]{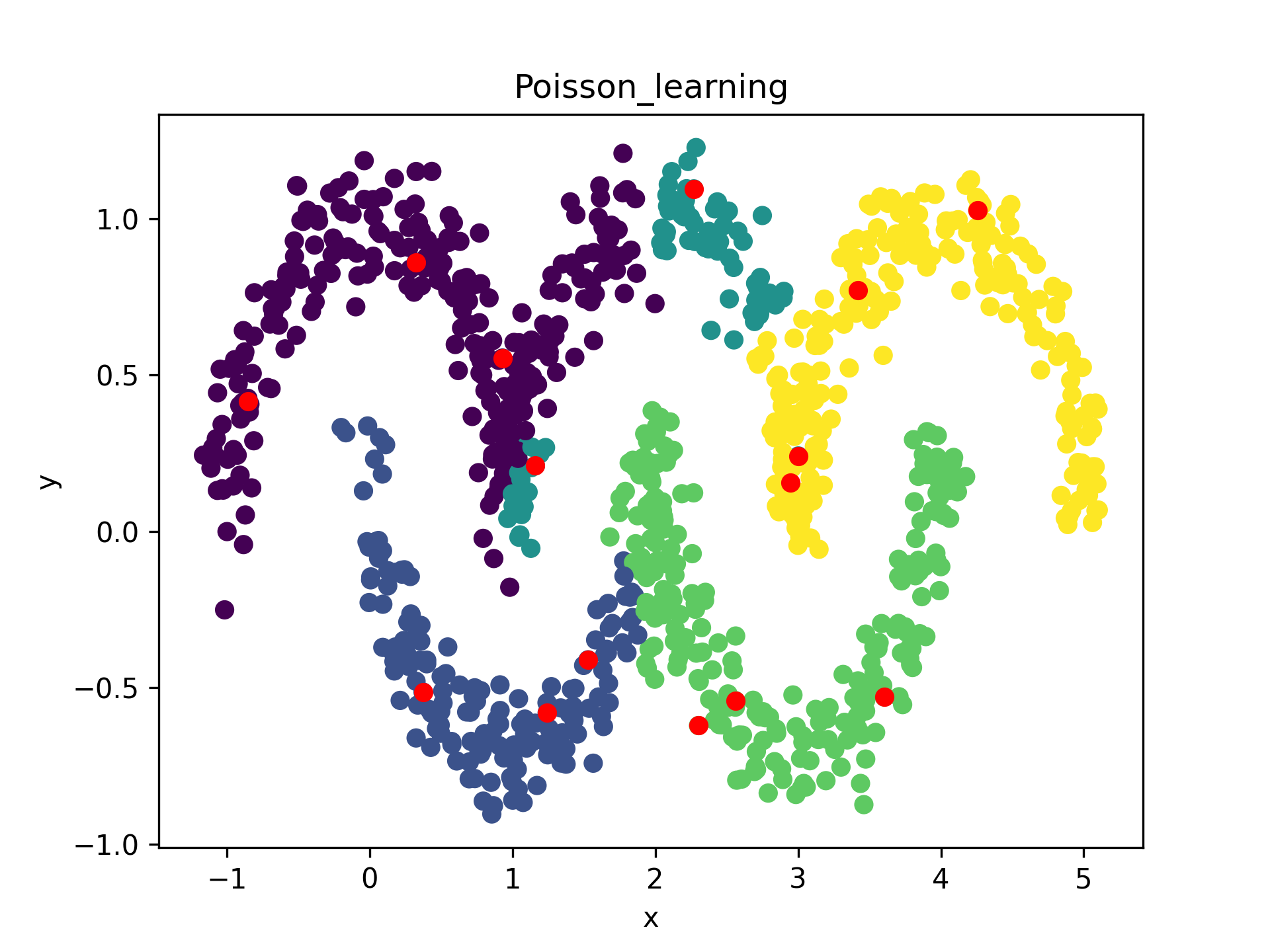}  \includegraphics[width=.32\linewidth,valign=m]{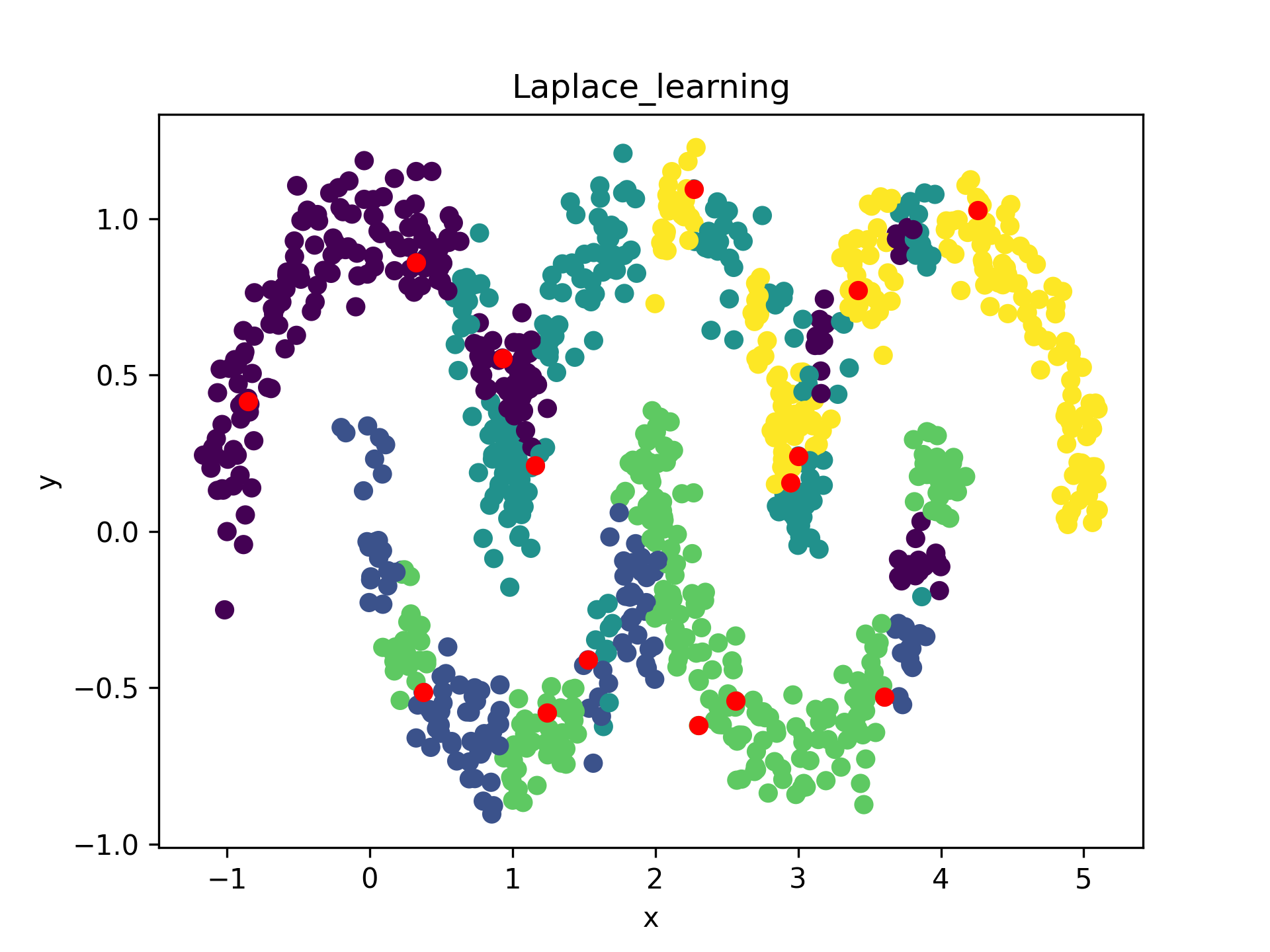}
	\caption{Comparison of Laplace, Poisson and Segregation learning algorithms for 5 classes and  initial 3 labels per class.}			
\end{figure}

\begin{figure}[!htb]\label{fig13}
	\includegraphics[width=.32\linewidth,valign=m]{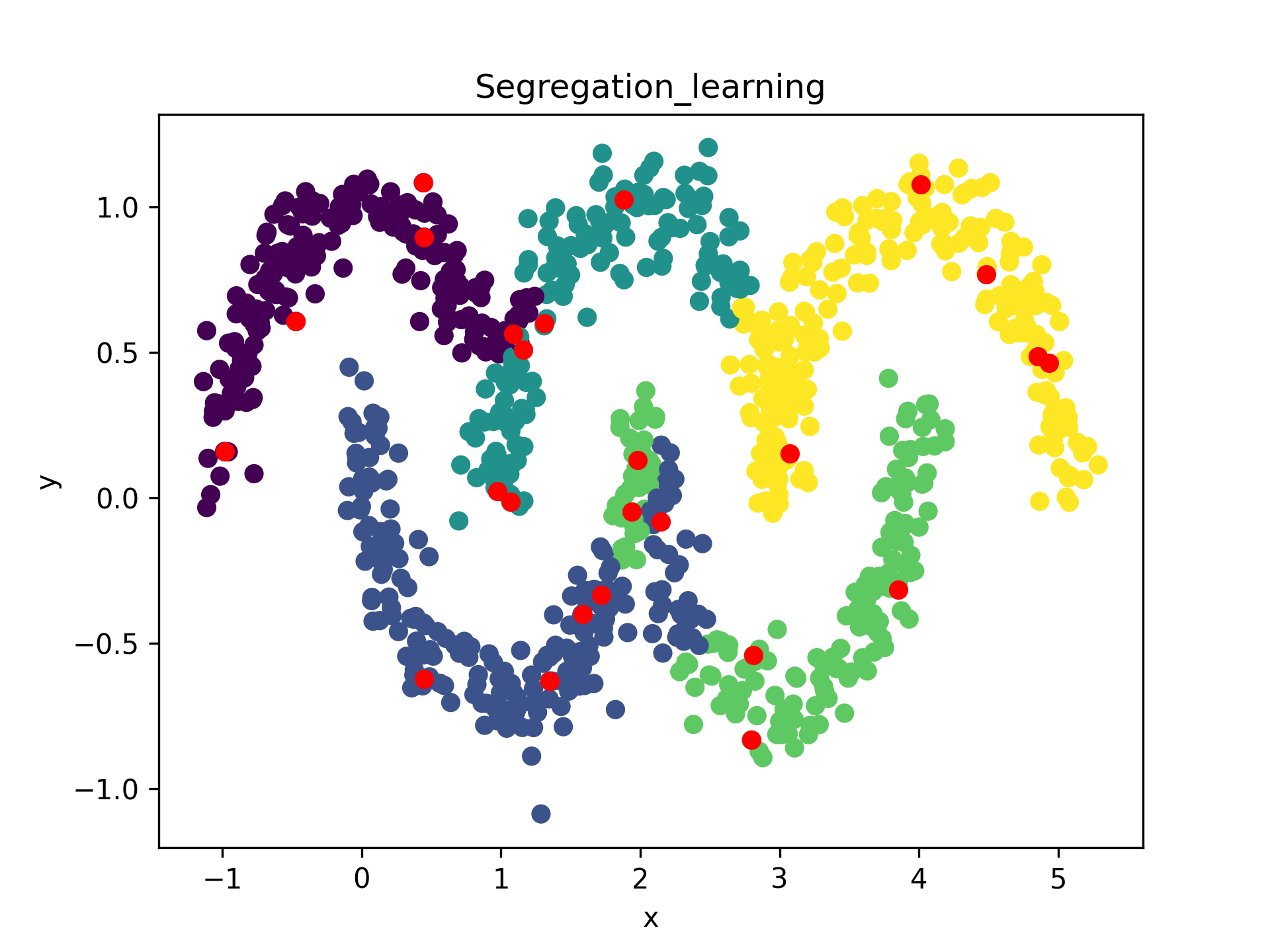}  \includegraphics[width=.32\linewidth,valign=m]{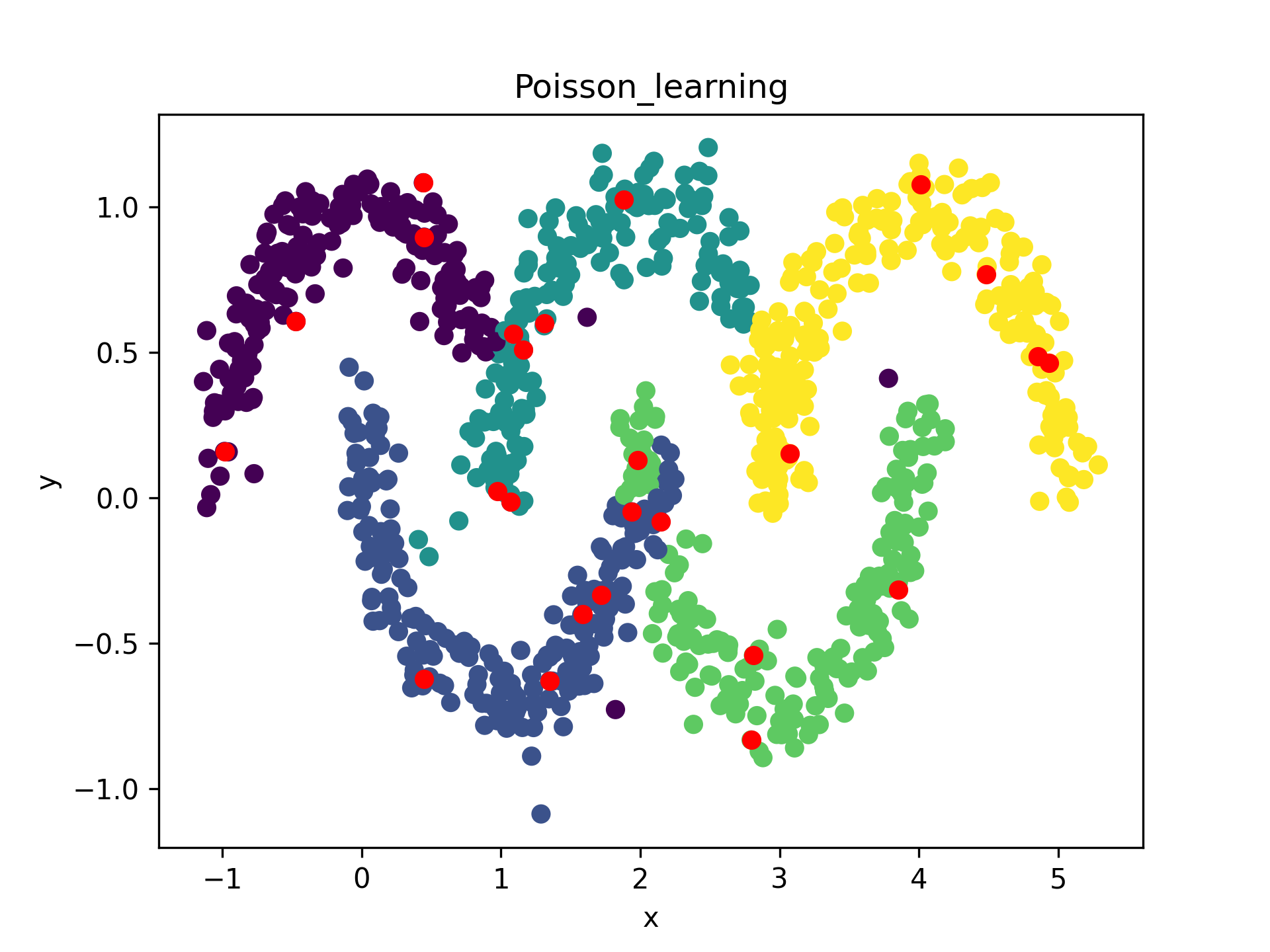}  \includegraphics[width=.32\linewidth,valign=m]{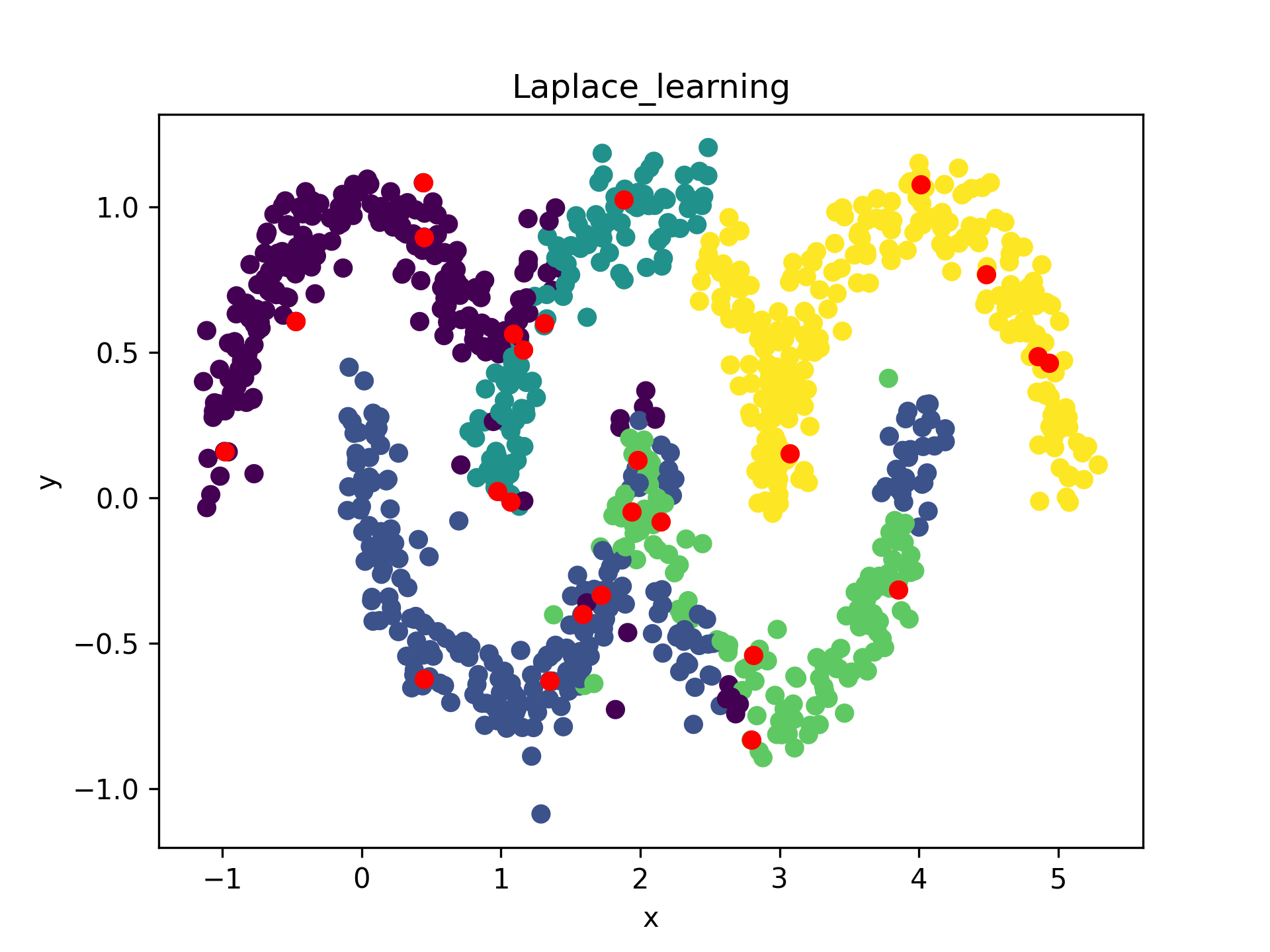}
	\caption{Comparison of Laplace, Poisson and Segregation learning algorithms for 5 classes and  initial 5 labels per class.}			
\end{figure}

\begin{figure}[!htb]
	\includegraphics[width=.32\linewidth,valign=m]{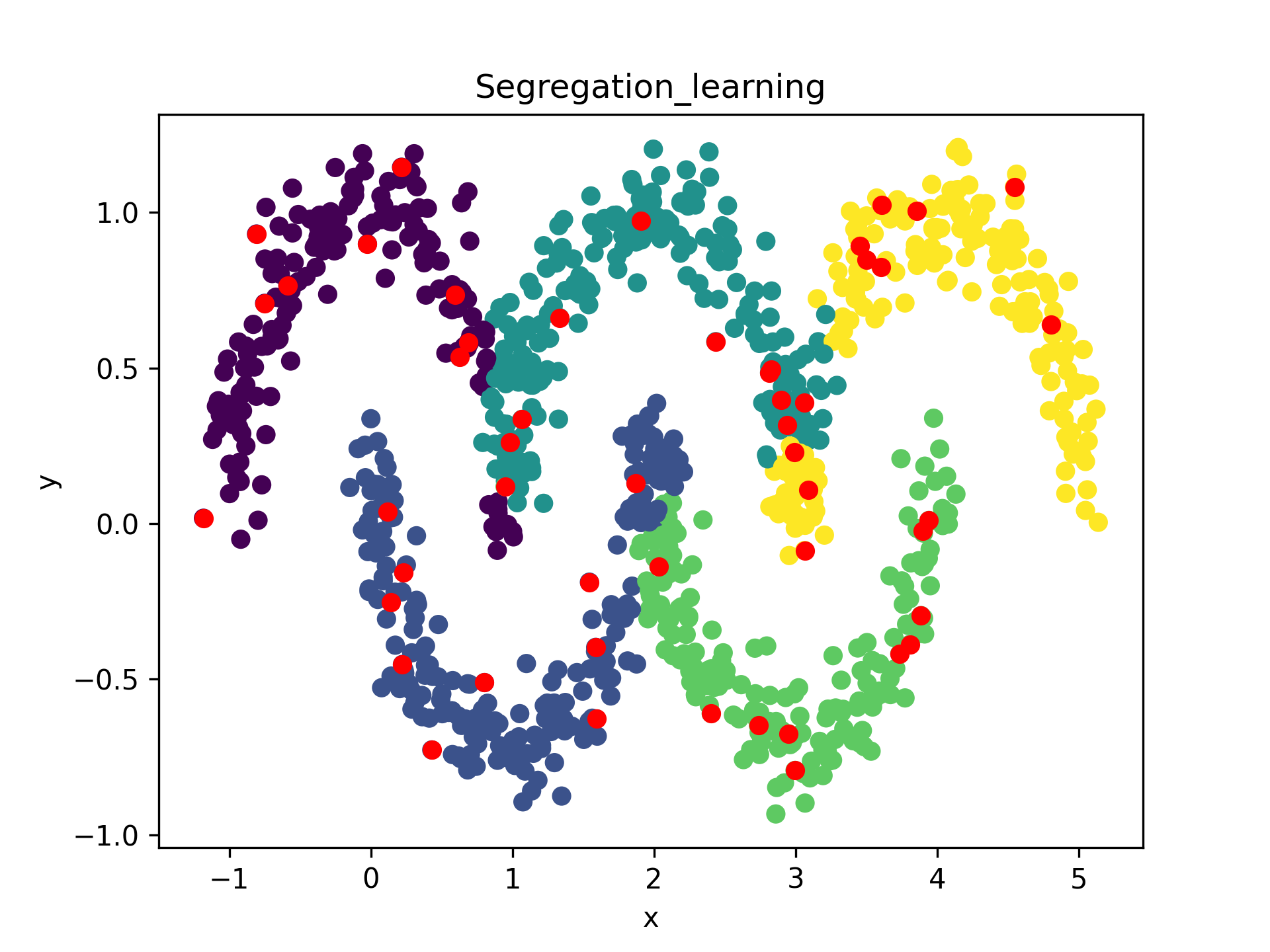}  \includegraphics[width=.32\linewidth,valign=m]{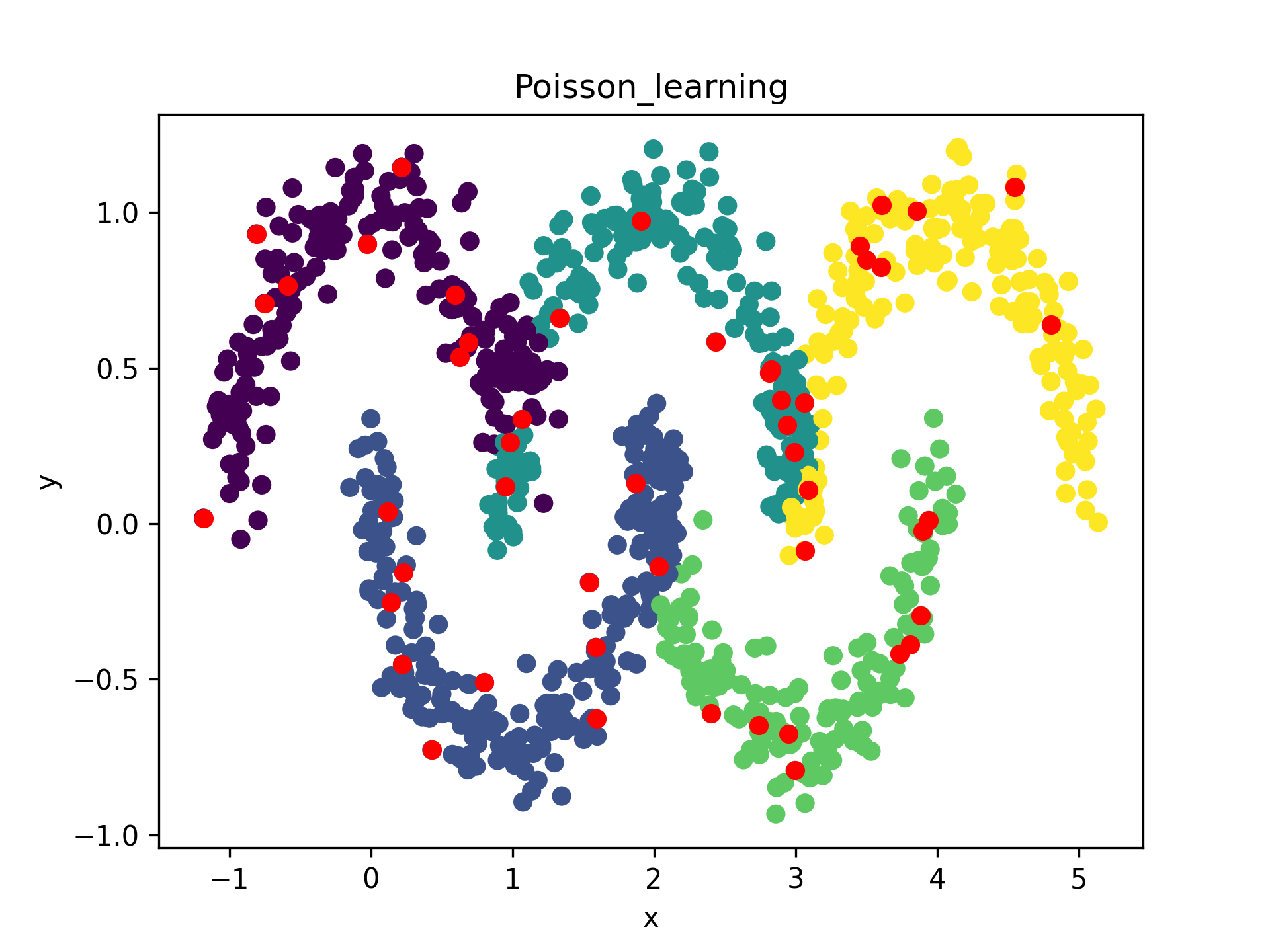}  \includegraphics[width=.32\linewidth,valign=m]{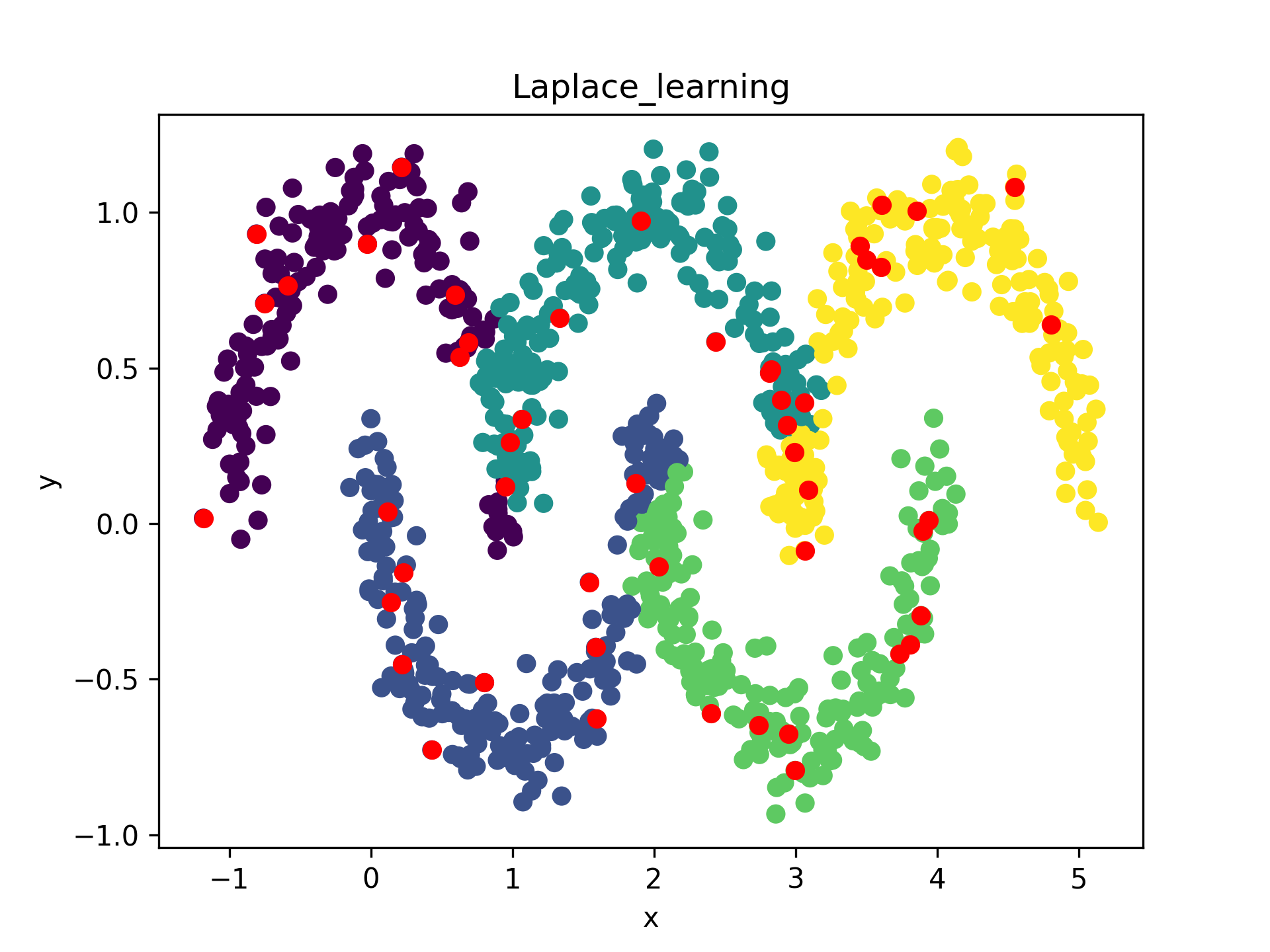}
	\caption{Comparison of Laplace, Poisson and Segregation learning algorithms for 5 classes and  initial 10 labels per class.}	
	\label{fig14}		
\end{figure}

\begin{figure}[!htb]
	\includegraphics[width=.32\linewidth,valign=m]{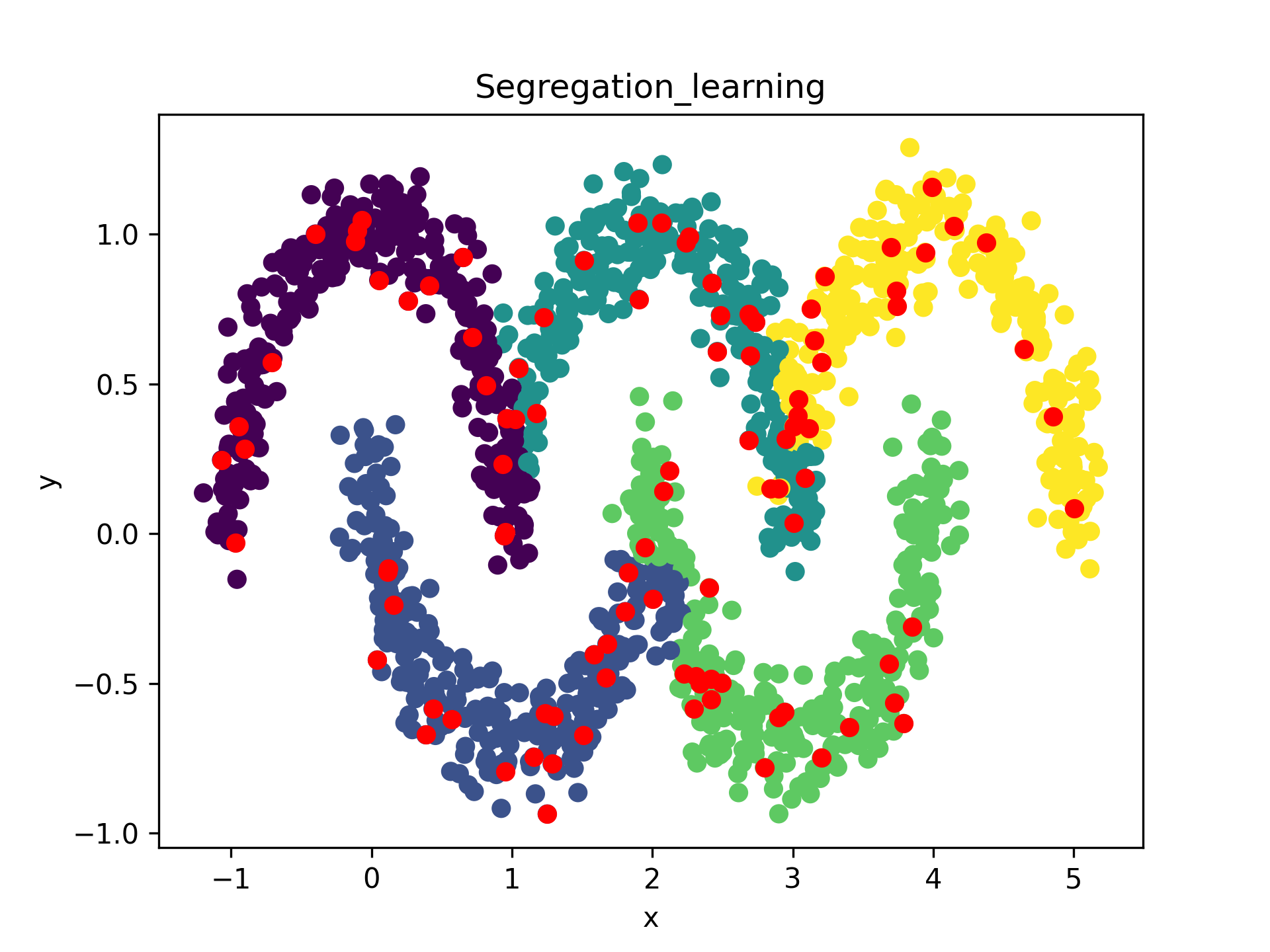}  \includegraphics[width=.32\linewidth,valign=m]{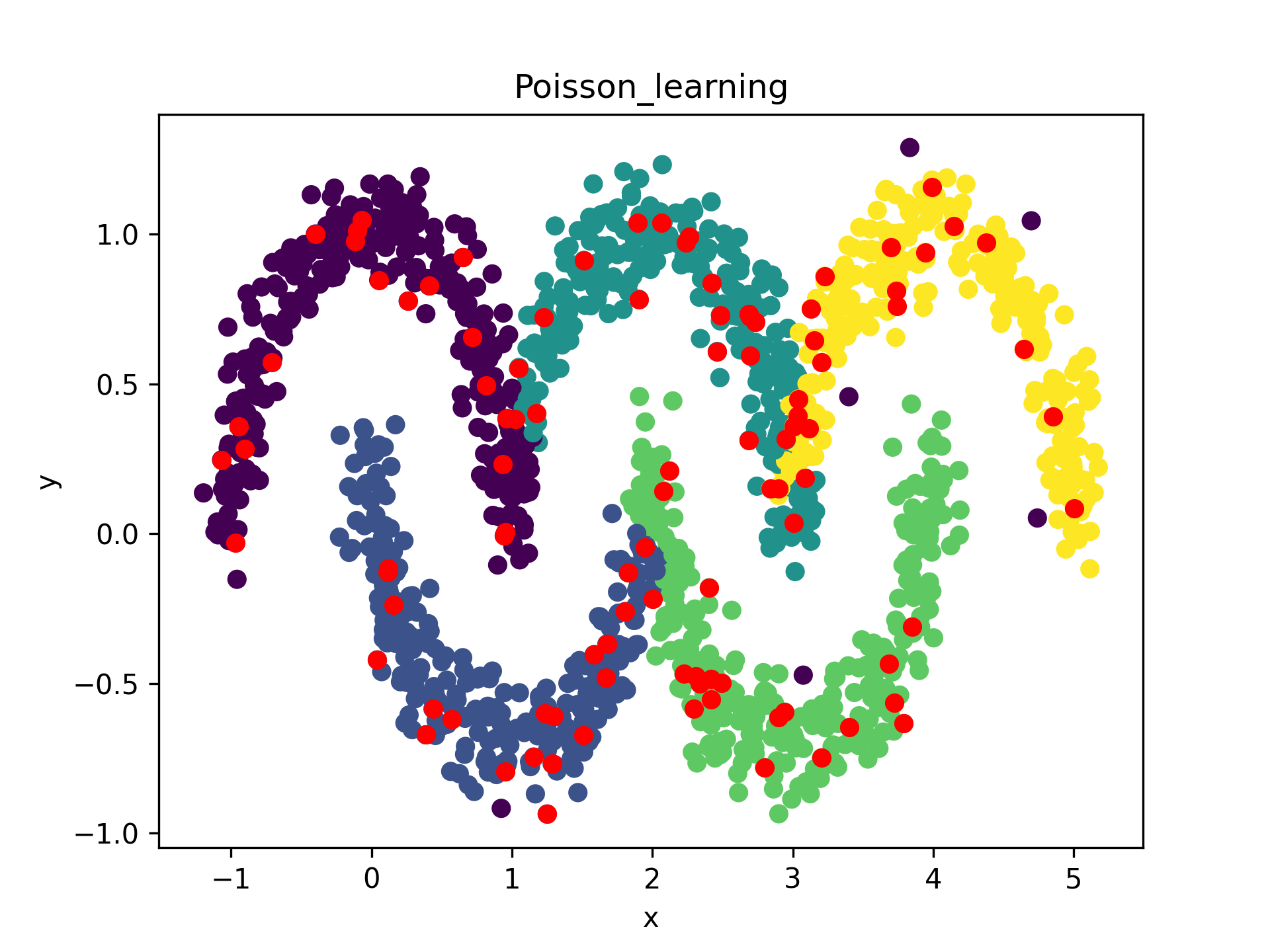}  \includegraphics[width=.32\linewidth,valign=m]{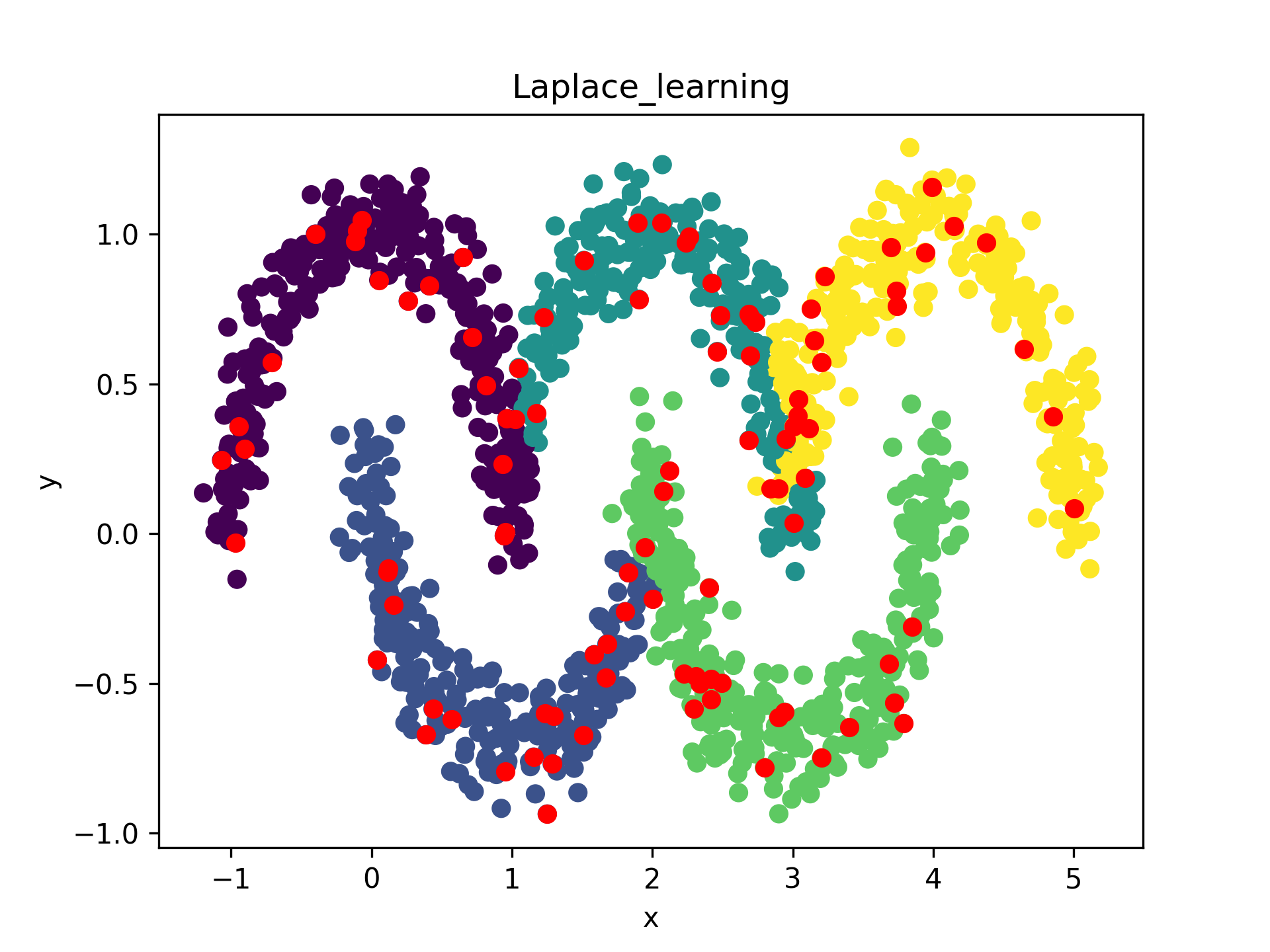}
	\caption{Comparison of Laplace, Poisson and Segregation learning algorithms for 5 classes and  initial 20 labels per class.}	
	\label{fig15}		
\end{figure}

\begin{table}[h]
	\caption{Average accuracy scores over 100 trials for 3 classes on MNIST dataset.} 
	\centering 
	\begin{tabular}{cllllll} 
		\hline\hline 
		Labels per class &\textbf{2} & \textbf{3} & \textbf{4} & \textbf{5}& \textbf{10}&  \\ [0.5ex]
		\hline 
		Laplace learning &31.3 & 45.4 & 58.2 & 67.7 & 83.4\\
		Poisson learning & 93.6  &  94.5  & 94.9 & 95.3 & 96.7&\\
	    Segregation learning  & \textbf{90.3}  & \textbf{92.1}  & \textbf{93.5} & \textbf{95.4} & \textbf{96.2} &\\[1ex] 
		\hline\hline 
	\end{tabular}
	\label{table_1}
\end{table}

\begin{table}[h]
	\caption{Average accuracy scores over 100 trials for 3 classes on MNIST dataset.} 
	\centering 
	\begin{tabular}{c llllll} 
		\hline\hline 
		Labels per class &\textbf{20} & \textbf{40} & \textbf{80} & \textbf{100}& \textbf{120}&\\ [0.5ex]
		\hline 
	Laplace learning &88.6  & 91.3 & 93.7 & 95.2  & 97.6\\
	Poisson learning & 95.6  &  97.2  & 97.9 & 98.3 & 99.4&\\
	Segregation learning  & \textbf{94.3}  & \textbf{96.7}  & \textbf{98.2} & \textbf{98.8} & \textbf{99.2} &\\[1ex] 
		\hline\hline 
	\end{tabular}
	 \label{table_2}
\end{table}

\section{Conclusion}
In this work we develop several semi-supervised algorithms on graphs. Our main algorithm is a new approach for graph-based semi-supervised learning based on spatial segregation theory. The method is efficient and simple to implement.
 We presented numerical results showing that Segregation Learning performs more or less as Poisson Learning algorithm not only at high label rates, but also at low label rates on MNIST dataset.

\renewcommand{\refname}{REFERENCES }

\section{Appendix}
In this section, we record two important statements in graphs, Poincar\'e inequality and maximum principle for superharmonic functions.

\begin{proposition}[Poincar\'e inequaity]\label{prop:Poincare inequaity}
	Assume the graph $\cX$ is connected. For every $ \Gamma  \subset  \cX $, there exists constant $\lambda_1>0$, the first eigenvalue of Laplacian, such that
	\[
	\lambda_1\|u\|_{\ell^2(\cX)}\leq \|\nabla u\|_{\ell^2(\cX^2)}
	\]
	for all $u\in\ell^2(\cX)$ satisfying $u=0$ on $\Gamma$.
\end{proposition}
\begin{proof}
	By the contradiction, we may  assume that there exist the sequence $u_n$ such that
	\[
	\|\nabla u_n\|_{\ell^2(\cX^2)}\leq \frac1n\|u_n\|_{\ell^2(\cX)},\quad u_n=0\text{ on }\Gamma.
	\]
	Let $\hat u_n:=u_n/\|u_n\|$, then $\|\nabla \hat u_n\|\rightarrow0$. The sequence $\{\hat u_n(x)\}_n$  is uniformly bounded ($|\hat u_n(x)|\leq \|\hat u_n\|=1$) and so there is a subsequence   $\hat u_{n_i}$ and the limit function $u$ such that $\hat u_{n_i}(x)\rightarrow u(x)$ for every $x\in \cX$.
	Hence, $\nabla \hat u_{n_i}\rightarrow \nabla u$, and so $\nabla u=0$ which yields that $u$ is a constant function on $\cX$. On the other hand, from the boundary data $ \hat u_{n_i}=0$ on $\Gamma$ we obtain that $u=0$ on $\Gamma$ and so $u=0$ on  all of the graph. This contradicts the condition $\|u\|=\lim_{n_i\rightarrow\infty}\|\hat u_{n_i}\|=1$.
\end{proof}

\begin{proposition}[Maximum principle]\label{Maximum principle}
	Assume the graph $\cX$ is connected.
	Let $p(x)$ be a nonnegative function on $\cX$, and $u$ satisfies $\cL u+p(x)u\geq0$ in $\cX\setminus\Gamma$.
	If $u\geq0$ on $\Gamma$, then $u\geq0$ in $\cX$.
\end{proposition}
\begin{proof}
	Define $A^+:=\{x\in \cX: u(x)\geq0\}$ and $A^-:=\{x\in \cX:u(x)<0\}$.
	Let $v:=\max(-u,0)$ and multiply by the equation to get
	\begin{align*}
	0\leq(\cL u+pu,v)=&(\nabla u,\nabla v)_{\ell^2(\cX^2)}+(pu,v)\\
	=&\frac12\sum_{x,y\in \cX}w_{xy}(u(x)-u(y))(v(x)-v(y))+\sum_{x\in \cX}p(x)u(x)v(x)\\
	=&-\frac12\sum_{x,y\in A^-}w_{xy}(u(x)-u(y))^2-\sum_{x\in A^-, y\in A^+}w_{xy}(u(x)-u(y))u(x)\\
	&-\sum_{x\in A^-}p(x)(u(x))^2\leq 0
	\end{align*}
	If $A^-\neq \emptyset$, we can find an edge between $A^+$ and $A^-$ due to the connectedness of $G$. (Note that the boundary condition ensures that $A^+\neq\emptyset$.)
	But the second term in the above calculation yields that $w_{xy}=0$ for every $x\in A^-$ and $y\in A^+$.
	Therefore, we must have $A^-=\emptyset$.
\end{proof}

\end{document}